\documentclass{birkart}
\usepackage{graphicx,amsmath,amssymb,amsthm,amscd}
\newtheorem{thm}{Theorem}[section]

\newtheorem{lem}[thm]{Lemma}
\newtheorem{prop}[thm]{Proposition}
\theoremstyle{definition}
\newtheorem{defn}[thm]{Definition}
\theoremstyle{remark}
\newtheorem{rem}[thm]{Remark}

\numberwithin{equation}{section}
\newcommand{\R}{\mathbb{R}}

\newcommand{\N}{\mathbb{N}}

\newcommand{\Z}{\mathbb{Z}}
\newcommand{\C}{\mathbb{C}}

\newcommand{\E}{\mathcal{E}}

\newcommand{\dbar}{{d\hspace{-0,05cm}\bar{}\hspace{0,05cm}}}
\newcommand{\Op}{\textup{Op}}

\newcommand{\op}{\textup{op}}

\begin{document}
%
\title[Branching Asymptotics on Manifolds with Edge]
{Branching Asymptotics on Manifolds with Edge}
\author[B.-Wolfgang Schulze]{B.-Wolfgang Schulze}

\address{%
Am Neuen Palais 10\\
14469 Potsdam\\
Germany}

\email{schulze@math.uni-potsdam.de}

\author{Andrea Volpato}
\address{%
Am Neuen Palais 10\\
14469 Potsdam\\
Germany}
\email{volpato.andrea@fastwebnet.it}


\date{}
  
\today


\begin{abstract}
We study pseudo-differential operators on a wedge with continuous and
variable discrete branching asymptotics.

\end{abstract}

\maketitle
\tableofcontents
\section*{Introduction}
The solutions to elliptic (pseudo-) differential equations on a manifold with edge are expected to be regular 
in weighted edge Sobolev spaces with some typical asymptotic behaviour in the distance variable $r\in \R_+$ 
to the singularity (we employ here the terminology from \cite{Schu32}, see also \cite{Schu20}; later on we recall a few basic 
definitions). The simplest case of a manifold with edge is a manifold with smooth boundary. Regularity of a
(Shapiro-Lopatinskij-) 
elliptic boundary value problem in this situation (say, for the Laplacian, or any other elliptic operator $A)$ implies smoothness of solutions up to the boundary when the
boundary data and right hand sides are smooth. In particular, we obtain Taylor asymptotics 
$u(r)\sim \sum_{j\in \N}c_j(y)r^j$ as $r\rightarrow 0$ where $r$ in this case is the normal variable 
to the boundary, and the coefficients $c_j(y)$ are smooth functions in the tangential variable $y.$ Also  regularity in Sobolev spaces up to the boundary can be formulated as an asymptotic result when we rephrase classical Sobolev spaces as edge spaces, cf. \cite{Schu20}, or Remark \ref{31locw} below. However, the situation drastically changes when $A$ is an elliptic operator from the context of edge singularities. Such an $A$ is locally in the variables $(r,x,y)$ of an open stretched wedge
$\R_+\times X\times \Omega $ of edge-degenerate form
\begin{equation} \label{ideg}
A=r^{-\mu}\sum_{j+|\alpha |\leq \mu}a_{j\alpha }(r,y)( -r\partial_r)^j(rD_y)^\alpha, 
\end{equation}
$a_{j\alpha }\in C^\infty(\overline{\R}_+\times \Omega , \mbox{Diff}^{\mu -(j+|\alpha |)}(X))$. Here  $X$ is a smooth (compact, closed) manifold and $\Omega \subseteq \R^q$ an open set, and $\mbox{Diff}^\nu (X) $ is the space of differential operators of order $\nu$ over $X$ with smooth coefficients. According to \cite{Schu28}, solutions $u(r,x,y)$ to $Au=f$ (say, in 
$C^\infty(\R_+\times X\times \Omega) ),$ under an ellipticity condition on $A,$ are expected to have variable discrete asymptotics, i.e.
\begin{equation} \label{ias}
u(r,x,y)\sim \sum_j \sum^{m_j(y)}_{k=0}c_{jk }(x,y)r^{-p_j(y)} \textup{log} ^k\,r\quad\mbox{as}\quad r\rightarrow 0, 
\end{equation}
modulo some flat remainder (to be specified). The sequences $$\mathcal{P}(y):= \{p_j(y),m_j(y)\}_{j=1,\ldots,J(y)}\subset \C\times \N,\,J(y)\in\N\cup\{+\infty\},$$ are individual for every fixed $y\in \Omega .$ They are determined by the operator $A,$ and the right hand side of the equation $Au=f$ where $f$ is assumed to be of analogous structure. For the coefficients we expect $c_{jk }(x,y)\in C^\infty(X)$ for every fixed $y.$ What concerns Sobolev smoothness in suitable scales of weighted edge spaces we even need the right framework to express such asymptotics. The present paper belongs to a larger program to investigate asymptotics of solutions in such a sense. We employ results of \cite{Schu34} and \cite{Schu36} for the case $\textup{dim}\, X=0,$ combined with the recent progress of the pseudo-differential analysis on (non-trivial) manifolds with edges. Note that asymptotics of that kind have been studied also by Bennish, cf. \cite{Benn1}, \cite{Benn2}, using factorisations of the involved symbols. Wiener-Hopf techniques for (pseudo-differential) boundary value problems (in general, without the transmission property at the boundary) are developed in detail in the monograph of Eskin \cite{Eski2}. In concrete examples it may be very efficient to explicitly compute exponents as factorisation indices. Our approach is completely different; it is based on
the idea of representing asymptotics in terms of analytic functionals, cf. \cite{Remp6}.  The basics on the relationship between asymptotics and analytic functionals are already explained in \cite{Schu2}, including crucial observations on variable and branching asymptotics in this set-up. A self-contained exposition on what we call continuous asymptotics is also given in \cite{Kapa10}.\\
The idea to establish asymptotics of the form \eqref{ias} is to embed the problem into a pseudo-differential calculus and to interpret such asymptotics as a feature of the symbolic calculus. The difficulty is to incorporate the variety of different $\mathcal{P}(y),y\in \Omega ,$ into spaces of symbols and also to design ``the right" weighted Sobolev spaces with asymptotics, such that our operators are continuous between such spaces. We employ the pseudo-differential calculus on manifolds with edge in the sense of \cite{Schu32} or \cite{Schu20} which includes constant discrete or continuous asymptotics. Here, we refer to a substructure of the edge algebra with continuous asymptotics. Some basics have to be completely reorganised, for instance, the notion of Green operators. The program is altogether extremely delicate  since a pseudo-differential operator might destroy the very fragile pointwise discrete structure \eqref{ias}; in fact, integration with respect to the edge variable $y$ smears out the system of exponents over a large region in the complex plane (such phenomena are just the background of continuous asymptotics). However, we show that the operators in the edge calculus (although very general in principle) are so specific that they respect, indeed, the pointwise discrete character.\\
This paper is organised as follows. First in Section 2.1 we give an idea on how elliptic edge-degenerate differential operators produce $y$-dependent families of meromorphic operator functions the poles of which contribute to the asymptotic data of solutions. The essential point is that we admit from the very beginning poles of non-constant multiplicity under varying $y.$ In Section 2.2 we rephrase the information
with the help of families of analytic functionals in the complex plane of the Mellin covariable, and we define the crucial notion of a variable discrete Mellin asymptotic type and associated spaces of Mellin amplitude functions. In Section 2.3 we study families of associated pseudo-differential operators, depending on $y,\eta ,$ the local variables and covariables on the edge. As such they play the role of specific operator-vaued symbols of the edge calculus with variable discrete asymptotics. Other important ingredients are the Green symbols with variable discrete asymptotics, studied in Section 2.4. Here we also introduce variable discrete asymptotic types, belonging to families of functions over a cone. In particular, we show (cf. Proposition \ref{22melgreen1}) that various involved data (such as the chosen cut-off functions, and several weights) only change the Mellin edge symbols by Green symbols. Another important issue is the nature of weighted edge distributions with variable discrete asymptotics, investigated in Section 3.1, including their Fr\'echet topology. Finally in Section 3.2 we show that Mellin and Green operators of the edge calculus with variable discrete asymptotics act as continuous operators in weighted edge spaces with such asymptotics. This belongs to the idea to deduce variable discrete asymptotics of solutions to elliptic equations on a manifold with edge as an aspect of elliptic regularity in the edge calculus. It will be necessary to establish more elements of the edge calculus, e.g. the action of Mellin operators with non-smoothing holomorphic amplitude functions, and the parametrix construction in this framework. This will be the topic of \cite{Schu69}.

\section{Mellin and Green operators with asymptotics}

\subsection{Examples and motivation}
Let us first recall the idea on how solutions to an elliptic equation $Au=f$ for
\begin{equation} \label{1fu}
 A=r^{-\mu}\sum_{j=0}^\mu a_j(-r\partial_r)^j
\end{equation}
acquire asymptotics. We employ the Mellin transform $Mu(z)=\int_0^\infty  r^{z-1}u(r)dr$ which is known to induce a continuous operator 
$M:C^\infty_0(\R_+)\rightarrow \mathcal{A}(\C).$
Here $\mathcal{A}(U)$ for an open set $U\subseteq \C$ is the space of all holomorphic functions 
in $U$ (in the Fr\'echet topology of uniform convergence on compact subsets). More precisely, 
setting
$\Gamma _\beta :=\{z\in \C:\textup{Re}\,z =\beta \}$ for any real $\beta ,$ and $M_\gamma u:=Mu|_{\Gamma _{1/2-\gamma }},$ 
we have a continuous operator $M_\gamma :C^\infty_0(\R_+)\rightarrow \mathcal{S}(\Gamma _{1/2-\gamma })$ that extends to an 
isomorphism $M_\gamma :r^{-\gamma }L^2(\R_+)\rightarrow L^2(\Gamma _{1/2-\gamma })$ for every $\gamma \in \R.$ We 
call $M_\gamma $ the weighted Mellin transform. Recall that the inverse is given by the formula 
$M_\gamma ^{-1}g(r)=\int_{\Gamma _{1/2-\gamma }}r^{-z}g(z)\dbar z$ for $\dbar z:=(2\pi i)^{-1}dz. $ In this paper a cut-off function on the 
half-axis is any real-valued $\omega (r)\in C^\infty_0(\overline{\R}_+)$ that is equal to 1 in a neighbourhood of $r=0.$
The Mellin transform will be used also in the set-up of vector- and operator-valued functions. In particular, we 
employ Mellin pseudo-differential operators with amplitude functions $f(r,r',z)$ taking values in the space 
of classical pseudo-differential operators over $X$ of order $\mu.$ Let $L^\mu_{\textup{cl}}(X;\R^l)$ denote the space of classical parameter-dependent pseudo-differential operators on $X$ with parameter $\lambda\in\R^l,l\in\N.$ The local amplitude functions are classical symbols in $(\xi,\lambda)$, treated as a covariable, while the parameter-dependent smoothing operators are Schwartz functions in $\lambda\in\R^l$ with values in $L^{-\infty}(X)$, the space of smoothing operators on X, with the Fr\'echet topology from an identification with $C^\infty(X\times X)$ via some Riemannian metric. In the case $l=0$ we simply write $L^\mu_{\textup{cl}}(X)$.\\
We assume 
$f(r,r',z)\in C^\infty (\overline{\R}_+\times \overline{\R}_+,L^\mu_{\textup{cl}}(X;\Gamma _{1/2-\gamma }))$ (with $\Gamma _{1/2-\gamma }$ being identified with $\R$ via $z\mapsto\textup{Im}\,z$), and we set
\begin{equation} \label{1mop}
\op_M^\gamma (f)u(r):=\int_{-\infty }^\infty \int_0^\infty (r/r')^{-(1/2-\gamma +i\rho )}
f(r,r',1/2-\gamma +i\rho )u(r')r'^{-1}dr'\dbar\rho , 
\end{equation}
$\dbar\rho:=(2\pi )^{-1}d\rho,$ first for $u(r)\in  C^\infty_0(\R_+,C^\infty(X)),$ and later on for weighted distributions. Observe that when $f$ is independent of $r,r'$ and $(\delta _\lambda u)(r):=u(\lambda r)$ for $\lambda \in \R_+$ we have
\begin{equation} \label{1mopcom}
\delta _\lambda\op_M^\gamma (f)=\lambda\op_M^\gamma (f) \delta _\lambda, \lambda \in \R_+.
\end{equation}

Throughout this paper $H^s(\R^n)$ indicates the standard Sobolev spaces in $\R^n$ of smoothness $s\in\R.$ Globally on a compact closed manifold $X$ we also have the spaces $H^s(X).$ Moreover, $H^s_{\textup{loc}}(X)$ over an open manifold $X$ means the space of all distributions $u$ such that $\chi_*(\varphi u)\in H^s(\R^n)$ for any chart $\chi:U\rightarrow\R^n,$ $\varphi\in C_0^\infty(U),$ while $H^s_{\textup{comp}}(X)$ is the subspace of compactly supported elements of $H^s_{\textup{loc}}(X).$\\
 
We use the fact that for every $\mu \in \R$ there exists an element $R^\mu \in L^\mu _{\textup{cl}} (X;\R^l)$ 
which induces isomorphisms $R^\mu (\lambda ):H^s(X)\rightarrow H^{s-\mu }(X)$ for all $s\in \R,\lambda \in \R^l.$
Set $X^\wedge:=\R_+\times X,$ and let 
$\mathcal{H}^{s,\gamma }(X^\wedge)$ for $s,\gamma \in \R$ denote the completion of $C^\infty_0(X^\wedge)$ with respect to the norm 
$$\|  u \| _{\mathcal{H}^{s,\gamma }(X^\wedge)}:=\Big\{ \int_{\Gamma _{(n+1)/2-\gamma }} \| R^s(\textup{Im}\, z )
M_{r\rightarrow z} u(z)\|^2_{L^2(X)}\dbar z\Big\}^{1/2} $$
for $n=\dim\,X.$ Moreover, set $X^\asymp :=\R\times X$ (interpreted as a manifold with conical exits to infinity $r\rightarrow \pm \infty $), and let $H^s_{\textup{cone}}(X^\asymp )$ denote the completion of
$C^\infty_0(X^\asymp)$ with respect to the norm
$$\|  u \| _{H^s_{\textup{cone}}(X^\asymp )}:=  \| [r]^{-s+n/2}\int e^{ir\rho }R^s([r]\rho ,[r]\eta ^1 )
(F_{r\rightarrow \rho } u)(\rho )\,\dbar \rho \|_{L^2(\R\times X)} $$
for any parameter-dependent elliptic element $R^s(\tilde{\rho },\tilde{\eta })\in L^s_{\textup{cl}}(X;\R^{1+q}_{\tilde{\rho },\tilde{\eta }})$
and sufficiently large $|\eta ^1|,\eta ^1\in \R^q.$ Here $r\mapsto [r]$ is any strictly positive function in $C^\infty (\R)$ such that $[r]=r$ for sufficiently large $|r|$. Set $H^s_{\textup{cone}}(X^\wedge):=H^s_{\textup{cone}}
(X^\asymp )|_{X^\wedge}, $ and define 
 \begin{equation} \label{1K}
\mathcal{K}^{s,\gamma }(X^\wedge):= \{\omega u+(1-\omega )v:u\in \mathcal{H}^{s,\gamma }(X^\wedge), v\in 
H^s_{\textup{cone}}(X^\wedge)\}
\end{equation}
for any cut-off function $\omega (r).$ Moreover, for $\Theta :=(\vartheta ,0],-\infty< \vartheta <0 ,$ we define the subspace
 \begin{equation} \label{1FK}
\mathcal{K}^{s,\gamma }_\Theta(X^\wedge) := \bigcap_{k\in \N}\mathcal{K}^{s,\gamma -\vartheta -(1+k)^{-1}}(X^\wedge),
\end{equation}
and $\mathcal{K}^{s,\gamma }_\Theta(X^\wedge):=\bigcap_{N\in \N}\mathcal{K}^{s,N }(X^\wedge)$ for $\Theta =(-\infty ,0].$
Now let $ -\infty \leq \vartheta <0,$ and consider a sequence
\begin{equation} \label{1as}
\mathcal{P}:=\{(p_j,m_j)\}_{j=0,\ldots ,J}\subset \C \times \N,\quad J=J(\mathcal{P})\in \N \cup \{\infty \},
\end{equation}
$\pi _\C \mathcal{P}:= \{p_j\}_{j=0,\ldots ,J}\subset \{z\in \C:(n+1)/2-\gamma +\vartheta <\textup{Re}\,z <(n+1)/2-\gamma \},$ and $\textup{Re}\,p_j \rightarrow -\infty $ as 
$j\rightarrow \infty $ (when $J(\mathcal{P})=\infty $). Such a $\mathcal{P}$ will be called a discrete
asymptotic type associated with the weight data $(\gamma ,\Theta); $ recall that $n=\textup{dim}\,X. $ For finite $\Theta $ we set
\begin{equation} \label{1sing}
\mathcal{E}_{\mathcal{P}}:=\{\omega \sum_{j=0}^J\sum_{k=0}^{m_j}c_{jk}\,r^{-p_j}\textup{log}^kr:c_{jk}\in C^\infty(X)\quad\mbox{for all}\quad j,k\}
\end{equation}
for a fixed cut-off function $\omega (r).$ Observe that we have $\mathcal{E}_{\mathcal{P}}\subset \mathcal{K}^{\infty ,\gamma }(X^\wedge)$ and $\mathcal{E}_{\mathcal{P}}\cap \mathcal{K}^{s,\gamma }_\Theta(X^\wedge)=\{0\}.$
Let 
\begin{equation} \label{1asK}
\mathcal{K}^{s,\gamma }_{\mathcal{P}}(X^\wedge):=\mathcal{K}^{s,\gamma }_\Theta(X^\wedge)+\mathcal{E}_{\mathcal{P}},
\end{equation}
for finite $\Theta.$ In the case $\Theta = (-\infty,0]$ we define
$\mathcal{K}^{s,\gamma }_{\mathcal{P}}(X^\wedge):=\bigcap _{k\in\N}\mathcal{K}^{s,\gamma }_{\mathcal{P}_k}
(X^\wedge)$ for $ {\mathcal{P}}_k:=\{p\in\ \pi _\C {\mathcal{P}}:(n+1)/2-\gamma -(k+1) <\textup{Re}\,
p <(n+1)/2-\gamma \},$ 
referring to the weight data $(\gamma ,\Theta _k)$ for $\Theta _k:= (-(k+1),0].$ Let us call the elements of $\mathcal{K}^{s,\gamma }_\Theta(X^\wedge)$ flat of order $\Theta$ relative to the weight $\gamma$ (recall that $\Theta$ is a half-open weight interval; therefore we do not talk about flatness of order $-\vartheta,$ except for the case $\vartheta=-\infty$).\\
Functions with asymptotics are transformed to meromorphic functions under the Mellin transform. For instance, the space $M_{\gamma -n/2,r\rightarrow z}\mathcal{E}_{\mathcal{P}}$ consists of meromorphic functions with poles at the points $p_j$ of multiplicity $m_j+1.$
Let us introduce some spaces of such functions in the complex Mellin covariable $z.$ Let $\mathcal{A}^
{s,\gamma }_\Theta (X)$ denote the space of all holomorphic functions $f$ in the open interior of the strip 
$S^\gamma _\Theta :=\{z\in \C:\beta _0<\textup{Re}\,z\leq  \beta _1 \}$ for $\beta _0
:=(n+1)/2-\gamma +\vartheta,\beta _1:=(n+1)/2-\gamma ,$ with values in $H^s(X),$ such that 
$f|_{\Gamma _\beta \times X}\in \hat{H}^s(\Gamma _\beta \times X)$ for every $\beta \in 
(\beta _0, \beta _1],$ uniformly in compact $\beta $-intervals. The condition at the right 
end point
means $f|_{\Gamma _{\beta_1 -\varepsilon} \times X}\rightarrow f|_{\Gamma _{\beta _1}  \times X}$  
as $\varepsilon \searrow 0$ in $\hat{H}^s(\R\times X).$ Here $\hat{H}^s(\R\times X)$ 
is the Fourier transform  of the cylindrical Sobolev space $H^s(\R\times X)$ with respect to the 
variable in $\R,$ and $\hat{H}^s(\Gamma_\beta\times X)$ is the corresponding space when we 
identify $\Gamma _\beta $ with $\R.$ Observe that $$M_{\gamma -n/2,r\rightarrow z}\omega \mathcal{K}^{s,\gamma }_\Theta(X^\wedge)\subset \mathcal{A}^{s,\gamma }_\Theta (X).$$ Moreover, set
$\mathcal{A}^{s,\gamma }_{\mathcal{P}} (X):=\mathcal{A}^{s,\gamma }_\Theta (X)+M_{\gamma -n/2,r\rightarrow z}
\mathcal{E}_{\mathcal{P}}$ for an asymptotic type $\mathcal{P}$ associated with the weight data $(\gamma ,\Theta ),$ 
first for finite $\Theta .$ In the infinite case we take the intersection of the respective spaces 
over all ${\mathcal{P}}_k,k\in \N,$ for ${\mathcal{P}}_k:=\{p\in \pi _\C {\mathcal{P}}: \textup{Re} \,p>(n+1)/2-
\gamma -(k+1)\}.$ In the case $\textup{dim}\,X=0$ we also write $\mathcal{A}^{s,\gamma }_{\mathcal{P}} .$ Note that the cut-off function involved in $\mathcal{E}_{\mathcal{P}}$ does not affect the 
space $\mathcal{A}^{s,\gamma }_{\mathcal{P}} (X).$ We employ the fact that 
the weighted Mellin transform induces continuous mappings 
\begin{equation} \label{11M}
M_{\gamma -n/2}\omega :\mathcal{K}^{s,\gamma }_{\mathcal{P}}(X^\wedge)\rightarrow \mathcal{A}^{s,\gamma }_{\mathcal{P}} (X)
\quad\mbox{and}\quad \omega M_{\gamma -n/2}^{-1}: \mathcal{A}^{s,\gamma }_{\mathcal{P}} (X)\rightarrow \mathcal{K}^
{s,\gamma }_{\mathcal{P}}(X^\wedge),
\end{equation}
respectively, for any cut-off function $\omega .$ Now consider the equation $Au=f$ where $A$ is of the form
\eqref{1fu} with constant coefficients, and $f\in \mathcal{K}_{\mathcal{Q}}^{s-\mu,\gamma -\mu}(\R_+), s\in \R,$ for some discrete asymptotic type $\mathcal{Q},$ associated with the weight data $(\gamma -\mu,(-\infty ,0]).$ Assume for simplicity that $f$ vanishes for $r>R$ for some $R>0.$ Then, under the condition $a_\mu \neq 0$ and 
\begin{equation} \label{11coel}
\sigma _{\textup{c}}(A)(z):=\sum_{j=0}^\mu a_j z^j\neq 0\quad \mbox{for all}\quad z\in \Gamma _{1/2-\gamma }
\end{equation}
we find a solution $u(r)$ of the form 
\begin{equation} \label{11sol}
u(r)=\textup{op}_M^\gamma (\sigma _{\textup{c}}(A)^{-1})(r^\mu  f)\in \mathcal{H}^{s,\gamma }(\R_+),
\end{equation} 
and the relation \eqref{11M} shows what happens near $r=0.$ We have $M_\gamma r^\mu f \in \mathcal{A}_{T^{-\mu }\mathcal{Q}}^{s-\mu,\gamma }$ ($T^{-\mu }$ indicates a corresponding translation of the asymptotic type in the complex plane), and $\sigma _{\textup{c}}(A)^{-1}(z)M_\gamma (r^\mu f)(z)\in \mathcal{A}_{\mathcal{P}}^{s,\gamma }$ for some discrete asymptotic type ${\mathcal{P}}$ associated with the weight data $(\gamma, (-\infty,0])$. The asymptotic type $\mathcal{P}$ is determined by the multiplication of two meromorphic functions, the inverse of the conormal symbol, and the Mellin transform of the right hand side. Thus, the second mapping in \eqref{11M} gives us $\omega u =\omega M_\gamma^{-1}\{ \sigma _{\textup{c}}(A)^{-1}(z)M_\gamma (r^\mu f)(z)\} \in \mathcal{K}^{s,\gamma }_{\mathcal{P}}(\R_+).$(Similar conclusions apply in the case $n=\textup{dim}\,X>0.$ The operator $A$ is asked to be elliptic with respect to the homogeneous principal symbol $\sigma _\psi (A),$ i.e. non-vanishing as a function in $C^\infty$ of the cotangent bundle (minus zero section) of the smooth part of our manifold with conical singularity, and non-vanishing of the reduced symbol $\tilde{\sigma }_\psi (A)(r,x,\rho, \xi ):= r^\mu\sigma _\psi (A)(r,x,r^{-1}\rho ,\xi )$ in the splitting of variables $(r,x)\in X^\wedge=\R_+\times X$ up to $r=0$ (which refers to a neighbourhood of the conical singularity).  In addition \eqref{11coel} is to be replaced by the condition that $\sigma _{\textup{c}}(A)(z):=\sum_{j=0}^\mu a_j z^j : H^s(X)\rightarrow H^{s-\mu}(X)$ is bijective for all $z\in \Gamma _{(n+1)/2-\gamma }$ for some real $s;$ this is then independent of $s.$ If the coefficients $a_j$ are not constant in $r,$ then, in order to deduce asymptotics of solutions we can employ a parametrix of the operator $A$ in the cone algebra rather than the inverse, cf. \cite{Schu2}. However, the main idea to obtain asymptotics is the same as in the case with constant coefficients, namely, to employ the meromorphic inverse of $\sigma _{\textup{c}}(A)(z),$ taking values in Fredholm operators over $X.$ The non-bijectivity points $z\in \C$ of $\sigma _{\textup{c}}(A)(z)$ substitute the former zeros in the scalar case, and the finite  multiplicities of poles give rise to the logarithmic terms of the asymptotics.\\
In the case of an edge-degenerate operator of the form \eqref{ideg} the problem to derive and to express asymptotics of solutions to $Au=f$ (under a corresponding ellipticity condition on $A$) is much more complicated. First we have the homogeneous principal symbol $\sigma _\psi (A)(r,x,y,\rho ,\xi ,\eta )$ of order $\mu$ and the reduced symbol $\tilde{\sigma} _\psi(A)(r,x,y,\rho ,\xi ,\eta )=r^\mu\sigma _\psi (A)(r,x,y,r^{-1}\rho ,\xi ,r^{-1}\eta ).$ In addition there is the principal edge symbol
\begin{equation} \label{11edge}
\sigma _\land (A)(y,\eta ):=r^{-\mu}\sum_{j+|\alpha |\leq \mu}a_{j\alpha }(0,y) (-r\partial r)^j(r\eta 
)^\alpha  
\end{equation}
for $(y,\eta )\in T^*\Omega \setminus 0,$ interpreted as a family of continuous operators
\begin{equation} \label{11edk}
\sigma _\land (A)(y,\eta ):=\mathcal{K}^{s,\gamma }(X^\wedge)\rightarrow  \mathcal{K}^{s-\mu,
\gamma -\mu}(X^\wedge) 
\end{equation}
for every $s\in \R$ and a fixed choice of $\gamma \in \R.$ The operators \eqref{11edk} belong to the cone calculus over $X^\wedge$. As such they have corresponding ``subordinate" principal symbols, namely, $\sigma _\psi \sigma _\land (A), \tilde{\sigma} _\psi \sigma _\land (A),$ including the principal conormal symbol $\sigma _{\textup{c}}\sigma _\land (A),$ 
\begin{equation} \label{11bi}
\sigma _{\textup{c}}\sigma _\land(A)(y,z):=\sum_{j=0}^\mu a_j(0,y) z^j : H^s(X)\rightarrow H^{s-\mu}(X);
\end{equation}
the latter is a smooth operator function in $y\in \Omega .$ Similarly as before the asymptotics is determined by \eqref{11bi}. However, it is evident that the pointwise inverse $\sigma _{\textup{c}}\sigma _\land(A)^{-1}(y,z)$ may have poles depending on $y $ of variable multiplicities. Those are just the source of the $y$-dependent asymptotic data in \eqref{ias}. Moreover, we need a machinery to express such variable asymptotics in the framework of a suitable version of weighted edge Sobolev spaces. Both problems belong to the focus of the present paper. The final goal, realised in a forthcoming paper, will be to obtain variable asymptotics in the edge case by applying a parametrix $P$ from the left to the equation $Au=f,$ similarly as in the cone calculus, where $P$ is sensitive enough not to destroy the very individual asymptotic types. One of the crucial points will be that $1-PA=:G$ is a smoothing operator which produces functions with variable asymptotics, coming from the involved operator $A.$ Moreover, $Pf$ also has such asymptotics, provided that $f$ is of that kind. This entails altogether such asymptotics of the solution $u.$
\\For an open set $U\subseteq\C$ and a Fr\'echet space $E$ by $\mathcal{A}(U,E)$ we denote the space of holomorphic functions with values in $E.$  We employ this, in particular, for $E=L^\mu_{\textup{cl}}(X),$ the space of classical pseudo-differential operators of order $\mu$ on $X$ in its natural Fr\'echet topology. 
\begin{defn} \label{21mhol}
Let $M^\mu_{\mathcal{O}}(X)$ for $\mu\in\R$ denote the space of all $h(z)\in\mathcal{A}(\C,L^\mu _{\textup{cl}}(X))$ such that $h(\beta+i\rho)\in L^\mu_{\textup{cl}}(X;\R_\rho)$ for every $\beta\in\R$, uniformly in compact $\beta$-intervals.
\end{defn}
The space $M^\mu_{\mathcal{O}}(X)$ is Fr\'echet in a natural way; thus, we can form the space $C^\infty(\Omega,M^\mu_{\mathcal{O}}(X)).$ Observe that, in particular, \eqref{11bi} belongs to the latter space. If the parameter $\rho\in\R$ is identified with the imaginary part of the complex variable on the line $\Gamma_\beta$ we also write $L^\mu_{\textup{cl}}(X;\Gamma_\beta)$ instead of $L^\mu_{\textup{cl}}(X;\R_\rho).$
\begin{thm} \label{21kcut}
For every $l(y,z)\in C^\infty(\Omega,L^\mu_{\textup{cl}}(X;\Gamma_\beta))$ and for any fixed $\beta\in\R$ there exists a $k(y,z)\in C^\infty(\Omega,M^\mu_{\mathcal{O}}(X))$ such that $$l(y,z)-k(y,z)|_{\Omega\times\Gamma_\beta}\in C^\infty(\Omega,L^{-\infty}(X;\Gamma_\beta)).$$
\end{thm}
The proof relies on the method of kernel cut-off for Mellin symbols, see, for instance, \cite{Schu2}, or \cite[Theorem 2.2.28]{Schu20}. An alternative proof may be found in \cite[Theorem 7.1.10]{Krai2}.

\subsection{Families of analytic functionals induced by conormal symbols}
If the operator \eqref{ideg} is elliptic with respect to $\sigma_\psi$ (i.e. $\sigma_\psi(A)\neq 0,$ and $\tilde{\sigma}_\psi(A)\neq 0 $ up to $r=0$) the operator function \eqref{11bi} is a family of Fredholm operators of index zero. For any fixed $y\in\Omega$ we are in a situation which is well-known in pseudo-differential operators on a manifold with conical singularities. There is a discrete set $D(y)\subset \C$ which intersects $ \{c<\textup{Re}\,z<c'\}$ in a finite set for every $c\leq c',$ such that \eqref{11bi} is bijective for all $z\notin D(y)$  (for abbreviation we often write $ \{c<\textup{Re}\,z<c'\}$ instead of $\{z\in\C: c<\textup{Re}\,z<c'\},$ (and similarly for $\textup{Im}$ 
rather than $\textup{Re}).$ For every open $U\subset \Omega$ with compact $\overline{U}\subset\Omega$ and $c\leq c'$ there is an $M>0$ such that
\begin{equation} \label{21du}
D(y)\cap  \{c<\textup{Re}\,z<c'\}\subseteq \{|\textup{Im}\,z|\leq M\}
\end{equation}
for all $y\in\overline{U}.$
Let us set $D(y)=\{p_j(y)\}_{j\in J(y)}$ for a corresponding index set $J(y)\subseteq \Z$ (the numeration is individual for every $y$). The function $f(y,z) := \sigma _{\textup{c}}\sigma_{\land}(A)^{-1}(y,z)$ is extendible to a meromorphic family of Fredholm operators with poles at the points $p_j(y)$ of multiplicity $l_j(y)+1$ and Laurent expansions
\begin{equation} \label{21lau}
f(y,z)=f_j(y,z)+\sum_{k=0}^{l_j(y)}d_{jk}(y)/(z-p_j(y))^{k+1}
\end{equation}
where $f_j(y,z)$ is an $L^{-\mu}_{\textup{cl}}(X)$-valued holomorphic function in a neighbourhood of $p_j(y),$ and $d_{jk}(y)$ belongs to $L^{-\infty}(X)$ and is of finite rank. Let us set
\begin{equation} \label{21mas}
\mathcal{R}(y):=\{(p_j(y),l_j(y))\}_{j\in J(y)},
\end{equation}
and $\pi_\C\mathcal{R}(y):=\{p_j(y)\}_{j\in J(y)}$ (which is equal to $D(y)$). Moreover, let $\mathcal{U}(\Omega)$ denote the system of all open $U\subset\Omega$ such that $\overline{U}$ is compact and contained in $\Omega.$ Now fix a $y_0\in\Omega$ and choose numbers $c\leq c'.$ Then for any fixed $\varepsilon (y_0)>0$ we have a disjoint decomposition 
\begin{equation} \label{21dec0}
\pi_\C\mathcal{R}(y_0)=D_{c,c'}(y_0)\cup E_{c,c'}(y_0)
\end{equation}
for $D_{c,c'}(y_0):= \{p\in \pi_\C\mathcal{R}(y_0):c-\varepsilon(y_0)<\textup{Re}\,p<c'+\varepsilon(y_0))\},$ and $E_{c,c'}(y_0):=\pi_\C\mathcal{R}(y_0)\setminus D_{c,c'}(y_0)$ where $E_{c,c'}(y_0)\subset \{c-\tilde{\varepsilon}(y_0)<\textup{Re}\,z\}\cup\{\textup{Re}\,z >c'+\tilde{\varepsilon}(y_0)\}$ for some $\tilde{\varepsilon}(y_0)<\varepsilon(y_0).$
There is then a neighbourhood $U_0\in\mathcal{U}(\Omega)$ of $y_0$ such that 
\begin{equation} \label{21dec1}
\pi_\C\mathcal{R}(y)=D_{c,c'}(y)\cup E_{c,c'}(y)
\end{equation}
for $D_{c,c'}(y):= \{p\in \pi_\C\mathcal{R}(y):c-\varepsilon(y_0)<\textup{Re}\,p<c'+\varepsilon(y_0)\},$ and $E_{c,c'}(y):=\pi_\C\mathcal{R}(y)\setminus D_{c,c'}(y)$ with $E_{c,c'}(y)\subset \{c-\tilde{\varepsilon}(y_0)<\textup{Re}\,z\}\cup\{\textup{Re}\,z >c'+\tilde{\varepsilon}(y_0)\}$ for the same $\tilde{\varepsilon}(y_0), \varepsilon(y_0),$ for all $y\in \overline{U}_0.$\\Taking into account the relation \eqref{21du} we form a rectangle consisting of the intervals
$$I_\pm:= \{z\in\C:c-\varepsilon_0\leq\textup{Re}\,z\leq c'+\varepsilon_0, \textup{Im}\,z=\pm M\},$$
and
$$ I:=\{z\in\C:\textup{Re}\,z=c-\varepsilon_0, |\textup{Im}\,z|\leq M\},I':=\{z\in\C:\textup{Re}\,z=c'+\varepsilon_0, |\textup{Im}\,z|\leq M\}$$
for some $\varepsilon_0>0$ such that $\tilde{\varepsilon}(y_0)<\varepsilon_0<\varepsilon(y_0).$\\

\setlength{\unitlength}{0.008466666666666667cm}
\begin{center}
\begin{picture}(1417.0,836.0)
\put(0,0){\includegraphics[width=11.997cm]{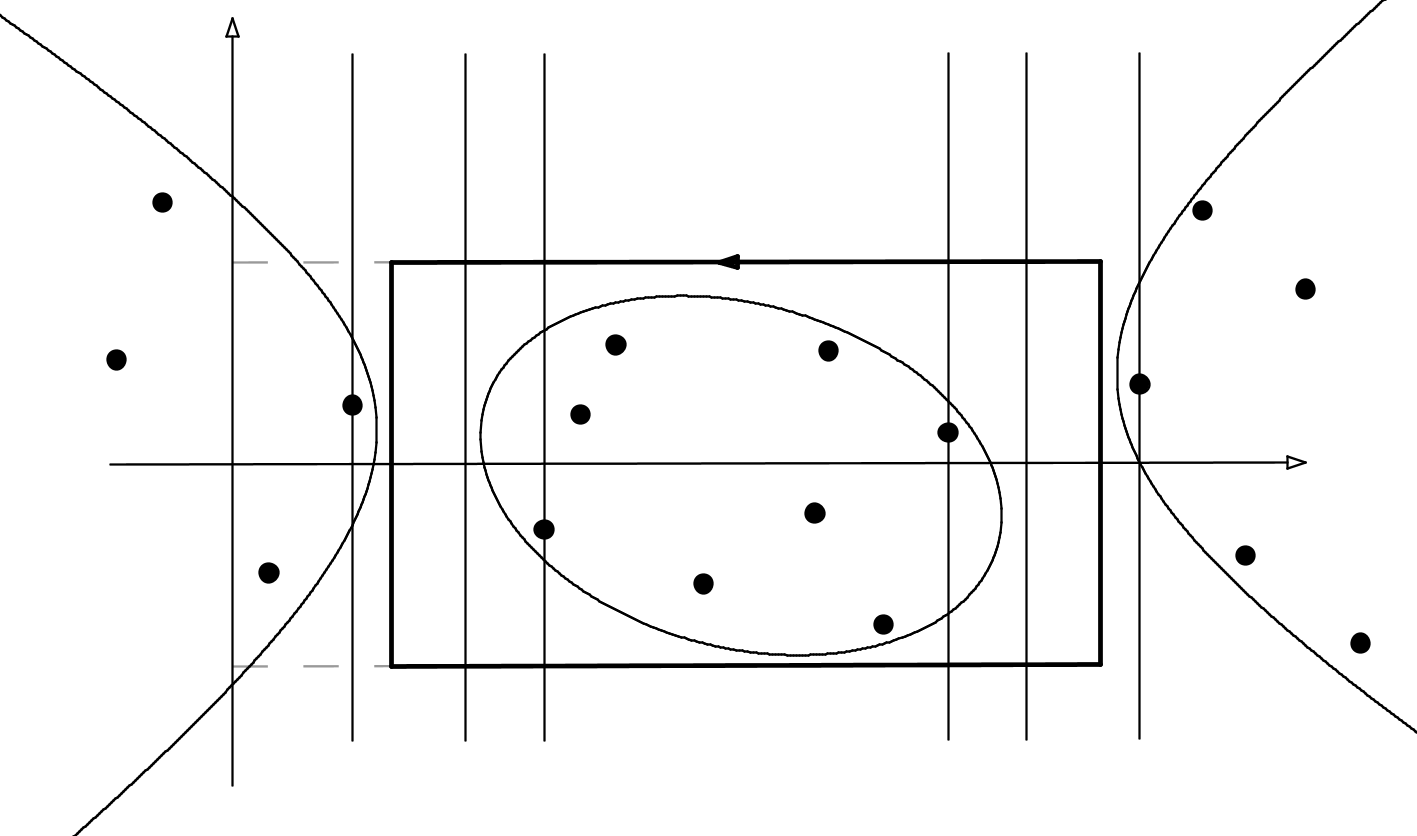}}
\put(178.181,559.308){$M$}
\put(141.937,151.332){$-M$}
\put(130.777,787.523){$\textup{Im }z$}
\put(1277.062,308.981){$\textup{Re }z$}
\put(750.163,612.694){$C_0$}
\put(308.147,52.597){$c-\varepsilon$}
\put(421.104,26.541){$c-\tilde\varepsilon$}
\put(974.247,26.541){$c'+\tilde\varepsilon$}
\put(1087.542,52.597){$c'+\varepsilon$}
\put(534.362,52.597){$c$}
\put(939.059,52.597){$c'$}
\put(641.256,402.364){$D_{c,c'}(y)$}
\put(1265.311,604.475){$E_{c,c'}(y)$}
\put(5.847,252.283){$E_{c,c'}(y)$}
\end{picture} 
\end{center}

\bigskip

The union of the small black disks inside the rectangle as a subset of $\C$ contains all $D_{c,c'}(y)$ for $y\in\overline U_0$ while the centers just form the set $D_{c,c'}(y_0).$ Then the piecewise smooth curve $C_0:=I_-\cup I'\cup I_+\cup I,$ taken in counter-clockwise orientation, surrounds the set $D_{c,c'}(y)$ for every $y\in \overline{U}_0.$ Thus, the integral
\begin{equation} \label{21func}
\langle\zeta_0(y),h\rangle := \int_{C_0}f(y,z)h(z)\dbar z,
\end{equation}
$\dbar z:=(2\pi i)^{-1}dz,$ for $h\in\mathcal{A}(\C),$ defines a family of analytic functionals \begin{equation} \label{21func0}
\zeta_0(y)\in C^\infty(U_0,\mathcal{A}'(K_0,L^{-\mu}_{\textup{cl}}(X)))
\end{equation} 
carried by the compact set
$$K_0:=\overline{\bigcup_{y\in \overline{U}_0}D_{c,c'}(y)}.$$
We use here notation and general terminology on analytic functionals, cf. \cite{Horm5}, vol.1; a selfcontained exposition may also be found in \cite{Kapa10}. In particular, if $E$ is a Fr\'echet space and $K\subset\C$ compact by $\mathcal{A}'(K,E)$ we denote the space of all $E$-valued analytic functionals carried by $K.$ In the case $E=\C$ we simply write $\mathcal{A}'(K)$ which is a nuclear Fr\'echet space, and then $\mathcal{A}'(K,E)=\mathcal{A}'(K)\hat{\otimes}_\pi E$ (with $\hat{\otimes}_\pi$ indicating the completed projective tensor product between the respective spaces). The background of the formula \eqref{21func} is the fact that for any $f(y,z)\in C^\infty(\Omega,\mathcal{A}(\C\setminus K,E)),\Omega\subseteq\R^q$ open, and for every (piecewise smooth) curve $C,$ counter-clockwise surrounding the set $K,$ the mapping 
\begin{equation} \label{21zeta}
h\mapsto\int_Cf(y,z)h(z)\dbar z,\quad h\in \mathcal{A}(\C),
\end{equation}
represents a $\zeta(y)\in C^\infty(\Omega,\mathcal{A}'(K,E)).$ We always choose the curve $C$ in such a way that its winding number with respect to every $z\in K$ is equal to $1.$ It can be proved (cf., for instance, \cite{Schu62}) that for every $\varepsilon>0$ there is a $C$ of that kind such that $\textup{dist}\,(K,C)<\varepsilon.$ \\Conversely, 
every $\zeta(y)\in C^\infty(\Omega,\mathcal{A}'(K,E))$ can be written in the form \eqref{21zeta}. For purposes below we give an explicit construction. Let $\omega(r)$ be a cut-off function on the half-axis, and apply the Mellin transform $(M_{r\rightarrow z}\omega)(z)=:\Phi(z).$ The function $\Phi(z)$ is meromorphic with a simple pole at $z=0,$ and $\Phi(z)-z^{-1} $ is an entire function. If $\chi(z)$ denotes any element of $C^\infty_0(\C)$ which is equal to $1$ in a neighbourhood of $0,$ then $\chi(z)\Phi(z)|_{\Gamma_\beta}\in\mathcal{S}(\Gamma_\beta),$ for every $\beta\in\R,$ uniformly in compact $\beta$-intervals. An analytic functional carried by $K$ can be applied to functions which are holomorphic in a neighbourhood of $K.$ In particular, for every $z\notin K$ the function $ \Phi(z-w)$ is holomorphic with respect to $w$ in a neighbourhood of $K.$ The pairing of $\zeta(y)\in C^\infty(\Omega,\mathcal{A}'(K,E))$ with $\Phi(z-w)$ in the variable $w$ gives us a function
\begin{equation} \label{21pair}
f_1(y,z):=\langle\zeta(y)_w,\Phi(z-w)\rangle\in C^\infty(\Omega,\mathcal{A}(\C\setminus K,E)).
\end{equation}
Note that the function \eqref{21pair} can be interpreted as the potential of $\zeta$ with respect to a fundamental solution of the Cauchy-Riemann operator (up to the factor $\pi^{-1}).$ If $\zeta(y)$ is defined by \eqref{21zeta} we have $f(y,z)-f_1(y,z)\in C^\infty(\Omega,\mathcal{A}(\C,E))$ and $\langle\zeta(y),h\rangle=\int_Cf_1(y,z)h(z)\dbar z$ for every $h\in \mathcal{A}(\C).$\\Observe that we can also write
\begin{equation} \label{21alt}
\langle\zeta(y)_w,\Phi(z-w)\rangle= M_{r\rightarrow z}\omega(r)\langle\zeta(y)_w,r^{-w}\rangle
\end{equation}
with $M_{r\rightarrow z}$ being interpreted as the weighted Mellin transform, with a weight $<1/2-\textup{Re}\,w.$
Applying this process to the present situation we obtain 
\begin{equation} \label{21func1}
\zeta_0(y)\in C^\infty(U_0,\mathcal{A}'(K_0,L^{-\infty}(X)))
\end{equation} 
which is more precise than \eqref{21func0}. This observation is compatible with another way of extracting the
asymptotic information from a Mellin symbol, namely, via the kernel cut-off Theorem \ref{21kcut}.
\begin{rem}
The operator function $f(y,z), z\in U_0,$ involved in \eqref{21func}, restricts to an element in $C^\infty(U_0,L^{-\mu}_{\textup{cl}}(X;\Gamma_\beta))$ for $\beta=c-\varepsilon_0$ (or $\beta=c'+\varepsilon_0$). Applying Theorem \ref{21kcut} we find a $k(y,z)\in\ C^\infty(U_0,M^{-\mu}_{\mathcal{O}}(X))$ such that  $k(y,z)|_{U_0\times\Gamma_\beta}-f(y,z)\in C^\infty(U_0,L^{-\infty}(X;\Gamma_\beta)).$ For any $\chi(z)$ that is equal to zero in a neighbourhood of the above-mentioned set $K_0$ with $\chi(z)=1$ for $\textup{dist}\,(z,K_0)>c$ for some $c>0$ we can apply 
Theorem \ref{21kcut} also to $\chi(z)f(y,z)|_{U_0\times\Gamma_\beta}$ for any other $\beta\in\R.$ This shows that $$(\chi(z)f(y,z)-k(y,z))|_{U_0\times\Gamma_\beta} \in C^\infty(U_0,L^{-\infty}(X;\Gamma_\beta)),$$
uniformly in compact $\beta$-intervals, and we obtain $$\langle\zeta_0(y),h\rangle := \int_{C_0}\{f(y,z)-k(y,z)\}h(z)\dbar z,$$ $h\in\mathcal{A}(\C),$ which is an integration over an $L^{-\infty}(X)$-valued function on $C_0$, and hence, $\zeta_0$ is $L^{-\infty}(X)$-valued as well.
\end{rem}
Let us set $b:=(c,c',U)$ for fixed reals $c<c'$ and $U\in \mathcal{U}(\Omega),$ and carry out the above construction for \eqref{21func1} for every $y_0\in\overline{U}.$ Then we obtain neighbourhoods $U_0$ of $y_0,$
associated compact sets $K_0,$ etc. The sets $U_0$ form an open covering of $\overline{U}.$ Let us choose a finite subcovering and denote it by $U_0,U_1,\ldots,U_N.$ There are associated compact sets 
$K_j, j=0,\ldots,N,$ and elements $$\zeta_j(y)\in C^\infty(U_j,\mathcal{A}'(K_j,L^{-\infty}(X))),j=0,\ldots,N.$$ Moreover, choose a system of functions $\varphi_j\in C^\infty(U_j),j=0,\ldots,N,$ such that $\sum_{j=0}^N\varphi_j(y)=1$ for all $y\in\overline{U}.$
Then $$\zeta(y):=\sum_{j=0}^N\varphi_j(y)\zeta_j(y)\in C^\infty(U,\mathcal{A}'(K_b,L^{-\infty}(X)))$$ for $K_b:=\bigcup_{j=0}^NK_j,$ and $$\langle\zeta(y)_w,\Phi(z-w)\rangle-f(y,z)$$ is holomorphic in $\{c<\textup{Re}\,z<c'\}$ for every $y\in U.$ In other words $\langle\zeta(y)_w,\Phi(z-w)\rangle$ has the same (singular parts of the) Laurent expansions at the points of
$D(y)\cap\{c<\textup{Re}\,z<c'\}$ as the function $f(y,z),$ for every $y\in U.$\\
The function $\sigma _{\textup{c}}(A)^{-1}(y,z)$ from \eqref{11bi} belongs to the essential ingredients of operator-valued Mellin symbols of the edge calculus. The distribution of the poles including multiplicities has the properties of the following definition.
\begin{defn} \label{21melas}
A variable discrete Mellin asymptotic type $\mathcal{R}$ over an open set $\Omega\subseteq\R^q$ is a system of sequences $\mathcal{R}(y)\subset\C\times\N$ as in \eqref{21mas} parametrised by $y\in\Omega,$ such that for every $b=(c,c',U),c<c',U\in\mathcal{U}(\Omega),$ there are sets
\begin{equation} \label{21sequ}
 \{U_i\}_{0\leq i\leq N},\quad\{K_i\}_{0\leq i\leq N}
\end{equation}
for some $N=N(b)\in\N$ where $U_i\in\mathcal{U}(\Omega),0\leq i\leq N,$ form an open covering of $\overline{U},$ moreover, $K_i\Subset\C,$ and
\begin{equation}
K_i\subset\{c-\varepsilon_i<\textup{Re}\,z<c'+\varepsilon_i\}\quad\mbox{for some}\quad\varepsilon_i>0, 
\end{equation}
\begin{equation}
\pi_\C\mathcal{R}(y)\cap\{c-\varepsilon_i<\textup{Re}\,z<c'+\varepsilon_i\}\subset K_i\quad\mbox{for all}\quad y\in U_i,
\end{equation}
and $\textup{sup}_{y\in U_i}\sum_j(1+l_j(y))<\infty$ where the sum is taken over those $j\in J(y)$ such that $p_j(y)\in K_i,i=0,\ldots,N.$ 
\end{defn}
The restriction of an $\mathcal{R}$ to an open subset $U\subseteq\Omega$ is defined by $\mathcal{R}|_U:= (\mathcal{R}(y))_{y\in U}.$ It satisfies Definition \ref{21melas} with respect to $U.$ Moreover, we can restrict $\mathcal{R}$ to a subset $A\subset \C$ by setting
\begin{equation} \label{21rest}
 p_A\mathcal{R}(y) :=\{(p(y),l(y))\in \mathcal{R}(y):p(y)\in A\},y\in \Omega .
\end{equation}
Such a definition allows us also to form restrictions $p_A\mathcal{R}|_U$ both with respect to $y$ and the complex variable $z.$ Another useful operation for variable discrete Mellin asymptotic types $\mathcal{R}_1,\dots,\mathcal{R}_L$ is the union $$\mathcal{R}_1\cup\dots\cup\mathcal{R}_L$$ defined in an obvious manner which yields again a variable discrete asymptotic type. Also infinite unions may be admitted, for instance, when we decompose an $\mathcal{R}$ as in Definition \ref{21melas} as 
\begin{equation} \label{21decasd}
 \mathcal{R}=\bigcup_{b=(c,c',U)}\bigcup_{0\leq i\leq N(b)} p_{K_i}\mathcal{R}|_{U_i}.
\end{equation}
The elements $\mathcal{R}=(\mathcal{R}(y))_{y\in\Omega}$ will classify $y$-dependent meromorphic functions $f(y,z)$ with poles $p_j(y)$ of multiplicity $\leq l_j(y)+1$ belonging 
to $\pi_\C\mathcal{R}(y).$ \\

Let us introduce the following general notation. Let $K\subset\C$ be compact, $U\in\mathcal{U}(\Omega)$, and $E$ a Fr\'echet space. By 
\begin{equation} \label{21asa}
 C^\infty(U,\mathcal{A}(\C\setminus K,E))^\bullet
\end{equation} 
we denote the subspace of all $f(y,z)\in C^\infty(U,\mathcal{A}(\C\setminus K,E))$ that have for every $y\in U$ an extension to a meromorphic function with finitely many poles in the set $K$ where the supremum over the sum of all multiplicities is less than some constant $M$ independent of $y\in U.$ With a function $f$ in the space \eqref{21asa} we can associate a family $\delta_f(y)\in C^\infty(U,\mathcal{A}'(K,E))$ in the usual way, namely, $\delta_f(y):h\mapsto \int_Cf(y,z)h(z)\dbar z, h\in\mathcal{A}(\C)$. Let
\begin{equation} \label{21ofu}
 C^\infty(U,\mathcal{A}'(K,E))^\bullet
\end{equation} 
denote the subspace of all $\delta\in C^\infty(U,\mathcal{A}'(K,E))$ of the form $\delta=\delta_f$ for some $f\in C^\infty(U,\mathcal{A}(\C\setminus K,E))^\bullet.$ \\
We say that a system of $E$-valued meromorphic functions $f(y,z),y\in\Omega,$ is subordinate to the variable discrete asymptotic type $\mathcal{R}$ over $\Omega$ if for any $y_0\in\Omega$ every pole $p$ of $f(y_0,z)$ belongs to $\pi_\C\mathcal{R}(y_0),$ say, $p=p_j(y_0)$ for some $j,$ and its multiplicity is $\leq l_j(y_0)+1.$\\
For every $b=(c,c',U),c<c',U\in\mathcal{U}(\Omega),$ and $U_i,K_i$ as in Definition \ref{21melas} we find smooth compact curves $C_i$ in the strips $\{c-\varepsilon_i<\textup{Re}\,z<c'+\varepsilon_i\},$ counter-clockwise surrounding the compact sets $K_i.$ This allows us to form $\delta_i(y)\in \mathcal{A}'(K_i,E)$ by $\langle\delta_i(y),h\rangle:=\int_{C_i}f(y,z)h(z)\dbar z.$ The function $f$ is called smooth in $y$ if $\delta_i(y)\in C^\infty(U_i,\mathcal{A}'(K_i,E))^\bullet$ for every 
$i.$ Then we have
\begin{equation} \label{21fu}
f_i(y,z):= M_{r\rightarrow z}\omega(r)\langle\delta_{i,w},r^{-w}\rangle\in C^\infty(U_i,\mathcal{A}(\C\setminus K_i,E))^\bullet,i=0,\ldots,N(b).
\end{equation} 
Let $C^\infty (\Omega ,\mathcal{M}_{\mathcal{R}}(E))$ denote the space of all such smooth functions $f$ subordinate to the variable discrete asymptotic type $\mathcal{R}.$

\begin{prop} \label{21Fre}
The space 
$C^\infty (\Omega ,\mathcal{M}_{\mathcal{R}}(E))$ is Fr\'echet in a canonical way.
\end{prop}
\begin{proof}
Let us set $f_b(y,z):=\sum_{i=0}^{N(b)}\varphi_i(y)f_i(y,z)$ for functions $\varphi_i(y)\in C^\infty_0(U_i)$ such that $\sum_{i=0}^{N(b)}\varphi_i(y)=1$ for all $y\in\overline{U},$ and $\delta_b(y,z):=\sum_{i=0}^{N(b)}\varphi_i(y)\delta_i(y,z).$ Then we have $$\delta_b\in C^\infty(U,\mathcal{A}'(K_b,E))^\bullet$$
for $K_b:=\bigcup_{i=0}^{N(b)}K_i,$ moreover, $f_b(y,z):= M_{r\rightarrow z}\omega(r)\langle\delta_{b,w},r^{-w}\rangle\in C^\infty(U,\mathcal{A}(\C\setminus K_b,E))^\bullet,$ and
 \begin{equation} \label{21dia}
f(y,z)-f_b(y,z)\in C^\infty(U,\mathcal{A}(\{c<\textup{Re}\,z<c'\},E)). 
\end{equation}
In this way for every $b=(c,c',U)$ we obtain a (non-direct) decomposition 
\begin{equation} \label{21decseq}
C^\infty(\Omega, \mathcal{M}_{p_{K_b}\mathcal{R}})|_U=C^\infty(U ,\mathcal{A}(\{c<\textup{Re}\,z<c'\},E))+C^\infty(U ,\mathcal{A}_{\mathcal{R}}(\C\setminus K_b,E))^\bullet. 
\end{equation}
Here
\begin{equation} \label{21subsp}
C^\infty(U ,\mathcal{A}_{\mathcal{R}}(\C\setminus K_b,E))^\bullet\subset C^\infty(U ,\mathcal{A}(\C\setminus K_b,E))^\bullet 
\end{equation}
means the subspace of all elements of the right hand side, subordinate to $p_{K_b}\mathcal{R}|_U.$ This in turn is a closed subspace of of the Fr\'echet space $C^\infty(U ,\mathcal{A}(\C\setminus K_b,E)).$ In fact, it suffices to verify that when $(f_\nu )_{\nu \in \N}$ is a sequence of elements of $C^\infty(U ,\mathcal{A}_{\mathcal{R}}\newline(\C\setminus K_b,E))^\bullet,$ convergent in $C^\infty(U ,\mathcal{A}(\C\setminus K_b,E)),$ the limit belongs to $C^\infty(U ,\mathcal{A}_{\mathcal{R}}\newline(\C\setminus K_b,E))^\bullet.$ However, for every fixed $y\in U$ the functions $f_\nu (y)$ extend to meromorphic functions with poles at the points $p(y)\in \pi _\C\mathcal{R}(y)\cap K_b$ of multiplicities $\leq l(y)+1.$ Then also the limit is meromorphic with those poles, including multiplicities and hence defines an element of $C^\infty(U ,\mathcal{A}_{\mathcal{R}}(\C\setminus K_b,E))^\bullet.$ We endow the space \eqref{21decseq} with the Fr\'echet topology of the non-direct sum of Fr\'echet spaces. Then $C^\infty (\Omega ,\mathcal{M}_{\mathcal{R}}(E))$ itself is Fr\'echet as an intersection of spaces \eqref{21decseq} over all $b=(c,c',U)$ in a countable set.
\end{proof}
\begin{defn} \label{21melas1}
Let $\Omega\subseteq\R^q$ be open, and $\mathcal{R}$ a variable discrete Mellin asymptotic type. A family of meromorphic functions $f(y,z)\in \mathcal{A}(\C\setminus\pi_\C\mathcal{R},L^{-\infty}(X)), y\in\Omega,$ is called a (smoothing) Mellin symbol with asymptotics of type $\mathcal{R}$ if $f$ is subordinate to $\mathcal{R}$ in the above-mentioned sense for $E=L^{-\infty}(X),$ i.e. for every $b=(c,c',U),c<c',U\in\mathcal{U}(\Omega),$ there is a $K_b\Subset\C$ such that for any cut-off function $\omega$ we have
\begin{equation}\label{21melhol}
\begin{split}
g(y,z):=(f(y,z)&-M_{r\rightarrow z}\omega(r)\langle\delta_{b,w},r^{-w}\rangle)|_{U\times\{c<\textup{Re}\,z<c'\}}\\&\in C^\infty(U,\mathcal{A}(\{c<\textup{Re}\,z<c'\},L^{-\infty}(X))
\end{split}
\end{equation}
for a suitable $\delta_b\in C^\infty(U,\mathcal{A}'(K_b,L^{-\infty}(X)))^\bullet.$ In addition we ask the property
\begin{equation} \label{21es}
g(y,z)|_{U\times\Gamma_\beta}\in C^\infty(U,\mathcal{S}(\Gamma_\beta,L^{-\infty}(X))
\end{equation}
for every $c<\beta<c'$, uniformly in compact $\beta$-intervals, and that the ($L^{-\infty}(X)$-valued) Laurent coefficients of $f(y,z)$ at the negative powers of $z-p,p\in\pi_\C\mathcal{R}(y),$ are of finite rank (required to be uniformly bounded in $y\in U,p\in K_b$, for every $b=(c,c',U)$). By
\begin{equation} \label{21et}
C^\infty(\Omega,\mathcal{M}_{\mathcal{R}}^{-\infty}(X))
\end{equation}
we denote the set of all those $f(y,z)$.
\end{defn}
\begin{rem}\label{21muas}
Note that for any $f(y,z)\in C^\infty(\Omega,\mathcal{M}_{\mathcal{R}}^{-\infty}(X))$ the system of pairs $\mathcal{R}(f):=\bigcup_{y\in \Omega }\mathcal{R}(f)(y)$ for
\begin{equation} \label{21muass}
\mathcal{R}(f)(y)\!:\!=\{(p(y),m(y))\in \C\times \N: p(y\!)\, \,\mbox{is a pole of}\, f(y,z\!) \,\,\mbox{of multiplicity }\,m(y)+1\}
\end{equation}
is a variable discrete Mellin asymptotic type, and $\mathcal{R}=\bigcup_{f\in C^\infty(\Omega,\mathcal{M}_{\mathcal{R}}^{-\infty}(X))}\mathcal{R}(f).$ 
\end{rem}
\begin{prop} \label{21merdec}
Let $f(y,z)\in C^\infty(U,\mathcal{A}(\C\setminus K,E))^\bullet$ for some compact set $K\subset \C$ and an open set $U\subset \R^q$ with compact closure. Then for every $\beta \in \R$ there are $0<\varepsilon _1<\varepsilon _0<\varepsilon $ for any $\varepsilon >0$ so small as we want and a decomposition
\begin{equation} \label{21merdec1}
f=f_0+f_1,f_i\in C^\infty(U,\mathcal{A}(\C\setminus K,E))^\bullet,i=0,1,
\end{equation}
such that $f_0$ is holomorphic in $\{\textup{Re}\,z>\beta +\varepsilon _0\}$ and $f_1$ in $\{\textup{Re}\,z<\beta +\varepsilon _1\},$ for all $y\in U.$
\end{prop}
\begin{proof}
We fix any $y_0\in \overline{U}$ and form the set $K(y_0)$ of all poles of $f(y_0,z)$ in $\{\textup{Re}\,z<\beta +\varepsilon (y_0)\}$ for some $\varepsilon (y_0)<\varepsilon .$ Let $C(y_0)$ be a curve counter-clockwise surrounding $K(y_0)$ in the half-plane $\{\textup{Re}\,z<\beta +\varepsilon \}$ such that all the other poles are outside that curve. Then there is an open neighbourhood $U(y_0)$ of $y_0$ such that all poles of $f(y,z),y\in U(y_0)$ in $\{\textup{Re}\,z<\beta +\varepsilon (y_0)\}$ are surrounded by $C(y_0)$. We produce a $\zeta \in C^\infty(U(y_0),\mathcal{A}'( K(y_0),E))^\bullet$ in the usual way by
\begin{equation} \label{21merdec2}
\langle \zeta (y),h(z)\rangle =\int_{C(y_0)}f(y,z)h(z)\dbar z,\,h \,\,\mbox{holomorphic in a neighbourhood of}\,\,K(y_0).
\end{equation}
Then $g_0(y,z):=M_{r\rightarrow z}(\omega (r)\langle \zeta _w(y),r^{-w}\rangle)$ belongs to $C^\infty(U(y_0),\mathcal{A}(\C\setminus K,E))^\bullet,$ and $f(y,z)-g_0(y,z)$ is holomorphic in $\{\textup{Re}\,z<\beta +\varepsilon (y_0)\}$ for all $y\in U(y_0).$ Doing this for every $y\in \overline{U}$ we get an open covering of $\overline{U}$ by open sets, and we find a finite subcovering $\{U(y_0),\dots,U(y_N)\}$ for some $N.$ Let us denote by $g_j$ the analogue of $g_0$ for the point $y_j.$ Then $f(y,z)-g_j(y,z)$ is holomorphic in $\{\textup{Re}\,z<\beta +\varepsilon (y_j)\}$ for all $y\in U(y_j).$ Now for a system of functions $\{\varphi _j\in C_0^\infty (U(y_j)),j=0,\dots,N\}$ such that $\sum\varphi _j\equiv 1$ over $\overline{U}$ we can form $f_0(y,z):=\sum_{j=0}^N\varphi _j(y)g_j(y,z)\in C^\infty(U,\mathcal{A}(\C\setminus K,E))^\bullet,$ and $f_1(y,z):= f(y,z)-f_0(y,z)$ is holomorphic in $\{\textup{Re}\,z<\beta +\textup{min}_{0\leq j\leq N}\{\varepsilon (y_j)\}\},$ while $f_0(y,z)$ is holomorphic in $\{\textup{Re}\,z>\beta +\textup{max}_{0\leq j\leq N}\{\varepsilon (y_j)\}\},$ for all $y\in U.$ This is the desired construction where $\varepsilon _1=\textup{min}_{0\leq j\leq N}\{\varepsilon (y_j)\},\varepsilon _0=\textup{max}_{0\leq j\leq N}\{\varepsilon (y_j)\}.$
\end{proof}
\begin{prop} \label{21mult}
Let $f_j(y,z)\in C^\infty(\Omega,\mathcal{M}_{\mathcal{R}_j}^{-\infty}(X))$ for variable discrete Mellin asymptotic types $\mathcal{R}_j, j=1,2;$ then we have $(f_1f_2)(y,z)\in C^\infty(\Omega,\mathcal{M}_{\mathcal{R}}^{-\infty}(X))$ for a resulting $\mathcal{R},$
and the multiplication defines a bilinear map
\begin{equation} \label{21bil}
C^\infty(\Omega,\mathcal{M}_{\mathcal{R}_1}^{-\infty}(X))\times C^\infty(\Omega,\mathcal{M}_{\mathcal{R}_2}^{-\infty}(X))\rightarrow C^\infty(\Omega,\mathcal{M}_{\mathcal{R}}^{-\infty}(X)).
\end{equation}
\end{prop}
The proof is elementary and left to the reader. Let us only note that the new asymptotic type $\mathcal{R}$ is not only affected by the involved poles and multiplicities but also by finite parts of Taylor expansions in the holomorphic regions of the factors.
\begin{prop} \label{21diff}
Let $f(y,z)\in C^\infty(\Omega,\mathcal{M}_{\mathcal{R}}^{-\infty}(X))$ for a variable discrete Mellin asymptotic type $\mathcal{R};$ then we have $D_y^\alpha f(y,z)\in C^\infty(\Omega,\mathcal{M}_{\mathcal{R_\alpha}}^{-\infty}(X))$ for every $\alpha\in\N^q$ and a resulting $\mathcal{R_\alpha},$ and $D_y^\alpha$ induces a linear operator
\begin{equation} \label{21bdiff}
D_y^\alpha:C^\infty(\Omega,\mathcal{M}_{\mathcal{R}}^{-\infty}(X))\rightarrow C^\infty(\Omega,\mathcal{M}_{\mathcal{R_\alpha}}^{-\infty}(X)).
\end{equation}
\end{prop}
\begin{proof}
The assertion is a consequence of the fact that the space $C^\infty(U,\mathcal{A}(\C\setminus K,E))^\bullet$ is preserved under differentiation with respect to $y,$ cf. \cite[Section 1.1.5, Theorem 6]{Schu2}.
\end{proof}
\begin{rem}\label{21dicont}
Our constructions have an analogue in the set-up of continuous asymptotics. Roughly speaking it suffices to forget about the pointwise discrete behaviour of analytic functionals and in the first part of Definition \ref{21melas1} to ask the conditions \eqref{21melhol} for suitable $\delta_b\in C^\infty(U,\mathcal{A}'(K_b,L^{-\infty}(X))).$ In the second part of Definition \ref{21melas1} we drop the condition on finite rank Laurent coefficients. Then instead of the asymptotic type $\mathcal{R}$ itself we can speak about a system of closed subsets $V\subset \C,V=V(U),$ such that $V\cap \{c\leq \textup{Re}\,z\leq c'\}$ is compact for every $c<c'.$ In some proofs it will be of help first to employ results for continuous asymptotics and then to observe the corresponding structure in the pointwise discrete case.
\end{rem}

\subsection{Smoothing Mellin symbols of the edge calculus}
Mellin symbols $f\in C^\infty(\Omega,\mathcal{M}_{\mathcal{R}}^{-\infty}(X))$ give rise to families of operators
\begin{equation} \label{21melop1}
\op_M^\beta (f)(y):C_0^\infty(X^\wedge)\rightarrow C^\infty(X^\wedge)
\end{equation}
for every $\beta \in \R$ with $\pi_\C\mathcal{R}(y)\cap\Gamma_{1/2-\beta}=\emptyset.$ 
These are involved in operators of the form
\begin{equation} \label{21melop2}
\omega_\eta r^{-\mu+j}\op_M^{\gamma_j-n/2} (f)(y)\omega'_\eta:\mathcal{K}^{s,\gamma}(X^\wedge)\rightarrow \mathcal{K}^{\infty,\gamma-\mu}(X^\wedge)
\end{equation}
for some $\mu\in\R$ and $j\in\N,$ and cut-off functions $\omega(r),\omega'(r),$
\begin{equation}\label{21omega}
\omega_\eta(r):=\omega(r[\eta]), 
\end{equation}
etc., $\eta\in\R^q.$ Here $\eta\mapsto [\eta]$ denotes some strictly positive function in $C^\infty(\R^q)$ 
with $[\eta]=|\eta|$ for $|\eta|\geq C$ for some $C>0.$ The choice of $C$ is unessential; for the sake of definiteness we assume $C=1.$ Concerning $\gamma_j=\gamma_j(y)$ we ask the conditions $\pi_\C\mathcal{R}(y)\cap\Gamma_{(n+1)/2-\gamma_j}=\emptyset$ (to avoid poles of $f(y)$ on the integration line), and $\gamma-j\leq\gamma_j(y)\leq\gamma$ (to have the continuity of the operator \eqref{21melop2}). For $j=0$ this means that the weight line $\Gamma_{(n+1)/2-\gamma}$ remains free of poles of $f(y,z)$ for all $y\in\Omega.$ For $j>0$ this is not necessary but we pass to $f_U,$ the restriction of $f$ to any $U\in\mathcal{U};$ this is adequate and sufficient for our purposes. In order to express an operator convention we first note that for every $y_0\in \overline{U}$ there is a $\gamma_j(y_0)$ with $\gamma-j\leq\gamma_j(y_0)\leq\gamma$ and a neighbourhood $U(y_0)\in \mathcal{U}(\Omega)$ of $y_0$ such that $\pi_\C\mathcal{R}(y)\cap\Gamma_{(n+1)/2-\gamma_j (y_0)}=\emptyset$ for all $y \in U(y_0).$ This can be done for every $y_0\in \overline{U},$ and since $\overline{U}$ is compact, there are finitely many points $y_0,\ldots,y_N$ such that the respective sets $U(y_0),\ldots,U(y_N)$ form an open covering of $\overline{U}.$ Choose a system of functions $\varphi_l\in C^\infty_0(U(y_l)),l=0,\ldots,N,$ such that $\sum_{l=0}^N \varphi_l(y)=1$ for all $y\in \overline{U}.$ Then we have
$f|_U(y,z)=\sum_{l=0}^N\varphi_l(y)f(y,z)$ for all $y\in U.$ We set
\begin{equation} \label{21melop3}
\omega_\eta r^{-\mu+j}\Op_M^{\gamma-n/2} (f|_U)(y)\omega'_\eta:= \omega_\eta r^{-\mu+j}\sum_{l=0}^N\op_M^{\gamma_j(y_l)-n/2} (\varphi_lf)(y)\omega'_\eta,
\end{equation}
$y\in U,$ for a fixed choice of data
\begin{equation} \label{21melconv}
\{U(y_l),\varphi_l,\gamma_j(y_l)\}_{l=0,\ldots,N}.
\end{equation}
\begin{rem} \label{21data}
Observe that when we have a sequence of Mellin symbols $f_{j\alpha}\in C^\infty(\Omega,\mathcal{M}_{\mathcal{R}_{j\alpha}}^{-\infty}(X))$ parametrised by finitely many indices $j,\alpha,$ then, in order to form expressions like \eqref{21melop3}, the above-mentioned $\{U(y_l),\varphi_l\}_{l=0,\ldots,N}$ can be chosen independently of $j,\alpha.$
\end{rem}
Finally, in order to define an operator convention globally over $\Omega$ we choose 
\begin{equation} \label{21melconv0n}
\{U^\iota,\varphi^\iota\}_{\iota\in I}
\end{equation}
where $U^\iota,\iota\in I,$ is a locally finite covering of $\Omega$ by sets $U^\iota\in\mathcal{U}(\Omega)$ and $\varphi^\iota,\iota\in I,$ a subordinate partition of unity. Then we define
\begin{equation} \label{21melconvglob}
\omega_\eta r^{-\mu+j}\Op_M^{\gamma-n/2} (f)(y)\omega'_\eta:= \sum_{\iota\in I}\varphi^\iota\omega_\eta r^{-\mu+j}\Op_M^{\gamma-n/2} ( f|_{U^\iota})(y)\omega'_\eta
\end{equation}
The question of characterising remainders under changing the data \eqref{21melconv}, \eqref{21melconv0n}, the cut-off functions, or the function $\eta\mapsto [\eta]$  leads to so-called Green symbols of the edge calculus, here with variable asymptotics. These will be studied below.
For the moment we fix these data and establish a number of important properties of the operator functions \eqref{21melop3}. First, we have $m_j(y,\eta):=\omega_\eta r^{-\mu+j}\Op_M^{\gamma-n/2} (f|_U)(y)\omega'_\eta\in C^\infty(U,\mathcal{L}(\mathcal{K}^{s,\gamma}(X^\wedge), \mathcal{K}^{\infty,\gamma-\mu}(X^\wedge)))$ for every $s$ and
\begin{equation} \label{21melhom}
m_j(y,\lambda\eta)=\lambda^{\mu-j}\kappa_\lambda m_j(y,\eta)\kappa_\lambda^{-1}
\end{equation}
for all $(y,\eta),|\eta|\geq1$ and $\lambda\geq1.$ Here
\begin{equation} \label{21kappa}
(\kappa_\lambda u)(r,x):=\lambda^{(n+1)/2}u(\lambda r,x);
\end{equation}
recall that $n=\textup{dim}\,X.$ The operators $\kappa_\lambda,\lambda\in\R_+,$ form a group action on the spaces $\mathcal{K}^{s,\gamma}(X^\wedge).$ In general, by a group action $\kappa=\{\kappa_\lambda\}_{\lambda\in\R_+}$ on a Hilbert space $H$ we understand a group of isomorphisms $\kappa_\lambda:H\rightarrow H, \lambda\in\R_+,$ such that $\kappa_\lambda\kappa_{\lambda'}=\kappa_{\lambda\lambda'}$ for every $\lambda, \lambda'\in\R_+$ and $\kappa_\lambda h\in C(\R_+,H)$ for every $h\in H$ (i.e. strong continuity). A similar notation is used when $H$ is replaced by a Fr\'echet space $E,$ written as a projective limit of Hilbert spaces $E_j,j\in\N,$ and continuous embeddings $E_j\hookrightarrow E_0$ for all $j$ where $E_0$ is endowed with a group action which restricts to a group action on $E_j$ for every $j.$ \\
Given two pairs $\{H,\kappa\},\{\tilde{H},\tilde{\kappa}\}$ of Hilbert spaces with group action, we have spaces 
\begin{equation} \label{21symb}
S^\mu(U\times\R^q;H,\tilde{H})\quad\mbox{and}\quad S_{\textup{cl}}^\mu(U\times\R^q;H,\tilde{H})
\end{equation}
of operator-valued symbols of order $\mu$ over an open set $U\subseteq \R^p.$ These are defined as follows. The first space of \eqref{21symb} is the set of all $a(y,\eta)\in C^\infty(U\times\R^p, \mathcal{L}(H,\tilde{H}))$ such that 
\begin{equation} \label{21symbest}
\|\tilde{\kappa}^{-1}_{\langle\eta\rangle}(D_y^\alpha D_\eta^\beta a(y,\eta))\kappa_{\langle\eta\rangle}\|_{\mathcal{L}(H,\tilde{H})}\leq c\langle\eta\rangle^{\mu-|\beta|}
\end{equation}
for all $(y,\eta)\in K\times\R^q,K\Subset U,\alpha\in\N^p, \beta\in\N^q,$ for constants $c=c(K,\alpha,\beta)>0.$ 
\begin{rem} \label{21cinf}
The space
\begin{equation} \label{21reminf}
S^{-\infty }(U\times\R^q;H,\tilde{H}):=\bigcap_{\mu \in \R}S^\mu (U\times\R^q;H,\tilde{H})
\end{equation}
is independent of the choice of $\kappa $ and $\tilde{\kappa }$.
\end{rem}
In fact, there are constants $c,M,\tilde{c},\tilde{M}>0$ such that 
$$\| \kappa _\lambda \| _{\mathcal{L}(H)}\leq c\,\textup{max}\,\{\lambda ,\lambda ^{-1}\}^M,
\|\tilde{\kappa }  _\lambda\|_{\mathcal{L}(\tilde{H})}\leq \tilde{c}\,\textup{max}\,\{\lambda ,\lambda ^{-1}\}^{\tilde{M}}.$$
If $\delta $ and $\tilde{\delta}$ are other group actions on the respective spaces we have analogous estimates, with exponents $D$ and $\tilde{D},$ respectively. Thus, \eqref{21symbest} implies
\begin{equation} \label{21symbestalt}
\begin{split}
\|\tilde{\delta }^{-1}_{\langle\eta\rangle}&(D_y^\alpha D_\eta^\beta a(y,\eta))\delta _{\langle\eta\rangle}\|_{\mathcal{L}(H,\tilde{H})}\\&=\|\tilde{\delta }^{-1}_{\langle\eta\rangle}\tilde{\kappa}_{\langle\eta\rangle}\tilde{\kappa}^{-1}_{\langle\eta\rangle}(D_y^\alpha D_\eta^\beta a(y,\eta))\kappa_{\langle\eta\rangle}\kappa^{-1}_{\langle\eta\rangle}\delta _{\langle\eta\rangle}\|_{\mathcal{L}(H,\tilde{H})}\\&=\|\tilde{\delta }^{-1}_{\langle\eta\rangle}\tilde{\kappa}_{\langle\eta\rangle}\|_{\mathcal{L}(\tilde{H})}\|\tilde{\kappa}^{-1}_{\langle\eta\rangle}(D_y^\alpha D_\eta^\beta a(y,\eta))\kappa_{\langle\eta\rangle}\|_{\mathcal{L}(H,\tilde{H})}\|\kappa^{-1}_{\langle\eta\rangle}\delta _{\langle\eta\rangle}\|_{\mathcal{L}(H)}\\&\leq c\langle\eta\rangle^{\nu -|\beta |}\,\,\mbox{for}\,\,\nu =\mu+M+D+\tilde{M}+\tilde{D}.
\end{split}
\end{equation}
In other words, if a symbol is of order $\mu $ with respect to $\kappa ,\tilde{\kappa}$ it is of order $\nu $ with respect to $\delta  ,\tilde{\delta },$ and the intersection over all orders is the same in both cases.
\begin{lem} \label{21clsymb}
A function $a(y,\eta)\in C^\infty(U\times\R^q, \mathcal{L}(H,\tilde{H}))$ with the homogeneity property $a(y,\lambda\eta)=\lambda^\mu\tilde{\kappa}_\lambda a(y,\eta)\kappa_\lambda^{-1}$ for all $|\eta|\geq C,\lambda\geq 1,$ for some $C>0,$ belongs to the space $S_{\textup{cl}}^\mu(U\times\R^q;H,\tilde{H}).$
\end{lem}
The proof is of the same structure as in the scalar case where it is evident.
Functions as in Lemma \ref{21clsymb} will be called homogeneous of order $\mu$ for large $|\eta|.$\\
The second space of \eqref{21symb} of so-called classical symbols is the set of all $a(y,\eta)\in S^\mu(U\times\R^q;H,\tilde{H}) $ such that there are functions $a_{\mu-j}(y,\eta)\in C^\infty(U\times\R^q, \mathcal{L}(H,\tilde{H})),\newline j\in\N,$ homogeneous of order $\mu-j$ for large $|\eta|,$ such that $r_N(y,\eta ):=a(y,\eta)-\sum_{j=0}^Na_{\mu-j}(y,\eta)\in S^{\mu-(N+1)}(U\times\R^q;H,\tilde{H}),$ for every $N\in\N.$ For references below by 
\begin{equation} \label{21hm}
S^{(\nu )}(U\times (\R^q\setminus \{0\});H,\tilde{H}) 
\end{equation}
we denote the subspace of all $a_{(\nu )}(y,\eta)\in C^\infty(U\times (\R^q\setminus \{0\}), \mathcal{L}(H,\tilde{H}))$ with the homogeneity property
\begin{equation} \label{21hm1}
a_{(\nu )}(y,\lambda \eta)=\lambda ^\nu \tilde{\kappa }_\lambda a_{(\nu )}(y, \eta)\kappa_\lambda ^{-1}
\end{equation}
for all $\lambda \in \R_+,(y,\eta)\in U\times (\R^q\setminus \{0\}).$ The space \eqref{21hm} is Fr\'echet in the topology induced by $ C^\infty(U\times (\R^q\setminus \{0\}), \mathcal{L}(H,\tilde{H})).$
\begin{rem}\label{21syfr}
The spaces \eqref{21symb} are Fr\'echet in a natural way. In the case $S^\mu(U\times\R^q;H,\tilde{H})$ the best possible constants in the symbolic estimates \eqref{21symbest} can be taken as the semi-norm system. For the subspace $S_{\textup{cl}}^\mu(U\times\R^q;H,\tilde{H})$ we take these semi-norms as well, and in addition the ones from the homogeneous components $a_{(\mu-j)}(y,\eta)\in S^{(\mu-j )}(U\times (\R^q\setminus \{0\});H,\tilde{H}),$  uniquely determined by $a(y,\eta ),$ together with those from the remainders $a(y,\eta)-\sum_{j=0}^N\chi (\eta )a_{(\mu-j)}(y,\eta)$ in the space  $S^{\mu-(N+1)}(U\times\R^q;H,\tilde{H}),$ for any excision function $\chi $.
\end{rem}
Analogous constuctions make sense when we replace $\tilde{H}$ by a Fr\'echet space $E$ with group action. Then, with the above notation,
\begin{equation} \label{21fre}
S_{\textup{(cl)}}^\mu(U\times\R^q;H,E):= \lim_{\substack{\longleftarrow\\j\in\N}}S_{\textup{(cl)}}^\mu(U\times\R^q;H,E_j). 
\end{equation}
Subscript ``(cl)" is used when we refer both to the classical and the general case. There is also a notion of symbols when both spaces are Fr\'echet with group action, cf., for instance, \cite{Schu20}, but this is not so urgent at the moment. Let $S_{\textup{(cl)}}^\mu(\R^q;\cdot,\cdot)$ denote the respective subspaces of symbols that are independent of $y.$ \\Clearly the choice of the group actions affects our symbol spaces, cf. also the estimates \eqref{21symbestalt}. If necessary we write $S_{\textup{(cl)}}^\mu(U\times\R^q;\cdot,\cdot)_{\kappa,\tilde{\kappa}}.$ In the definitions we also admit the case of trivial group actions (i.e. identity operators for all $\lambda\in\R_+$). These are always taken when the Hilbert space is of finite dimension. If both $H$ and $\tilde{H}$ are equal to $\C$ then we recover the scalar symbol spaces $S_{\textup{(cl)}}^\mu(U\times\R^q).$ \\
Let us set
\begin{equation} \label{21mmmel}
m(y,\eta):=\sum_{j=0}^k\sum_{|\alpha|\leq j}\omega_\eta r^{-\mu+j}\Op_M^{\gamma-n/2} (f_{j\alpha})(y)\eta^\alpha\omega'_\eta
\end{equation}
for elements $f_{j\alpha}(y,z)\in C^\infty(\Omega,\mathcal{M}_{\mathcal{R}_{j\alpha}}^{-\infty}(X))$ for variable discrete Mellin asymptotic types $\mathcal{R}_{j\alpha}.$ 
\begin{prop} \label{21smmel}
We have $m(y,\eta)\in S_{\textup{cl}}^\mu(\Omega\times\R^q;\mathcal{K}^{s,\gamma}(X^\wedge),\mathcal{K}^{\infty,\gamma-\mu}(X^\wedge))$
for every $s\in\R.$
\end{prop}
\begin{proof}
The $(j,\alpha)\mbox{th}$ summand of \eqref{21mmmel} is homogeneous of order $\mu-j+|\alpha|$ for large $|\eta|.$ Thus, it suffices to apply Lemma \ref{21clsymb}.
\end{proof}
Operator functions $m(y,\eta)$ of the form \eqref{21mmmel} will be referred to as smoothing Mellin symbols of the edge calculus, here with variable discrete asymptotics.
\begin{prop} \label{21smad}
The family $m^*(y,\eta)$ of formal adjoints of \eqref{21mmmel}, defined by $(m(y,\eta)u,v)_{\mathcal{K}^{0,0}(X^\wedge)}=(u,m^*(y,\eta)v)_{\mathcal{K}^{0,0}(X^\wedge)}$ for all $u,v\in C_0^\infty(X^\wedge),$ has the form
\begin{equation} \label{21add}
m^*(y,\eta)=\sum_{j=0}^k\sum_{|\alpha|\leq j}\omega'_\eta  \Op_M^{-\gamma-n/2} (f^*_{j\alpha })(y) r^{-\mu+j}\eta^\alpha\omega_\eta
\end{equation}
where $\sum_{j=0}^k\sum_{|\alpha|\leq j}\sum_{\iota\in I}\sum_{l=0}^N\omega'_\eta \varphi^\iota\varphi_l \op_M^{-\gamma_j(y_l)-n/2} (f^*_{j\alpha})(y) r^{-\mu+j}\eta^\alpha\omega_\eta $ is the interpretation of the right hand side of \eqref{21add}, and $f^*_{j\alpha}(y,z)=f^{(*)}_{j\alpha }(y,n+1-\overline{z})$ with $(^*)$ as the pointwise formal adjoint of operators over $X.$
\end{prop}
\begin{proof}
The formal adjoint can be carried out for the involved summands separately; then the assertion follows by a straightforward computation.
\end{proof}
Let us now compare operators $\op_M^{\delta-n/2}(f)(y)$ and $\op_M^{\delta+\beta-n/2}(f)(y)$ for different $\delta,\beta\in\R$ and $f\in C^\infty(\Omega,M^{-\infty}_{\mathcal{R}}(X))$ under the assumption $\pi_\C\mathcal{R}(y)\newline\cap\{\Gamma_{(n+1)/2-\delta}\cup\Gamma_{(n+1)/2-(\delta+\beta)}\}=\emptyset$ for all $y\in U,$ for some $U\in \mathcal{U}(\Omega),$ first applied to functions in $C^\infty_0(X^\wedge).$ Writing $(T^\sigma f)(y,z)=f(y,z+\sigma)$ for any real $\sigma$ we have
$$\op_M^{\delta-n/2}(f)=r^{\delta-n/2}\op_M(T^{-\delta+n/2}f)r^{-\delta+n/2}$$
for $\op_M:=\op_M^0,$ and
$\op_M^{\delta+\beta-n/2}(f)=r^\beta\op_M^{\delta-n/2}(T^{-\beta}f)r^{-\beta}.$ We assume $u\in C_0^\infty (\R_+,C^\infty(X));$ then $Mu(z)\in\mathcal{A}(\C,C^\infty(X))$ is strongly decreasing on lines parallel to the imaginary axis, uniformly in finite intervals with respect to the real part. Thus,
\begin{equation}\label{21mcomm}
\begin{split}
\op_M^{\delta+\beta-n/2}&(f)(y)u(r)=r^\beta\op_M^{\delta-n/2}(T^{-\beta}f)(y)r^{-\beta}u(r)\\
&=\int_{\Gamma_{(n+1)/2-\delta}}r^{-z+\beta}f(y,z-\beta) (Mu)(z-\beta)\dbar z\\
&=\int_{\Gamma_{(n+1)/2-(\delta+\beta)}}r^{-w}f(y,w) Mu(w)\dbar w\\
&=\int_{\Gamma_{(n+1)/2-\delta}}r^{-z}f(y,z) Mu(z)\dbar z+G(y)u(r)
\end{split}
\end{equation}
for 
\begin{equation} \label{21curve}
 G(y)u(r):=\int_{\Delta_{\delta,\beta}}r^{-z}f(y,z)Mu(z)\dbar z
\end{equation}
where $\Delta_{\delta,\beta}=\Gamma_{(n+1)/2-\delta}\cup\Gamma_{(n+1)/2-(\delta+\beta)}.$ In other words,
\begin{equation} \label{21G}
 G(y):=\op_M^{\delta+\beta-n/2}(f)(y)-\op_M^{\delta-n/2}(f)(y).
\end{equation}
The numbers $\delta$ and $\beta$ are arbitrary; thus, without loss of generality we assume $\beta>0.$ The orientation of parallels to the imaginary axis is from $\textup{Im}\,z=-\infty$ to $+\infty.$ Thus, $\Delta_{\delta,\beta}$ in 
\eqref{21curve} may be replaced by a compact smooth counter-clockwise oriented curve $C_{\delta,\beta}$ in the strip $\{(n+1)/2-(\delta+\beta)<\textup{Re}\,z<(n+1)/2-\delta\},$ surrounding all poles of $f(y,z)$ in this strip for all $y\in U.$ From \eqref{21curve}, \eqref{21G} it follows that 
\begin{equation} \label{21g}
\begin{split}
g(y,\eta)&u(r):=\omega_\eta\op_M^{\delta+\beta-n/2}(f)(y)\omega'_\eta u(r)-\omega_\eta\op_M^{\delta-n/2}(f)(y)\omega'_\eta u(r)\\
&=\omega_\eta\int_0^\infty\Big\{\int_ {C_{\delta,\beta}}(r/r')^{-z}f(y,z)\dbar z\Big\} u(r')\omega'_\eta(r')dr'/r'\\
&=\omega_\eta\int_0^\infty\Big\{\int_ {C_{\delta,\beta}}(r[\eta]/r'[\eta])^{-z}f(y,z)\dbar z\Big\} u(r')\omega'_\eta(r')dr'/r'.
\end{split}
\end{equation}

\subsection{Green symbols}
Green symbols are specific operator-valued symbols, well-known in the special case of (pseudo-differential) boundary value problems, where the associated so-called Green operators characterise the contribution from the boundary in parametrices of elliptic problems. For instance, the Green function of the Dirichlet problem for the Laplace equation in a smooth domain is a sum of such a Green operator and a fundamental solution of the Laplacian. The situation is similar for edge problems where the boundary is substituted by an edge, and the half space, locally modelling the domain near the boundary, is replaced by a wedge. In contrast to the comparatively very simple case of boundary value problems with the transmission property at the boundary (the Dirichlet problem for the Laplacian is just of that type) the Green symbols in edge problems inherit the full complexity of asymptotic phenomena such as variable discrete asymptotics coming from meromorphic inverses of families of Mellin symbols, and it is necessary to analyse their nature in a separate consideration.\\
In studying boundary or edge problems one of the main issues is to recognise the structure (and, of course, also the role) of the Green symbols in the calculus. Recall that in boundary value problems with the transmission property at the boundary a Green symbol of order $m$ (and type $0$) is a $g(y,\eta)\in C^\infty(U\times\R^q,\mathcal{L}(L^2(\R_+),L^2(\R_+)))$ which has the structure of a symbol such that 
\begin{equation} \label{22gtrans}
g(y,\eta), g^*(y,\eta)\in S_{\textup{cl}}^m(U\times\R^q;L^2(\R_+),\mathcal{S}(\overline{\R}_+)).
\end{equation}
Here $^*$ means the $(y,\eta)$-wise $L^2(\R_+)$-adjoint, the open set $U\subseteq\R^q$ corresponds to a chart on the boundary of dimension $q,$ moreover, $\R_+$ to the inner normal, and $\mathcal{S}(\overline{\R}_+)=\mathcal{S}(\overline{\R})|_{\R_+}.$ As the group action in \eqref{22gtrans} we take $\kappa_\lambda:u(r)\mapsto \lambda^{1/2}u(\lambda r),\lambda \in \R_+,$ which also makes sense on $\mathcal{S}(\overline{\R}_+)=\projlim_{j\in\N}\langle r\rangle^{-j}H^j(\R_+).$
A Green symbol of the edge calculus is a
\begin{equation*}
g(y,\eta)\in C^\infty(U\times\R^q,\mathcal{L}(\mathcal{K}^{s,\gamma}(X^\wedge),\newline\mathcal{K}^{\infty,\gamma-\mu}(X^\wedge)))
\end{equation*}
for some weight shift parameter $\mu,$ such that
\begin{equation} \label{22g}
g(y,\eta)\in S_{\textup{cl}}^m(U\times\R^q;\mathcal{K}^{s,\gamma;g}(X^\wedge),\mathcal{K}_{\mathcal{P}}^{\infty,\gamma-\mu;\infty}(X^\wedge)),
\end{equation}
\begin{equation} \label{22gad}
g^*(y,\eta)\in S_{\textup{cl}}^m(U\times\R^q;\mathcal{K}^{s,-\gamma+\mu;g}(X^\wedge),\mathcal{K}_{\mathcal{Q}}^{\infty,-\gamma;\infty}(X^\wedge))
\end{equation}
for every $s,g\in\R.$
In this case the open set $U\subseteq\R^q$ corresponds to a chart on the edge of dimension $q,$ and $\mathcal{P},\mathcal{Q}$ indicate certain additional asymptotic properties in the spaces in the image. Note that the conditions in \eqref{22gtrans} concerning $\mathcal{S}(\overline{\R}_+)$ just imply Taylor asymptotics in the image at the boundary. In the edge case we have to be aware that the asymptotic properties may depend on the variable $y.$ The relations \eqref{22g} and \eqref{22gad} make sense with spaces with constant discrete asymptotic types $\mathcal{P},\mathcal{Q}$ as in Section 2.1. There are also continuous asymptotic types that admit such a definition, cf. \cite{Schu2}, \cite{Schu20}, or \cite{Schu55}. However, in the variable discrete case we should find an alternative description, since the involved spaces, e.g. $\mathcal{K}_{\mathcal{P}(y)}^{\infty,\gamma-\mu;\infty}(X^\wedge)$ depend on $y\in U.$ Let us assume for the moment that $\mathcal{P}$ and $\mathcal{Q}$ are constant discrete asymptotic types. Then we have a representation of Green symbols by a kernel function.\\
Given a Fr\'echet space $E$ by $S^m(U\times\R^q,E)$ for an open set $U\subseteq \R^p$ we denote the set of all $a(y,\eta)\in C^\infty(U\times\R^q,E)$ such that for every $\pi$ (belonging to the countable system of semi-norms of $E$) we have $\pi(D_y^\alpha D_\eta^\beta a(y,\eta))\leq c\langle\eta\rangle^{m-|\beta|}$ for all $(y,\eta)\in K\times\R^q$ and $K\Subset U,$  for all multi-indices $\alpha\in\N^p,\beta\in\R^q,$ and constants $c=c(\alpha,\beta,K)>0.$ Moreover, $S^m_{\textup{cl}}(U\times\R^q,E),$ the subspace of classical $E$-valued symbols $a(y,\eta)$ is defined by the condition of existence of functions $a_{m-j}(y,\eta)\in C^\infty(U\times\R^q,E),j\in\N,$ with $a_{m-j}(y,\lambda\eta)=\lambda^{m-j}a_{m-j}(y,\eta)$ for all $\lambda\geq1,|\eta|\geq C$ for some $C>0,$ and $a(y,\eta)-\sum_{j=0}^Na_{m-j}(y,\eta)\in S^{m-(N+1)}(U\times\R^q,E)$ for every $N\in\N.$\\
Now let us set $E:=E^1\cap E^2$ for
\begin{equation} \label{21E}
\begin{split}
E^1&:=\mathcal{K}_{\mathcal{P}}^{\infty,\gamma-\mu;\infty}(X^\wedge_{r,x})\hat{\otimes}_\pi\mathcal{K}^{\infty,-\gamma;\infty}(X^\wedge_{r',x'}),\\E^2&:=\mathcal{K}^{\infty,\gamma-\mu;\infty}(X^\wedge_{r,x})\hat{\otimes}_\pi\mathcal{K}_{\overline{{\mathcal{Q}}}}^{\infty,-\gamma;\infty}(X^\wedge_{r',x'})
\end{split}
\end{equation}
where $\hat{\otimes}_\pi$ indicates the completed projective tensor product between the respective Fr\'echet spaces, and $\overline{\mathcal{Q}}=\{(\overline{q},m):(q,m)\in \mathcal{Q}\}.$ 
\begin{thm} \label{21grn}
Consider a function
\begin{equation} \label{21Gr}
k(y,\eta;r,x,r',x')\in S^\nu_{\textup{cl}}(U\times\R^q,E),
\end{equation}
$U\subseteq\R^q$ open, for discrete asymptotic types $\mathcal{P}$ and $\mathcal{Q}$ associated with the weight data $(\gamma-\mu,\Theta)$ and $(-\gamma,\Theta),$ respectively. Then the family of mappings 
\begin{equation} \label{21Gmap}
g(y,\eta):u\mapsto \int_X\!\!\int_0^\infty k(y,\eta;r[\eta],x,r'[\eta],x')u(r',x')(r')^ndr'dx'
 \end{equation}
defines a Green symbol in the sense of the relations \eqref{22g}, \eqref{22gad}, for $\nu=m+n+1$. Conversely, every such $g(y,\eta)$ admits a representation \eqref{21Gmap} for some $k$ as in \eqref{21Gr}.
\end{thm}

The first part of the latter theorem is straightforward. The second part is proved in \cite{Schu55} (also in the version of continuous asymptotics). Note that Green symbols are $(y,\eta)$-wise Green operators in the cone algebra on the (open stretched) infinite cone $X^\wedge.$ These have kernel characterisations, too; they can be obtained in a much more precise form, cf. the paper of Seiler \cite{Seil2}.
\begin{rem}\label{21gralt}
The choice of the function $\eta \mapsto [\eta ]$ in \eqref{21Gmap} affects the respective symbol by a Green symbol of order $-\infty .$
\end{rem}
In fact, if $\eta \rightarrow [\eta ]_1$ is another function of this type (i.e. smooth, strictly positive, and equal to $|\eta |$ for large $|\eta |$) then
$$k(y,\eta ;r[\eta ],x,r'[\eta ],x')-k(y,\eta ;r[\eta ]_1,x,r'[\eta ]_1,x')$$
is of compact support in $\eta $ and hence generates via \eqref{21Gmap} a symbol with the properties \eqref{22g}, \eqref{22gad}, for $m=-\infty .$
Let us now assume the weight interval $\Theta$ to be finite (the infinite case is then automatic since it is reduced to the case of any finite $\Theta$). Similarly as \eqref{1asK} we have a direct decomposition \begin{equation}\label{21asdde}
\mathcal{K}_{\mathcal{P}}^{\infty,\gamma-\mu;\infty}(X^\wedge)= \mathcal{K}_\Theta^{\infty,\gamma-\mu;\infty}(X^\wedge)+\mathcal{E}_{\mathcal{P}}
\end{equation}  
which allows us to write $\mathcal{K}_{\mathcal{P}}^{\infty,\gamma-\mu;\infty}(X^\wedge)\hat{\otimes}_\pi\mathcal{K}^{\infty,-\gamma;\infty}(X^\wedge)=E^1_\Theta+E^1_{\mathcal{P}}$ for
\begin{equation} \label{21dec5}
E^1_\Theta:=\mathcal{K}_\Theta^{\infty,\gamma-\mu;\infty}(X^\wedge)\hat{\otimes}_\pi \mathcal{K}^{\infty,-\gamma;\infty}(X^\wedge),
E^1_{\mathcal{P}}:=\mathcal{E}_{\mathcal{P}}\hat{\otimes}_\pi\mathcal{K}^{\infty,-\gamma;\infty}(X^\wedge).
 \end{equation}

In a similar manner, we obtain $\mathcal{K}^{\infty,\gamma-\mu;\infty}(X^\wedge)\hat{\otimes}_\pi\mathcal{K}_{\overline{\mathcal{Q}}}^{\infty,-\gamma;\infty}(X^\wedge)=E^2_\Theta+E^2_{\mathcal{Q}}$ for 
\begin{equation} \label{21dec7}
E^2_\Theta:=\mathcal{K}^{\infty,\gamma-\mu;\infty}(X^\wedge)\hat{\otimes}_\pi \mathcal{K}_\Theta^{\infty,-\gamma;\infty}(X^\wedge),
E^2_{\mathcal{Q}}:=\mathcal{K}^{\infty,\gamma-\mu;\infty}(X^\wedge)\hat{\otimes}_\pi \mathcal{E}_{\mathcal{Q}}.
 \end{equation}

The property \eqref{21Gr} is equivalent to $k(y,\eta;r,x,r',x')\in S^\nu_{\textup{cl}}(U\times\R^q,E^i)$ \newline for $i=1,2,$ or
\begin{equation} \label{21dec8}
\begin{split}
k(y,\eta;r,x,r',x')&\in S^\nu_{\textup{cl}}(U\times\R^q,E^1_{\mathcal{P}})+S^\nu_{\textup{cl}}(U\times\R^q,E^1_\Theta),\\
k(y,\eta;r,x,r',x')&\in S^\nu_{\textup{cl}}(U\times\R^q,E^2_{\mathcal{Q}})+S^\nu_{\textup{cl}}(U\times\R^q,E^2_\Theta).
\end{split}
\end{equation}
Applying \eqref{21Gmap} to the decompositions \eqref{21dec8} we obtain 
\begin{equation} \label{21dec9}
g(y,\eta)=g^1_{\mathcal{P}}(y,\eta)+g^1_\Theta(y,\eta),g^*(y,\eta)=g^2_{\mathcal{Q}}(y,\eta)+g^2_\Theta(y,\eta)
\end{equation}
with obvious meaning of notation.
The summands with subscript $\Theta$ represent symbols that produce flatness of order $\Theta$ (relative to the involved weights) under the mappings. The specific parts are those with $\mathcal{P}$ and $\mathcal{Q},$ respectively. Let us consider, for instance, $\mathcal{P}.$ The space $\mathcal{E}_{\mathcal{P}}$ consists of all functions of the form $\omega(r)\langle\zeta_z,r^{-z}\rangle$ (with the pairing in $z$) for arbitrary $\zeta\in\mathcal{A}'(\pi_\C \mathcal{P},C^\infty(X))$ that are represented by $C^\infty(X)$-valued meromorphic functions with poles at the points $p_j\in\pi_\C \mathcal{P}$ of multiplicity $m_j+1,j=0,\ldots,N_{\mathcal{P}},$ cf. the formula \eqref{1as}. Analogously, in connection with $\mathcal{Q}$ we speak about the points $q_l\in\pi_\C \mathcal{Q}$ of multiplicity $n_l+1,l=0,\ldots,N_{\mathcal{Q}}.$ This gives us the following result.
\begin{thm}\label{21ggrep}
A Green symbol $g(y,\eta)$ as in Theorem \textup{\ref{21grn}} is characterised by the decompositions \eqref{21dec9} for flat symbols $g_\Theta^i(y,\eta),i=1,2,$ and asymptotic parts
\begin{equation} \label{21anP}
\begin{split}
g^1_{\mathcal{P}}(y,\eta)u(r,x)&=\int\!\!\int_0^\infty \omega_\eta(r)\langle \zeta^1(y,\eta,r'[\eta],x,x'),(r[\eta])^{-z}\rangle u(r',x'))(r')^n dr'dx',\\
g^2_{\mathcal{Q}}(y,\eta)v(r',x')&=\int\!\!\int_0^\infty \omega_\eta(r')\langle \zeta^2(y,\eta,r[\eta],x,x'),(r'[\eta])^{-z}\rangle v(r,x))r^n drdx.
\end{split}
\end{equation}
for $\zeta^1\in\mathcal{A}'(\pi_\C \mathcal{P},S^\nu_{\textup{cl}}(U\times\R^q,C^\infty(X_x)\hat{\otimes}_\pi\mathcal{K}^{\infty,-\gamma;\infty}(X^\wedge_{r',x'}))),$
represented by meromorphic functions with poles at the points $p_j\in\pi_\C \mathcal{P}$ of multiplicity $m_j+1,j=0,\ldots,N_{\mathcal{P}},$ and $\zeta^2\in\mathcal{A}'(\pi_\C Q,S^\nu_{\textup{cl}}(U\times\R^q,C^\infty(X_{x'})\hat{\otimes}_\pi\mathcal{K}^{\infty,-\gamma+\mu;\infty}(X^\wedge_{r,x}))),$ represented by meromorphic functions with poles at the points $q_l\in\pi_\C \mathcal{Q}$ of multiplicity $n_l+1,l=0,\ldots,N_{\mathcal{Q}}.$
\end{thm}
\begin{rem}\label{21ggcont}
Theorem \ref{21ggrep} has an immediate analogue in the case of continuous asymptotics, cf. \cite{Schu55}. If the carriers of the involved asymptotic types are compact sets in the respective weight strips, i.e. $K\subset \{(n+1)/2-(\gamma -\mu )+\vartheta <\textup{Re}\,z<(n+1)/2-(\gamma -\mu )\}$ in the case of $\mathcal{P}$, and similarly for $\mathcal{Q}$, the formal scheme of the proof is the same as for Theorem \ref{21ggrep}.
\end{rem}
In fact, we have a decomposition analogously as \eqref{21asdde}, namely,
\begin{equation}\label{21asdconnt}
\mathcal{K}_{\mathcal{P}}^{\infty,\gamma-\mu;\infty}(X^\wedge)= \mathcal{K}_\Theta^{\infty,\gamma-\mu;\infty}(X^\wedge)+\mathcal{E}_K
\end{equation}
for
\begin{equation}\label{21asfuctc}
\mathcal{E}_K :=\{\omega (r)\langle\zeta ,r^{-z}\rangle:\zeta \in \mathcal{A}'(K,C^\infty (X))\},
\end{equation}
and similarly for $\mathcal{Q}$ for another compact set. In the following we argue for $\mathcal{P};$ the constructions for $\mathcal{Q}$ are analogous. Similarly as \eqref{21dec5} we can form
\begin{equation}\label{21asfuctc1}
E_K :=\mathcal{E}_K\hat{\otimes}_\pi\mathcal{K}^{\infty,-\gamma;\infty}(X^\wedge)
\end{equation}
and replace $S^\nu_{\textup{cl}}(U\times\R^q,E^1_{\mathcal{P}})$ by its continuous analogue
$$S^\nu_{\textup{cl}}(U\times\R^q,E^1_K)  =C^\infty (U,S^\nu_{\textup{cl}}(\R^q,E^1_K)) = S^\nu_{\textup{cl}}(\R^q,C^\infty (U,E^1_K)).$$
The latter space can be identified with
\begin{equation}\label{21disc}
S^\nu_{\textup{cl}}(\R^q,C^\infty (U,\mathcal{E}_K))\hat{\otimes}_\pi\mathcal{K}^{\infty,-\gamma;\infty}(X^\wedge).
\end{equation}
This allows us to formulate an analogue for the variable discrete case, namely, to replace $C^\infty (U,\mathcal{E}_K)$ in \eqref{21disc} by
\begin{equation}\label{21asfuctc3}
\{\omega (r)\langle\zeta ,r^{-z}\rangle:\zeta \in C^\infty (U,\mathcal{A}'(K,C^\infty (X))^\bullet\}.
\end{equation}
However, in order to have reasonable Fr\'echet space structures the variable discrete behaviour will be controlled by corresponding asymptotic types:
\begin{defn} \label{21vars}
A variable discrete asymptotic type $\mathcal{P}$ over an open set $\Omega\subseteq\R^q$ associated with the weight data $(\gamma,\Theta),\Theta=(\vartheta,0],\infty<\vartheta<0,$ is a system of sequences
\begin{equation}
\mathcal{P}(y)=\{(p_j(y),m_j(y))\}_{j=0,\ldots,J(y)} 
\end{equation}
for $J(y)\in\N,y\in\Omega,$ such that $\pi_\C\mathcal{P}(y)=\{(p_j(y))\}_{j=0,\ldots,J(y)}\subseteq\{(n+1)/2-\gamma+\vartheta<\textup{Re}\,z<(n+1)/2-\gamma\}$ for all $y\in\Omega,$ and for every $b:=(c,U)$ with $(n+1)/2-\gamma+\vartheta<c<(n+1)/2-\gamma, U\in\mathcal{U}(\Omega),$ there are sets
\begin{equation} \label{21sequ1}
 \{U_i\}_{0\leq i\leq N},\quad\{K_i\}_{0\leq i\leq N}
\end{equation}
for some $N=N(b)\in\N$ where $U_i\in\mathcal{U}(\Omega),0\leq i\leq N,$ form an open covering of $\overline{U},$ moreover, $K_i\Subset\C,$ and
\begin{equation}
K_i\subset\{c-\varepsilon_i<\textup{Re}\,z<(n+1)/2-\gamma\}\quad\mbox{for some}\quad\varepsilon_i>0, 
\end{equation}
\begin{equation}
\pi_\C\mathcal{P}(y)\cap\{c-\varepsilon_i<\textup{Re}\,z\}\subset K_i\quad\mbox{for all}\quad y\in U_i,
\end{equation}
and $\textup{sup}_{y\in U_i}\sum_j(1+m_j(y))<\infty$ where the sum is taken over those $0\leq j\leq J(y)$ such that $p_j(y)\in K_i,i=0,\ldots,N.$ 
 \end{defn}
We will say that a variable discrete asymptotic type $\mathcal{P}$ satisfies the shadow condition if
\begin{equation} \label{21shadow}
\begin{split}
(p(y),m(y))& \in \, \mathcal{P}(y) \,\,\mbox{implies}\,\, (p(y)-l,m(y))\in \mathcal{P}(y) \\& \mbox{for all}\,\, l=l(y)\in \N \,\,\mbox{such that}\,\, \textup{Re}\,p(y)-l>(n+1)/2-\gamma +\vartheta ,y\in \Omega .
\end{split}
\end{equation} 
Similarly as in Section 2.2 we can restrict $\mathcal{P}$ to open sets $U\subseteq \Omega $ and sets $A\subset \C$ which gives us again variable discrete asymptotic types $p_A\mathcal{P}|_U$ associated with $(\gamma ,\Theta ).$ Moreover, if $E$ is a Fr\'echet space we say that a system of $E$-valued meromorphic functions $f(y,z),y\in\Omega,$ in the strip $\{(n+1)/2-\gamma+\vartheta<\textup{Re}\,z<(n+1)/2-\gamma\}$ is  subordinate to the variable discrete asymptotic type $\mathcal{P}$ over $\Omega$ as in Definition \ref{21vars} if for any $y_0\in\Omega$ every pole $p$ of $f(y_0,z)$ belongs to $\pi_\C\mathcal{P}(y_0),$ say, $p=p_j(y_0)$ for some $j,$ and its multiplicity is $\leq m_j(y_0)+1.$\\
For every $b=(c,U)$ and $U_i,K_i$ as in Definition \ref{21vars} we choose smooth compact curves $C_i$ in the strips $\{c-\varepsilon_i<\textup{Re}\,z<(n+1)/2-\gamma\},$ counter-clockwise surrounding the compact sets $K_i,$ and we define $\delta_i(y)\in \mathcal{A}'(K_i,E)$ by   $\langle\delta_i(y),h\rangle:=\int_{C_i}f(y,z)h(z)\dbar z, h\in\mathcal{A}(\C),y\in U_i.$ The function $f$ is called smooth in $y$ if $\delta_i(y)\in C^\infty(U_i,\mathcal{A}'(K_i,E))^\bullet$ for every $i.$ This gives us elements 
\begin{equation} \label{21fu1}
f_i(y,z):= M_{r\rightarrow z}\omega(r)\langle\delta_{i,w},r^{-w}\rangle\in C^\infty(U_i,\mathcal{A}(\C\setminus K_i,E))^\bullet,i=0,\ldots,N(b). 
\end{equation}
Let $f_b(y,z):=\sum_{i=0}^{N(b)}\varphi_i(y)f_i(y,z)$ for functions $\varphi_i(y)\in C^\infty_0(U_i)$ with $\sum_{i=0}^{N(b)}\varphi_i(y)\newline=1$ for all $y\in\overline{U},$ and $\delta_b(y,z):=\sum_{i=0}^{N(b)}\varphi_i(y)\delta_i(y,z).$ Then 
\begin{equation} \label{21bul}
\delta_b\in C^\infty(U,\mathcal{A}'(K_b,E))^\bullet
 \end{equation}
for $K_b:=\bigcup_{i=0}^{N(U)}K_i,$ moreover, $f_b(y,z):= M_{r\rightarrow z}\omega(r)\langle\delta_{b,w},r^{-w}\rangle\in C^\infty(U,\mathcal{A}(\C\setminus K_b,E))^\bullet,$ and
 \begin{equation} \label{21dia2}
f(y,z)-f_b(y,z)\in C^\infty(U,\mathcal{A}(\{c-\varepsilon <\textup{Re}\,z<(n+1)/2-\gamma\},E)) 
\end{equation}
for any $0<\varepsilon <\textup{min}_{i=0,\dots,N(b)}\{\varepsilon _i\}=:\varepsilon (b).$ In the following definition we set 
\begin{equation} \label{21dia4}
E^1:= S^\nu_{\textup{cl}}(\R^q,C^\infty(X_x)\hat{\otimes}_\pi\mathcal{K}^{\infty,-\gamma;\infty}(X^\wedge_{r',x'})),
\end{equation}
\begin{equation} \label{21dia5}
E^2:= S^\nu_{\textup{cl}}(\R^q,C^\infty(X_{x'})\hat{\otimes}_\pi\mathcal{K}^{\infty,-\gamma+\mu;\infty}(X^\wedge_{r,x}))
\end{equation}
for $\nu:=m+n+1.$
\begin{defn} \label{21gr1}
The space $R^m_G(\Omega \times \R^q,{\bf{g}})_{\mathcal{P},\mathcal{Q}},{\bf{g}}=(\gamma ,\gamma -\mu ,\Theta ),$ of Green symbols of order $m\in\R$ with variable discrete asymptotic types $\mathcal{P}$ and $\mathcal{Q}$ associated with the weight data $(\gamma-\mu,\Theta)$ and $(-\gamma,\Theta),$ respectively, is defined to be the set of 
\begin{equation} \label{22g1}
g(y,\eta)\in S_{\textup{cl}}^m(\Omega \times\R^q;\mathcal{K}^{s,\gamma;e}(X^\wedge),\mathcal{K}^{\infty,\gamma-\mu;\infty}(X^\wedge)),
\end{equation}
with
\begin{equation} \label{22gad1}
g^*(y,\eta)\in S_{\textup{cl}}^m(\Omega \times\R^q;\mathcal{K}^{s,-\gamma+\mu;e}(X^\wedge),\mathcal{K}^{\infty,-\gamma;\infty}(X^\wedge))
\end{equation}
for every $s,e\in\R,$ such that for every $b^1:=(c^1,U)\in\big((n+1)/2-(\gamma-\mu)+\vartheta,(n+1)/2-(\gamma-\mu)\big) \times\mathcal{U}(\Omega),b^2:=(c^2,U)\in\big((n+1)/2+\gamma+\vartheta,(n+1)/2+\gamma\big) \times\mathcal{U}(\Omega)$ there are $K^1_{b^1}\Subset \{(n+1)/2-(\gamma-\mu)+\vartheta<\textup{Re}\,z<(n+1)/2-(\gamma-\mu)\},$ $K^2_{b^2}\Subset \{(n+1)/2+\gamma+\vartheta<\textup{Re}\,z<(n+1)/2+\gamma\}$ and elements 
\begin{equation} \label{22gad3}
\zeta^l_{b^l}\in C^\infty(U,\mathcal{A}'(K^l_{b^l},E^l))^\bullet,\quad l=1,2,
\end{equation}
described by families $f^l_{b^l}(y,z)$ of $E^l$-valued meromorphic functions over $U,$ subordinate to the variable discrete asymptotic type $\mathcal{P}|_U$ for $l=1$ and $\mathcal{Q}|_U$ for $l=2,$
such that for
\begin{equation} \label{21anP1}
\begin{split}
g_{b^1,\mathcal{P}}(y,\eta)u(r,x)&\!:=\!\int\!\!\int_0^\infty \!\omega_\eta(r)\langle \zeta^1_{b^1}(y,\eta,r'[\eta],x,x'),\!(r[\eta])^{-z}\rangle u(r',x'))(r')^n dr'\!dx',\\
g_{b^2,\mathcal{Q}}(y,\eta)v(r',x')&:=\int\!\!\int_0^\infty \omega_\eta(r')\langle \zeta^2_{b^2}(y,\eta,r[\eta],x,x'),(r'[\eta])^{-z}\rangle v(r,x))r^n drdx.
\end{split}
\end{equation}
we have
\begin{equation} \label{22gad5}
\begin{split}
&g(y,\eta)-g_{b^1,\mathcal{P}}(y,\eta)\in S_{\textup{cl}}^m(U\times\R^q;\mathcal{K}^{s,\gamma;g}(X^\wedge),\mathcal{K}^{\infty,\gamma-\mu+\beta^1;\infty}(X^\wedge)),\\
&g^*(y,\eta)-g_{b^2,\mathcal{Q}}(y,\eta)\in S_{\textup{cl}}^m(U\times\R^q;\mathcal{K}^{s,-\gamma+\mu;g}(X^\wedge),\mathcal{K}^{\infty,-\gamma+\beta^2;\infty}(X^\wedge))
\end{split}
\end{equation}
for $\beta^1=\beta^1_0+\varepsilon $ for any $0<\varepsilon <\varepsilon ^1=\varepsilon ^1(b^1),$ $\beta^1_0:=(n+1)/2-(\gamma-\mu)-c^1,$ and $\beta^2=\beta^2_0+\varepsilon $ for any $0<\varepsilon <\varepsilon ^2=\varepsilon ^2(b^2), \beta^2_0:=(n+1)/2+\gamma-c^2$ (for brevity in \eqref{22gad5} we wrote $g(y,\eta )$ rather than $g(y,\eta )|_{U\times \R^q}$). We set
\begin{equation}\label{22grn}
R^m_G(\Omega \times \R^q,{\bf{g}}):=\bigcup_{\mathcal{P},\mathcal{Q}}R^m_G(\Omega \times \R^q,{\bf{g}})_{\mathcal{P},\mathcal{Q}}.
\end{equation}
\end{defn}
\begin{rem}\label{grfr}
The space $R^m_G(\Omega \times \R^q,{\bf{g}})_{\mathcal{P},\mathcal{Q}}$ is Fr\'echet in a natural way.
\end{rem}
In fact, the arguments concerning the asymptotic part are similar to those for Proposition \ref{21Fre}.
A special subspace of Green symbols is
\begin{equation}\label{22grnflat}
R^m_G(\Omega \times \R^q)_{\mathcal{O}},
\end{equation}
the set of symbols $g(y,\eta )$ of infinite flatness, defined by the properties
\begin{equation}\label{22infflat}
g(y,\eta),g^*(y,\eta)\in S_{\textup{cl}}^m(\Omega \times\R^q;\mathcal{K}^{s,\gamma;e}(X^\wedge),\mathcal{K}^{\infty,\delta ;\infty}(X^\wedge)),
\end{equation}
for any reals $\gamma ,\delta, e. $
Observe that
\begin{equation}\label{22grnadas}
g(y,\eta )\in R^m_G(\Omega \times \R^q,{\bf{g}})_{\mathcal{P},\mathcal{Q}}\Leftrightarrow g^*(y,\eta )\in R^m_G(\Omega \times \R^q,{\bf{g^*}})_{\mathcal{Q},\mathcal{P}}
\end{equation}
for ${\bf{g}}=(-\gamma +\mu ,-\gamma ,\Theta ).$
\begin{rem} \label{22weight1}
Definition \ref{21gr1} admits a straightforward generalisation to a notion of Green symbols $g(y,\eta)$ between any pairs of weights $\gamma,\delta$ rather than $\gamma,\gamma-\mu.$ However, the main application here concerns the weight shift $\mu$ which is coming from the order of the non-smoothing operators in the edge calculus.
\end{rem}
For references below we fix the following notation.
\begin{defn} \label{21m+g}
The space 
\begin{equation}
R^\mu _{M+G}(\Omega \times \R^q,{\bf{g}})
\end{equation}
for ${\bf{g}}=(\gamma ,\gamma -\mu ,\Theta ),\Theta =(-(k+1),0],k\in \N,$ is the set of all operator functions $(m+g)(y,\eta )$ for any $m(y,\eta )$ of the form \eqref{21mmmel} and $g(y,\eta )\in R^\mu _G(\Omega \times \R^q,{\bf{g}}).$
\end{defn}
\begin{prop} \label{24grop}
Let $g(y,\eta )\in R^m _G(\Omega \times \R^q,{\bf{g}}),{\bf{g}}=(\gamma ,\gamma -\mu ,\Theta ).$ Then we have the following properties:
\begin{itemize}
\item[\textup{(i)}] $r^j g(y,\eta ), g(y,\eta )r^j\in R^{m -j}_G(\Omega \times \R^q,{\bf{g}})$ for every $j\in \N;$
\item[\textup{(ii)}] $a(y,\eta )g(y,\eta )\in R^{m +\nu }_G(\Omega \times \R^q,{\bf{g}})$ for every $ a(y,\eta )\in S^\nu _{\textup{cl}}(\Omega \times \R^q);$
\item[\textup{(iii)}] $D_y^\alpha D_\eta ^\beta g(y,\eta )\in R^{m -|\beta | }_G(\Omega \times \R^q,{\bf{g}})$ for every $\alpha ,\beta \in \N^q.$
\end{itemize}
\end{prop}
\begin{proof}
$\textup{(i)}$ From the first relation of \eqref{22gad5} it follows that $r^jg(y,\eta)-r^jg_{b^1,\mathcal{P}}(y,\eta)\in S_{\textup{cl}}^{m-j}(U\times\R^q;\mathcal{K}^{s,\gamma;g}(X^\wedge),\mathcal{K}^{\infty,\gamma-\mu+\beta^1;\infty}(X^\wedge)).$ Thus, setting for the moment $\tilde{g}_{b^1,\mathcal{P}}\newline(y,\eta):=r^jg_{b^1,\mathcal{P}}(y,\eta)$ we have to show that there is a $\tilde{\zeta}^1_{b^1}$ of analogous structure as \eqref{22gad3} for $l=1$ such that there is an analogous relation between $\tilde{g}_{b^1,\mathcal{P}}$ and $\tilde{\zeta}^1_{b^1}$ as in \eqref{21anP1}. Multiplying the right hand side of \eqref{21anP1} by $r^j$ gives us
\begin{equation}
\begin{split}
r^j& \int\!\!\int_0^\infty \omega_\eta(r)\langle \zeta^1_{b^1}(y,\eta,r'[\eta],x,x'),(r[\eta])^{-z}\rangle u(r',x'))(r')^n dr'dx'\\&=\int\!\!\int_0^\infty \omega_\eta(r)\langle [\eta]^{-j}\zeta^1_{b^1}(y,\eta,r'[\eta],x,x'),(r[\eta])^{-z+j}\rangle u(r',x'))(r')^n dr'dx'\\&=\int\!\!\int_0^\infty \omega_\eta(r)\langle [\eta]^{-j}(T^{-j}\zeta^1_{b^1})(y,\eta,r'[\eta],x,x'),(r[\eta])^{-z}\rangle u(r',x'))(r')^n dr'dx'.
\end{split}
\end{equation}
Here $T^{-j}\zeta$ means the analytic functional translated to the left in the complex plane, defined by
$\langle T^{-j}\zeta,h\rangle=\langle\zeta ,T^{-j}h\rangle $ for $(T^{-j}h)(z)=h(z-j).$ Because of the factor $[\eta]^{-j}$ we have $\tilde{\zeta}^1_{b^1}=[\eta]^{-j}(T^{-j}\zeta^1_{b^1})\in C^\infty (U,\mathcal{A}'(\tilde{K}^1_{b^1},[\eta]^{-j}E^1))^\bullet$ (in the notation of \eqref{21dia4}) where $\tilde{K}^1_{b^1}$ is the translation of $K^1_{b^1}$ to the left by $-j.$ In particular, the order in $\eta $ is diminished by $j.$ Concerning the position of $\tilde{K}^1_{b^1}$ relative to the original weight strip it might happen now that this set has a non-trivial intersection with $\Gamma _{(n+1)/2-(\gamma -\mu) +\vartheta }.$ However, we can write $\tilde{\zeta}^1_{b^1}$ as a sum $\tilde{\zeta}^1_{b^1,0}+\tilde{\zeta}^1_{b^1,1}$ where the first summand contributes a flat symbol of order $m-j$ in the sense of the first relation of \eqref{22gad5} while the second summand belongs to a family of analytic functionals carried by a compact set in $\{(n+1)/2-(\gamma-\mu)+\vartheta<\textup{Re}\,z<(n+1)/2-(\gamma-\mu)\},$ as required in the definition, cf. analogously, Proposition \ref{21merdec}. Thus, we have characterised $r^jg(y,\eta)$ in the desired way. The arguments for $g(y,\eta)r^j$ are similar when we pass to the formal adjoint, cf. also the relation \eqref{22grnadas}.
$\textup{(ii)}$ is straightforward.
$\textup{(iii)}$ Derivatives in $y,\eta $ may be carried out in the relations \eqref{22gad5}, because of the symbolic estimates \eqref{21symbest}. The $\eta $-differentiations of \eqref{21anP1} are straightforward concerning the first $\eta $-variables, while in the differentiations with respect to the remaining $\eta $-dependence we produce extra powers of $r$ and $r'$ that can be treated in a similar manner as in the proof of $\textup{(i)}.$ Differentiations in $y$ are possible for similar reasons as in the proof of Proposition \ref{21diff}.
\end{proof}
\begin{prop} \label{24asymp}
Let $g_j(y,\eta )\in R^{m -j}_G(\Omega \times \R^q,{\bf{g}})_{\mathcal{P},\mathcal{Q}}, j\in \N,$ be arbitrary Green symbols with $j$-independent asymptotic types $\mathcal{P},\mathcal{Q}.$ Then there is an asymptotic sum $g(y,\eta )\sim \sum_{j=0}^\infty g_j(y,\eta )$ in $R^m _G(\Omega \times \R^q,{\bf{g}})_{\mathcal{P},\mathcal{Q}}$ in the sense that $g(y,\eta )-\sum_{j=0}^Ng_j(y,\eta )\in R^{m -(N+1)}_G(\Omega \times \R^q,{\bf{g}})_{\mathcal{P},\mathcal{Q}}$ for every $N\in \N,$ and $g(y,\eta )$ is unique $\textup{mod}\,R^{-\infty }_G(\Omega \times \R^q,{\bf{g}})_{\mathcal{P},\mathcal{Q}}.$
\end{prop}
\begin{proof}
We employ the Fr\'echet topology of the space $R^m _G(\Omega \times \R^q,{\bf{g}})_{\mathcal{P},\mathcal{Q}},$ cf. Remark \ref{grfr}, and $g(y,\eta )$ as a convergent sum $\sum_{j=0}^\infty \chi (\eta /c_j)g_j(y,\eta )$ for an excision function $\chi (\eta )$ and a sequence $c_j>0$ tending to $\infty $ sufficiently fast as $j\rightarrow \infty .$
\end{proof}
Let us now turn to other properties of Green and Mellin symbols that play a role in the calculus.
\begin{prop} \label{24grdiff}
Let $U\in \mathcal{U}(\Omega )$ and consider an $f\in C_0^\infty (U,\mathcal{M}_{\mathcal{R}}^{-\infty }(X))$ such that $\pi _\C{\mathcal{R}}\cap \Gamma _{(n+1)/2-\gamma }=\emptyset$ for all $y\in U,$ and set $m(y,\eta ):=r^{-\mu }\omega _\eta \op_M^{\gamma -n/2}(f)\omega '_\eta .$ Then we have 
\begin{equation}\label{24dff}
D_\eta ^\alpha m(y,\eta )\in R^{\mu -|\alpha |}_G(\Omega \times \R^q,{\bf{g}})
\end{equation}
for every $\alpha \in \N,\alpha \neq 0,{\bf{g}}=(\gamma ,\gamma -\mu ,\Theta ).$
\end{prop}
\begin{proof}
For our Mellin operators we assume that $\Theta $ is finite, i.e. $\Theta =(-(k+1),0]$ for some $k\in \N$. We study, for instance, the derivative $\partial _{\eta _j}$ for any $0\leq j\leq q.$ Then the assertion for higher derivatives is a consequence of Proposition \ref{24grop}, $\textup{(iii)}$. We have 
$$\partial _{\eta _j}m(y,\eta )=r^{-(\mu -1)}\{\omega _\eta\op_M^{\gamma -n/2}(f)(y)(\partial _{\eta _j}\omega '_\eta)+(\partial _{\eta _j}\omega _\eta)\op_M^{\gamma -n/2}(f)(y)\omega '_\eta \}. $$
The first summand can be written as
\begin{equation}\label{24dff00}
g_1(y,\eta )=r^{-(\mu -1)}\omega _\eta\op_M^{\gamma -n/2}(f)(y)\varphi  _\eta \partial _{\eta _j}[\eta ]
\end{equation}
for $\varphi :=\partial _r\omega \in C_0^\infty (\R_+),\varphi  _\eta(r)=\varphi (r[\eta ]).$ Since $\partial _{\eta _j}[\eta ]\in S^0 _{\textup{cl}}(\Omega \times \R^q)$ by virtue of Proposition \ref{24grop} it suffices to show $g_0(y,\eta ):=r^{-\mu }\omega _\eta\op_M^{\gamma -n/2}(f)(y)\varphi  _\eta \in R^\mu _G(\Omega \times \R^q,{\bf{g}}).$ As noted in Remark \ref{21dicont} it makes sense first to interpret $f$ is a Mellin symbol with continuous asymptotics, carried by the compact set $K_b$ as in Definition \ref{21melas1}. For sufficiently large $\beta >0$ the Mellin symbol $f(y,z)$ is holomorphic in a neighbourhood of the weight line $\Gamma _{(n+1)/2-(\gamma +\beta )},$ and we can form
\begin{equation}\label{24dff1}
g_2(y,\eta ):=\omega _\eta\op_M^{(\gamma+\beta ) -n/2}(f)(y)\varphi  _\eta-\omega _\eta\op_M^{\gamma -n/2}(f)(y)\varphi  _\eta.
\end{equation}
Similarly as \eqref{21g} we have
\begin{equation}\label{24dff2}
g_2(y,\eta )=\omega_\eta\int_0^\infty\Big\{\int_ {C_{\gamma ,\beta}}(r[\eta]/r'[\eta])^{-z}f(y,z)\dbar z\Big\} u(r')\varphi _\eta(r')dr'/r'
\end{equation}
(for the computation it is unessential that we have $\varphi _\eta$ on the right instead of $\omega  _\eta$). For sufficiently large $\beta $ we have
$$g_0(y,\eta ):=r^{-(\mu -1)}\omega _\eta\op_M^{(\gamma+\beta ) -n/2}(f)(y)\varphi  _\eta \partial _{\eta _j}[\eta ]\in R_G^{\mu -1}(U\times \R^q,{\bf{g}})_{\mathcal{O}}$$
(the latter observation has nothing to do with the asymptotic nature of $f$). Thus, it remains to verify that $r^{-(\mu -1)}g_2(y,\eta )\partial _{\eta _j}[\eta ]$ is a Green symbol of order $\mu -1$. Because of $a(\eta )\partial _{\eta _j}[\eta ]\in S^0_{\textup{cl}}(\R^q)$ and Proposition \ref{24grop} (ii) we may ignore $a(\eta ).$ Moreover, write $r^{-(\mu -1)}=[\eta ]^{\mu -1}(r[\eta ])^{-(\mu -1)}$ and $[\eta ]^{\mu -1}\in S^{\mu -1}_{\textup{cl}}(\R^q).$ Again by virtue of Proposition \ref{24grop} (ii)  it suffices to show that $g_3(y,\eta ):=r([\eta ])^{-(\mu -1)}g_2(y,\eta )$ is a Green symbol of order $0,$ more precisely,
\begin{equation}\label{24dff3}
\begin{split}
g_3(y,\eta )u(r)&=(r[\eta ])^{-(\mu -1)}\omega_\eta(r)\!\int_0^\infty\!\!\Big\{\!\int_ {C_{\gamma ,\beta}}\!\!\!(r[\eta]/r'[\eta])^{-z}\!f(y,z)\dbar z\Big\} u(r')\varphi _\eta(r')dr'/r'\\&
=\int_0^\infty \!\omega_\eta(r)\langle \zeta^1_{b^1}(y,\eta,r'[\eta]),(r[\eta])^{-z}\rangle u(r'))(r')^n dr'+g_4(y,\eta ),
\end{split}
\end{equation}
in the notation of the first expression of \eqref{21anP1} (we now suppressed the variables $x,x'$ that are involved via an integration with a kernel in $ C^\infty (X_x\times X_{x'});$ such an abbreviation is contained in \eqref{24dff2} anyway). $g_4(y,\eta )$ will be a flat remainder of a similar meaning as the difference in the first relation of \eqref{22gad5}. According to \eqref{21anP1}, \eqref{22gad3}, we have to recognise that 
$$\zeta^1_{b^1}(y,\eta,r')\in C^\infty (U,\mathcal{A}'(K^1_{b^1}E^1))^\bullet,$$
here for $E^1:= S^\nu_{\textup{cl}}(\R^q,C^\infty(X_x)\hat{\otimes}_\pi\mathcal{K}^{\infty,-\gamma;\infty}(X^\wedge_{r',x'})),\nu =n+1,$ cf. the formula \eqref{21dia4}. By reformulating the expression in the middle of \eqref{24dff3} we obtain
\begin{equation}\label{24dff4}
\begin{split}
g_3(y,\eta )=\omega_\eta(r)\int_0^\infty\Big\{\int_ {C_{\gamma ,\beta}}\varphi _\eta(r')[\eta ]^{n+1}(r'[\eta])^{z-n-1}f(y,z)(r&[\eta])^{-z-\mu +1}\dbar z\Big\}\\& u(r')(r')^n dr'/r',
\end{split}
\end{equation}
which yields after substituting $z=w-\mu +1$
\begin{equation}\label{24dff5}
\begin{split}
g_3(y,\eta )=\omega_\eta(r)\int_0^\infty\Big\{\int_ {C_{\gamma -\mu +1,\beta}}\varphi _\eta(r')[\eta ]^{n+1}&(r'[\eta])^{w-\mu -n}f(y,w-\mu +1)\\&(r[\eta])^{-w}\dbar w\Big\} u(r')(r')^n dr'/r'.
\end{split}
\end{equation}
The curve $C_{\gamma -\mu +1,\beta}$ surrounds the poles of $(r'[\eta ])^{w-\mu -n}f(y,w-\mu +1)$ for all $y\in U,$ and the integration against a holomorphic function in $w$ represents a family of analytic functionals carried by a compact set in the strip $\{(n+1)/2-(\gamma -\mu +1+\beta )<\textup{Re}\,w<(n+1)/2-(\gamma -\mu +1)\},$ pointwise discrete and of finite order. Recall that $\beta >0$ is chosen as large as we want, but for the connection with a variable discrete asymptotic type we only need the poles of real part $>c$ for some $c\in \{(n+1)/2-(\gamma -(k+1))<\textup{Re}\,w<(n+1)/2-\gamma \},\Theta =(-(k+1),0].$ Analogously as Proposition \ref{21merdec} we find a decomposition $f=f_0+f_1$ such that $f_0$ is holomorphic in $\{\textup{Re}\,w > c-\varepsilon _0\}$ and $f_1$ in $\{\textup{Re}\,w < c-\varepsilon _1\}$ for some small $0<\varepsilon _0<\varepsilon _1.$ This gives us a decomposition of the right hand side of \eqref{24dff5}. Because of the position of the poles in corresponding half-planes we may replace the curves for the integrals with $f_i$ by curves $C_i$ surrounding the poles of $f_i,i=0,1,$ such that $C_0$ is contained in 
$\textup{Re}\,w<c-\varepsilon _0'$ and $C_1$ in $\textup{Re}\,w>c-\varepsilon _1'$ for certain $0<\varepsilon _0'<\varepsilon _0,0<\varepsilon _1\varepsilon _1',c-\varepsilon _1'>(n+1)/2-(\gamma -(k+1)).$
Then we may set
\begin{equation}\label{24dff6}
\begin{split}
g_4(y,\eta )=\omega_\eta(r)\int_0^\infty\Big\{\int_ {C_0}\varphi _\eta(r')[\eta ]^{n+1}&(r'[\eta])^{w-\mu -n}f_0(y,w-\mu +1)\\&(r[\eta])^{-w}\dbar w\Big\} u(r')(r')^n dr'/r',
\end{split}
\end{equation}
and define $ \zeta^1_{b^1}$ by
\begin{equation}\label{24dff7}
\begin{split}
\langle \zeta^1_{b^1}(y,\eta,&r'[\eta]),(r[\eta])^{-z}\rangle\\&:=\int_ {C_1}\varphi _\eta(r')[\eta ]^{n+1}(r'[\eta])^{w-\mu -n}f_1(y,w-\mu +1)(r[\eta])^{-w}\dbar w.
\end{split}
\end{equation}
To finish the proof it remains to note that for the formal adjoint of $g_1(y,\eta )$ we have $ \zeta^2_{b^2}=0$ in the characterisation \eqref{22gad5}, since $\varphi \in C_0^\infty (\R_+).$ Moreover, the formal adjoint of the second summand in \eqref{24dff00} is of the same structure as the first one.
\end{proof}
\begin{prop} \label{22melgreen1}
Let $f_{j\alpha}(y,z)\in C^\infty(\Omega,\mathcal{M}_{\mathcal{R}_{j\alpha}}^{-\infty}(X))$ for variable discrete Mellin asymptotic types $\mathcal{R}_{j\alpha}, |\alpha|\leq j,j=0,\ldots,k,$ form the operator functions \eqref{21mmmel}, and
\begin{equation} \label{21mmmel5}
\widetilde{m}(y,\eta):=\sum_{j=0}^k\sum_{|\alpha|\leq j}\tilde{\omega}_\eta r^{-\mu+j}\widetilde{\Op}_M^{\gamma-n/2} (f_{j\alpha})(y)\eta^\alpha\tilde{\omega}'_\eta
\end{equation}
where $\widetilde{\Op}_M^{\gamma-n/2}(\cdot)$ indicates the corresponding operators for another choice of data \eqref{21melconv}, \eqref{21melconv0n}, and $\tilde{\omega},\tilde{\omega}'$ are other cut-off functions. Then 
\begin{equation} \label{21mmmel6}
g(y,\eta):= m(y,\eta)-\widetilde{m}(y,\eta)
\end{equation}
is a Green symbol of order $\mu$ in the sense of Definition \textup{\ref{21gr1}}.
\end{prop}
\begin{proof}
The notation in \eqref{22melgreen1} refers to \eqref{21mmmel}, \eqref{21melop3}. The assertion can be reduced to a decomposition of the operators with respect to a joint refinement of the open coverings and to a common partition of unity. Then we may compare the summands separately. Thus, if $f\in C_0^\infty (U,\mathcal{M}^
{-\infty }_{\mathcal{R}}(X))$ is given for an open set $U$, we consider two operators 
\begin{equation} \label{21mvergl}
m(y,\eta)\!:=\omega _\eta r^{-\mu +j}\op_M^{\gamma_j -n/2}\!(f)(y)\eta ^\alpha \omega' _\eta ,\,\,\tilde{m}(y,\eta)\!:=\tilde{\omega} _\eta r^{-\mu +j}\op_M^{\tilde{\gamma}_j -n/2}\!(f)(y)\eta ^\alpha\tilde{\omega}' _\eta 
\end{equation}
for different cut-off functions $\omega, \omega',\tilde{ \omega}, \tilde{\omega}' ,|\alpha |\leq j,$ and weights $\gamma -j\leq \gamma _j\leq \gamma ,\gamma -j\leq \tilde{\gamma}_j\leq \gamma ,$ such that $\pi _\C\mathcal{R}(y)\cap\Gamma _{(n+1)/2-\gamma _j}=\emptyset,\pi _\C\mathcal{R}(y)\cap\Gamma _{(n+1)/2-\tilde{\gamma} _j}=\emptyset$ for all $y\in U.$ By virtue of the latter assumptions on the weights it makes sense to form $\op_M^{\gamma_j -n/2}(f)(y)$ as well as $\op_M^{\tilde{\gamma}_j -n/2}(f)(y),$ and we first consider the difference
\begin{equation} \label{21mg1}
g(y,\eta):=\omega _\eta r^{-\mu +j}\{\op_M^{\tilde{\gamma}_j -n/2}\!(f)(y)-\op_M^{\gamma_j -n/2}(f)(y)\}\eta ^\alpha \omega' _\eta.
\end{equation}
To characterise $g(y,\eta)$ as a Green symbol for abbreviation we set $\delta :=\gamma_j$ and assume without loss of generality  $\tilde{\gamma}_j =\delta +\beta $ for some $\beta >0.$ Analogously as \eqref{21g} we obtain
\begin{equation} \label{21mg2}
g(y,\eta)=\omega _\eta r^{-\mu +j}\Big\{\int_ {C_{\delta,\beta}}(r[\eta]/r'[\eta])^{-z}f(y,z)\dbar z\Big\}\eta ^\alpha \omega' _\eta.
\end{equation}
Now similarly as in the preceding proof it follows that $g(y,\eta )$ is a Green symbol. It remains to discuss the effect under changing the cut-off functions, i.e. to consider
\begin{equation} \label{21mg3}
g_0(y,\eta):=(\omega _\eta -\tilde{ \omega}_\eta)r^{-\mu +j}\op_M^{\gamma_j -n/2}\!(f)(y)\eta ^\alpha \omega' _\eta
\end{equation}
and
\begin{equation} \label{21mg4}
g_1(y,\eta):=\omega _\eta r^{-\mu +j}\op_M^{\gamma_j -n/2}\!(f)(y)\eta ^\alpha (\omega' _\eta-\tilde{\omega}'_\eta ).
\end{equation}
The formal adjoint of $g_0$ is of analogous nature as $g_1$ and vice versa; therefore, it suffices to consider $g_1.$ The operators  $g^*_1$ contain the factor $\varphi _\eta :=\omega' _\eta-\tilde{\omega}'_\eta $ from the left where $\varphi \in C_0^\infty (\R_+).$ Thus, $\zeta^2_{b^2}=0,$ in the notation of \eqref{21anP1} applied to $g_1.$ Concerning the identification of $\zeta^1_{b^1}$ we may apply once again the trick of commuting powers $r^\beta $ for suitable $\beta >0$ from the right to the left through the Mellin operator after replacing $\varphi (r)$ by $r^\beta r^{-\beta}\varphi (r),$ using that $r^{-\beta}\varphi (r)\in C_0^\infty (\R_+).$ Because of the $y$-wise discrete character of $\mathcal{R}$ and since the set $U$ may be taken of a diameter as small as we want, we can choose $\beta $ in such a way that the resulting operator obtained after the commutation process is flat in the sense of symbols as in the first relation of \eqref{22gad5}. The computation leaves again a remainder of a similar structure as what we discussed before, see once again \eqref{21g}. Thus, this remainder gives us $\zeta^1_{b^1}$ of the desired quality.
\end{proof}

\section{Branching asymptotics}

\subsection{Weighted edge spaces with asymptotics}
As noted at the beginning solutions to elliptic (edge-degenerate) equations on a manifold with edge belong to weighted spaces of the edge calculus, cf. \cite{Schu32}, or \cite{Schu20}. Recall that locally near the edge of dimension $q$ in the variables $(r,x,y)$ of the (open stretched) wedge $X^\wedge\times \R^q$ they are modelled on the spaces $\mathcal{W}^s(\R^q,\mathcal{K}^{s,\gamma}(X^\wedge)).$ If $H$ is a Hilbert space with group action $\kappa=\{\kappa_\lambda\}_{\lambda\in\R_+}$ the space $\mathcal{W}^s(\R^q,H)$ is the completion of $\mathcal{S}(\R^q,H)$ with respect to the norm $\|\langle\eta\rangle^s\kappa_{\langle\eta\rangle}^{-1} \hat{u}(\eta)\|_{L^2(\R^q,H)}.$ Here $\hat{u}(\eta)=F_{y\rightarrow\eta}u(\eta)$ is the Fourier transform of $u$ in $\R^q.$ Equivalently we may (and will) replace $\langle\eta\rangle$ by $[\eta].$ The group action on $\mathcal{K}^{s,\gamma}(X^\wedge)$ is defined by \eqref{21kappa}. From the definition we see that there is an isomorphism
\begin{equation} \label{31pot}
K:=F^{-1}\kappa_{[\eta]}F:H^s(\R^q,H)\rightarrow\mathcal{W}^s(\R^q,H)
\end{equation}
for every $s\in\R$ where $H^s(\R^q,H)$ is the standard Sobolev space of $H$-valued distributions of smoothness $s$ on $\R^q$ (which is equal to the respective $\mathcal{W}^s$-space when we take the trivial group action on $H$). The spaces $\mathcal{W}^s  $ also admit $\textup{loc}$- and $\textup{comp}$-variants over an open set $\Omega \subseteq \R^q,$ denoted by $\mathcal{W}^s_{\textup{loc}}(\Omega ,H)$ and $\mathcal{W}^s_{\textup{comp}}(\Omega ,H),$ respectively. $\mathcal{W}^s_{\textup{loc}}(\Omega ,H)$ means the space of all distributions $u$ such that $\varphi u\in \mathcal{W}^s(\R^q,H),$ while $\mathcal{W}^s_{\textup{comp}}(\Omega ,H)$ is the subspace of compactly supported elements of $\mathcal{W}^s_{\textup{loc}}(\Omega ,H).$ A similar notation is used when $H$ is replaced by a Fr\'echet space.

\begin{rem}\label{31locw}
Although the definition of $\mathcal{W}^s(\R^q,\mathcal{K}^{s,\gamma}(X^\wedge))$ is anisotropic insofar it treats the direction of the edge $\R^q$ in a different manner than $X^\wedge,$ we have
\begin{equation} \label{31locww}
H_{\textup{comp}}^s(\R^q\times X^\wedge)\subset \mathcal{W}^s(\R^q,\mathcal{K}^{s,\gamma}(X^\wedge))\subset H_{\textup{loc}}^s(\R^q\times X^\wedge)
\end{equation}
for every $s,\gamma \in\R. $ This is a consequence of the identity
\begin{equation} \label{31locwww}
 \mathcal{W}^s(\R^q,H^s(\R^{n+1}))=H^s(\R^q \times \R^{n+1})
\end{equation}
where $H^s(\R^{n+1})$ is endowed with the group action $(\kappa _\lambda u)(\tilde{x})=\lambda ^{(n+1)/2}u(\lambda \tilde{x}),\lambda \in \R_+,$ cf. \cite{Schu20}.
\end{rem}
\begin{rem}\label{31potK}
For every $u(r,y)\in \mathcal{W}^s(\R^q,\mathcal{K}^{s,\gamma }(X^\wedge))$ (with $x\in X$ being suppressed in this notation) there exists a unique $v(r,y)\in H^s(\R^q,\mathcal{K}^{s,\gamma }(X^\wedge))$ such that
\begin{equation} \label{31potKK}
(F_{y'\rightarrow \eta }u)(r,\eta )=[\eta ]^{(n+1)/2}\hat{v}(r[\eta ],\eta ).
\end{equation} 
\end{rem}
This is an immediate consequence of relation \eqref{31pot}
The constructions also apply to a Fr\'echet space $E=\projlim_{j\in\N}E^j$ endowed with a group action (cf. the notation in connection with \eqref{21fre}); then we set
\begin{equation} \label{31fre}
\mathcal{W}^s(\R^q,E):= \lim_{\substack{\longleftarrow\\j\in\N}}\mathcal{W}^s(\R^q,E^j). 
\end{equation}
In particular, if $E:=\mathcal{K}_{\mathcal{P}}^{s,\gamma}(X^\wedge)$ for a constant discrete asymptotic type $\mathcal{P},$ cf. Section 2.1, we have a corresponding edge space with such asymptotics, namely,
\begin{equation} \label{31asd}
\mathcal{W}^s(\R^q,\mathcal{K}_{\mathcal{P}}^{s,\gamma}(X^\wedge)).
\end{equation}
Such a definition shows us immediately that we have a problem when we admit the asymptotic type $\mathcal{P}$ to be $y$-dependent. Therefore, the main issue of this section is to give a well-motivated notion of variable discrete asymptotics in edge spaces of finite smoothness $s.$ The case $\textup{dim}\,X=0$ has been treated in \cite{Schu34}. It will be instructive to consider once again the constant discrete case.\\
Let us first have a look at the standard Sobolev space $H^s(\R^{n+1}\times \R^q)$ where $\R^q$ is interpreted as an edge embedded in $\R^{n+1}\times \R^q.$ The transversal cone to the edge is $\R^{n+1}$ which can be identified via polar coordinates with $(S^n)^\Delta:=(\overline{\R}_+\times S^n )/(\{0\}\times S^n).$ In this case, taking the group action $\kappa_\lambda:u(\tilde{x})\rightarrow\lambda^{(n+1)/2}u(\lambda\tilde{x}),\lambda\in\R_+,$ in $H^s(\R^{n+1}),$ we have a canonical identification 
\begin{equation} \label{31sob}
H^s(\R^{n+1}\times \R^q)=\mathcal{W}^s(\R^q,H^s(\R^{n+1})).
\end{equation}
For any $s\geq0,s-(n+1)/2\notin\N $ fixed, we have a direct decomposition
\begin{equation} \label{31sdec}
H^s(\R^{n+1})=H^s_0(\R^{n+1})+\mathcal{E}_{\mathcal{T}}
\end{equation}
where
\begin{equation} \label{31sdect}
\mathcal{E}_{\mathcal{T}}:= \Big\{\omega(|\tilde{x}|)\sum_{|\alpha|< s-(n+1)/2}c_\alpha\tilde{x}^\alpha:c_\alpha\in\C,|\alpha|< s-(n+1)/2\Big\},
\end{equation}
and $H_0^s(\R^{n+1}):= \{u\in H^s(\R^{n+1}):D_{\tilde{x}}^\alpha u(0)=0 \quad\mbox{for all}\quad|\alpha|< s-(n+1)/2\}.$ The space $\mathcal{E}_T$ can also be written as 
\begin{equation} \label{31sdece}
\mathcal{E}_{\mathcal{T}}:= \Big\{\omega(r)\sum_{0\leq j<s-(n+1)/2}c_j(x)r^j:c_j\in L_j,0\leq j< s-(n+1)/2\Big\},
\end{equation}
for certain well-defined finite-dimensional subspaces $L_j$ of $C^\infty(S^n).$ According to \eqref{31sob} we can pass to the direct decomposition
\begin{equation} \label{31sdecw}
H^s(\R^{n+1}\times \R^q)=\mathcal{W}^s(\R^q,H^s_0(\R^{n+1}))+KH^s(\R^q,\mathcal{E}_{\mathcal{T}}),
\end{equation}
for the operator $K$ of \eqref{31pot}. The second term of \eqref{31sdecw} shows the form of the singular functions of the Taylor asymptotics of functions in $H^s(\R^{n+1}\times \R^q)$ transversally to $\R^q,$ namely,
\begin{equation} \label{31tay}
\begin{split}
KH^s(\R^q,\mathcal{E}_{\mathcal{T}})=F^{-1}_{\eta\rightarrow y}&\Big\{[\eta]^{(n+1)/2}\sum_{j<s-(n+1)/2}\omega(r[\eta])c_j(x)(r[\eta])^j\hat{v}_j(\eta):\\
&c_j\in L_j, v_j\in H^s(\R^q),0\leq j< s-(n+1)/2\Big\}.
\end{split}
\end{equation}
(Of course, in this case we can also write $KH^s(\R^q,\mathcal{E}_{\mathcal{T}})=F^{-1}_{\eta\rightarrow y}\big\{[\eta]^{(n+1)/2}\omega(|\tilde{x}|[\eta])\newline\sum_{|\alpha|< s-(n+1)/2}\hat{c}_\alpha(\eta)([\eta]\tilde{x})^\alpha:c_\alpha\in H^s(\R^q)\big\}$.)
In any case, every $u(r,x,y)\in H^s(\R^{n+1}\times \R^q)$ has the form
$u(r,x,y)=u_{\textup{flat}}(r,x, y)+u_{\textup{sing}}(r,x,y),$ for $u_{\textup{flat}}(r,x,y)\newline\in\mathcal{W}^s(\R^q\!,\!H^s_0(\R^{n+1})),u_{\textup{sing}}(r,x,y)\!\in\! KH^s(\R^q\!,\mathcal{E}_{\mathcal{T}}).$ More details on the relationship between standard Sobolev spaces and embedded submanifolds interpreted as edges may be found in \cite{Dine4}, \cite{Liu3}, and also in \cite{Haru13}, with applications to mixed elliptic problems.\\
Using the decomposition \eqref{1asK}, for the space \eqref{31asd} we obtain 
\begin{equation} \label{31asdecw}
\mathcal{W}^s(\R^q,\mathcal{K}_{\mathcal{P}}^{s,\gamma}(X^\wedge))=\mathcal{W}^s(\R^q,\mathcal{K}_\Theta^{s,\gamma}(X^\wedge))+KH^s(\R^q,\mathcal{E}_{\mathcal{P}}),
\end{equation}
applying the operator $K:H^s(\R^q,\mathcal{E}_{\mathcal{P}})\rightarrow \mathcal{W}^s(\R^q,\mathcal{K}^{\infty ,\gamma}(X^\wedge)).$ Thus, every $u(r,x,y)\newline\in \mathcal{W}^s(\R^q,\mathcal{K}_{\mathcal{P}}^{s,\gamma}(X^\wedge))$ can be written as $u(r,x,y)=u_{\textup{flat}}(r,x,y)+u_{\textup{sing}}(r,x,y),$ for $u_{\textup{flat}}(r,x,y)\in\mathcal{W}^s(\R^q,\mathcal{K}_\Theta^{s,\gamma}(X^\wedge)),$ $u_{\textup{sing}}(r,x,y)\in KH^s(\R^q,\mathcal{E}_{\mathcal{P}}). $ In the present case the space of singular functions with constant discrete asymptotics is given by
\begin{equation} \label{31tayK}
\begin{split}
KH^s(\R^q,\mathcal{E}_{\mathcal{P}})=F^{-1}_{\eta\rightarrow y}\Big\{[\eta]^{(n+1)/2}&\sum_{j=0}^J\sum_{l=0}^{m_j}\omega(r[\eta])c_{jl}(x,\eta)(r[\eta])^{-p_j}\textup{log}^l(r[\eta]):\\&c_{jl}(x,\eta)\in C^\infty(X,\hat{H}^s(\R_\eta^q))\quad\mbox{for all}\quad j,l\Big\},
\end{split}
\end{equation} 
cf. the equation \eqref{1sing}, where $\hat{H}^s(\R_\eta^q):=F_{y\rightarrow \eta }H^s(\R^q_y).$ It is clear that when we change the cut-off function, or replace $[\eta]$ by another function of that kind, or by $\langle\eta\rangle,$ and denote the resulting decomposition by $u(r,x,y)=\tilde{u}_{\textup{flat}}(r,x,y)+\tilde{u}_{\textup{sing}}(r,x,y),$ then $u_{\textup{sing}}-\tilde{u}_{\textup{sing}}= \tilde{u}_{\textup{flat}}-u_{\textup{flat}}\in\mathcal{W}^s(\R^q,\mathcal{K}_\Theta^{s,\gamma}(X^\wedge)).$ Moreover, we have $\mathcal{W}^\infty(\R^q,H)=H^\infty(\R^q,H)$ for any Hilbert (or Fr\'echet) space
with group action. Thus, in particular, 
\begin{equation} \label{31asinf}
\begin{split}
\mathcal{W}^\infty(\R^q,\mathcal{K}_{\mathcal{P}}^{\infty,\gamma}(X^\wedge))&=H^\infty(\R^q,\mathcal{K}_{\mathcal{P}}^{\infty,\gamma}(X^\wedge))\\
&=H^\infty(\R^q,\mathcal{K}_\Theta^{\infty,\gamma}(X^\wedge))+H^\infty(\R^q,\mathcal{E}_{\mathcal{P}}).
\end{split}
\end{equation}
In order to find a formulation for the variable discrete asymptotics in weighted edge spaces we first observe that the space \eqref{31tayK} can be written in the form
\begin{equation} \label{31tayF}
\begin{split}
KH^s(\R^q,\mathcal{E}_{\mathcal{P}})=F^{-1}_{\eta\rightarrow y}\big\{[\eta]^{(n+1)/2}&\omega(r[\eta])\langle\hat{\zeta}(\eta)_z,(r[\eta])^{-z}\rangle:\\&\hat{\zeta}(\eta)\in\mathcal{A}'_{\mathcal{P}}(\pi_\C\mathcal{P}, C^\infty(X,\hat{H}^s(\R_\eta^q)))\big\},
\end{split}
\end{equation} 
where $\pi_\C\mathcal{P}=\{p_j\}_{j=0,\ldots,J},$ and $\mathcal{A}'_{\mathcal{P}}(\pi_\C\mathcal{P},E)$ for a Fr\'echet space $E$ is the space of all $E$-valued analytic functionals, carried by $\pi_\C\mathcal{P},$ of the form
\begin{equation} \label{31asana}
\langle\zeta,h\rangle=\sum_{j=0}^J\sum_{l=0}^{m_j}c_{jl}\frac{d^l}{dz^l}h(z)|_{z=p_j}\quad\mbox{for any}\quad c_{jl}\in E.
\end{equation}

\begin{rem}\label{31aspot}
Let 
\begin{equation} \label{31as}
\mathcal{P}:=\{(p_j,m_j)\}_{j=0,\ldots ,J}\subset \C \times \N,\quad J=J(\mathcal{P})<\infty ,
\end{equation}
be a sequence such that $\pi _\C \mathcal{P}= \{p_j\}_{j=0,\ldots ,J}\subset \{z\in \C:(n+1)/2-\gamma +\vartheta <\textup{Re}\,z <(n+1)/2-\gamma \},$ and let $\mathcal{A}_{\mathcal{P}} $ defined to be the space of all scalar meromorphic functions with poles at the points $p_j$ of multiplicity $m_j+1,j=0,\ldots ,J.$ Moreover, let $C$ be a smooth compact curve counter-clockwise surrounding $\pi _\C\mathcal{P}$ in the usual way. Every $f\in \mathcal{A}_{\mathcal{P}}$ gives rise to an element $\zeta _f\in \mathcal{A}'(\pi_\C\mathcal{P})$ via $h\mapsto \int_C f(z)h(z)\dbar z, h\in \mathcal{A}(\C).$ Then
\begin{equation} \label{31asfu}
k_{\zeta _f}(\eta )c:=[\eta]^{(n+1)/2}\omega(r[\eta])\langle\zeta_f,(r[\eta])^{-z}\rangle c,\,\,c\in C^\infty (X),
\end{equation}
represents a symbol $k_{\zeta _f}(\eta )\in S^0_{\textup{cl}}(\R^q;C^\infty (X),\mathcal{K}^{\infty ,\gamma }(X^\wedge))_{\textup{id},\kappa }$ for the trivial group action $\textup{id}$ on $C^\infty (X)$ and $(\kappa _\lambda u)(\lambda r,x)=\lambda ^{(n+1)/2}u(\lambda r,x),\lambda \in \R_+,$ on the space $\mathcal{K}^{\infty ,\gamma }(X^\wedge).$ The associated pseudo-differential operator 
\begin{equation} \label{31asfop}
\Op(k_{\zeta _f})=F^{-1}k_{\zeta _f}(\eta )F
\end{equation}
is continuous as an operator
\begin{equation} \label{31asfopL}
\Op(k_{\zeta _f}):H^s(\R^q,C^\infty (X))\rightarrow \mathcal{W}^s(\R^q,\mathcal{K}^{\infty ,\gamma }(X^\wedge)),
\end{equation}
and we have
\begin{equation} \label{31asfopK}
KH^s(\R^q,\mathcal{E}_{\mathcal{P}})=\{\Op(k_{\zeta _f})v:v\in H^s(\R^q,C^\infty (X)),f\in \mathcal{A}_{\mathcal{P}}\}.
\end{equation}
\end{rem}
Setting for the moment 
$$\mathcal{A}'_{\mathcal{P}}:=\{\zeta _f:f\in \mathcal{A}_{\mathcal{P}}\}$$
we easily see that $\mathcal{A}'_{\mathcal{P}}$ is of finite dimension. Moreover, for any fixed cut-off function $\omega $  the transformation $\zeta _f\mapsto \omega (r)\langle\zeta _f,r^{-z}\rangle$ gives us an isomorphism $\mathcal{A}'_{\mathcal{P}}\rightarrow \mathcal{A}^1_{\mathcal{P}}$
to a finite-dimensional subspace $\mathcal{A}^1_{\mathcal{P}}$ of $\mathcal{A}_{\mathcal{P}},$ and we have an isomorphism
$$KH^s(\R^q,\mathcal{E}_{\mathcal{P}})\cong \mathcal{A}^1_{\mathcal{P}}\otimes H^s(\R^q,C^\infty (X)).$$
Let $\mathcal{P}$ be a variable discrete asymptotic type over $\Omega,$ associated with the weight data $(\gamma,\Theta).$ Then we define $C^\infty(\Omega,\mathcal{K}_{\mathcal{P}}^{\infty,\gamma}(X^\wedge))$ to be the set of all $u(r,x,y)\in C^\infty(\Omega_y,\mathcal{K}^{\infty,\gamma}(X^\wedge_{r,x}))$ such that for any cut-off function $\omega$ the Mellin transform $M_{r\rightarrow z}\omega(r)u(r,x,y)=:f(y,z)$ is a 
(in $y$ smooth) family of meromorphic functions in the strip $\{(n+1)/2-\gamma+\vartheta<\textup{Re}\,z<(n+1)/2-\gamma\},$ subordinate to $\mathcal{P},$ as explained after Definition \ref{21melas} (we often suppress the variable $x$ when we speak about $C^\infty(X)$-valued functions).
\begin{defn} \label{31wedgeas}
Let $\mathcal{P}$ be a variable discrete asymptotic type over $\Omega,$ associated with the weight data $(\gamma,\Theta).$ The space $\mathcal{W}^s_{\textup{loc}}(\Omega,\mathcal{K}_{\mathcal{P}}^{s,\gamma}(X^\wedge))$ for $s\in\R$ is defined to be the set of all $u\in\mathcal{W}^s_{\textup{loc}}(\Omega,\mathcal{K}^{s,\gamma}(X^\wedge))$ such that for $b:=(c,U)$ for any $(n+1)/2-\gamma+\vartheta<c<(n+1)/2-\gamma$ and $U\in\mathcal{U}(\Omega)$ there exists a compact set $K_b\subset\{(n+1)/2-\gamma+\vartheta<\textup{Re}\,z<(n+1)/2-\gamma\}$ and a function 
\begin{equation} \label{31asef}
\hat{f}_b(y,z,\eta)\in C^\infty(U,\mathcal{A}(\C\setminus K_b,E^s))^\bullet\quad\mbox{for}\quad E^s:=C^\infty(X,\hat{H}^s(\R^q_\eta))
\end{equation}
subordinate to $\mathcal{P}|_U$ and a corresponding
\begin{equation} \label{31ased}
\hat{\delta}_b(y,\eta)\in C^\infty(U,\mathcal{A}'(K_b,E^s))^\bullet, \langle\hat{\delta}_b(y,\eta),h\rangle=\int_{C_b}\hat{f}_b(y,z,\eta)h(z)\dbar z,
\end{equation}
$h\in\mathcal{\C}$, with $C_b$ counter-clockwise surrounding $K_b,$ such that 
\begin{equation} \label{31asinwg}
\begin{split}
u(r,x,y)-F^{-1}_{\eta\rightarrow y}\big\{[\eta]^{(n+1)/2}&\omega(r[\eta])\langle\hat{\delta}_b(y,\eta),(r[\eta])^{-z}\rangle\big\}\\
&\in\mathcal{W}^s_{\textup{loc}}(U,\mathcal{K}^{s,\gamma+\beta}(X^\wedge))
\end{split}
\end{equation}
for $\beta =\beta _0+\varepsilon $ for any $0<\varepsilon <\varepsilon (b),$ and $\beta_0=(n+1)/2-\gamma-c,$ cf. also the relation \eqref{21dia2}. Moreover, we set $$\mathcal{W}^s_{\textup{comp}}(\Omega,\mathcal{K}_{\mathcal{P}}^{s,\gamma}(X^\wedge)):=\mathcal{W}^s_{\textup{loc}}(\Omega,\mathcal{K}_{\mathcal{P}}^{s,\gamma}(X^\wedge))\cap\mathcal{W}^s_{\textup{comp}}(\Omega,\mathcal{K}^{s,\gamma}(X^\wedge)).$$
\end{defn}
In other words, the space $\mathcal{W}^s_{\textup{loc}}(\Omega,\mathcal{K}_{\mathcal{P}}^{s,\gamma}(X^\wedge))$ is the set of all $\tilde{u}\in \mathcal{W}^s_{\textup{loc}}(\Omega,\mathcal{K}^{s,\gamma}\newline(X^\wedge))$ such that (in the notation of Definition \ref{31wedgeas}) for every $b=(c,U)$ the restriction $u:=\tilde{u}|_U$ belongs to 
\begin{equation} \label{32assum}
\mathcal{W}^s_{\textup{loc}}(U,\mathcal{K}^{s ,\gamma +\beta}(X^\wedge)) +  \mathcal{W}^s_{b,\mathcal{P}}                 
\end{equation}
for 
\begin{equation} \label{32contras}
\begin{split}
  &\mathcal{W}^s_{b,\mathcal{P}} :=\big\{F^{-1}_{\eta\rightarrow y}\{[\eta]^{(n+1)/2}\omega(r[\eta])\langle\hat{\delta}_b(y,\eta),(r[\eta])^{-z}\rangle\}:\\ &\hat{\delta}_b(y,\eta)\,\,\mbox{belonging to an}\,\,\hat{f}_b(y,z,\eta)\,\,\mbox{in the sense of \eqref{31ased} }\\ &\mbox{and}\,\,\hat{f}_b(y,z,\eta) \,\,\mbox{as in \eqref{31asef} subordinate to}\,\,\mathcal{P}_b:=p_{K_b}\mathcal{P}|_U\big\}. 
\end{split}               
\end{equation}
The notation $\mathcal{W}^s_{b,\mathcal{P}}$ is only an abbreviation for $\mathcal{W}^s_{b,\mathcal{P}_b}.$\begin{rem}\label{32fr}
The space $\mathcal{W}^s_{b,\mathcal{P}}$ is Fr\'echet. Thus, \eqref{32assum} is Fr\'echet as well in the topology of the non-direct sum.
\end{rem}
In fact, this can be verified by similar arguments as for Proposition \ref{21Fre}, although the structure here is slightly more complicated. 
\begin{rem}\label{32frsppot}
The function in \eqref{31asef} can also be interpreted as an element of $C^\infty (U,\mathcal{A}_{\mathcal{P}}(\C\setminus K_b))^\bullet \hat{\otimes}_\pi E^s,E^s=C^\infty (X,\hat{H}^s(\R^q_\eta )).$ Here $C^\infty (U,\mathcal{A}_{\mathcal{P}}(\C\setminus K_b))^\bullet$ means the space of all scalar smooth functions in $U$ with values in $\mathcal{A}(\C\setminus K_b)$ that $y$-wise extend to meromorphic functions across $K_b$ subordinate to $\mathcal{P};$ this space is Fr\'echet in a natural way. Similarly as Remark \ref{31aspot} the space $\mathcal{W}^s_{b,\mathcal{P}}$ is associated with potential symbols, namely, 
\begin{equation} \label{32asepott}
k_{\zeta  _f}(y,\eta ) c:=  [\eta]^{(n+1)/2}\omega(r[\eta])\langle\zeta _f(y),(r[\eta])^{-z}\rangle c,\,\,c\in C^\infty (X),              
\end{equation}
$k_{\zeta  _f} (y,\eta )\in S^0_{\textup{cl}}(U\times \R^q;C^\infty (X),\mathcal{K}^{\infty ,\gamma }(X^\wedge))_{\textup{id},\kappa }.$ This gives us a continuous operator
\begin{equation}\label{32edfrepot}
\Op(k_{\zeta  _f} ):H^s(\R^q,C^\infty (X))\rightarrow \mathcal{W}^s_{\textup{loc}}(U,\mathcal{K}^{\infty ,\gamma }(X^\wedge)),
\end{equation}
and we have a tensor product expansion
\begin{equation}\label{32edfrepotas}
\begin{split}
\mathcal{W}^s_{b,\mathcal{P}}=\Big\{&\sum_{j=0}^\infty \lambda _j\Op(k_{\zeta  _{f_j}} )v_j:\sum_{j=0}^\infty |\lambda _j|<\infty ,\\&f_j\in C^\infty (U,\mathcal{A}_{\mathcal{P}}(\C\setminus K_b))^\bullet,v_j\in H^s(\R^q,C^\infty (X)),\\&f_j\rightarrow 0,v_j\rightarrow 0 \,\,\mbox{in the respective spaces, as}\,\,j\rightarrow \infty \Big\}.
\end{split}
\end{equation}
\end{rem}
\begin{prop}\label{32edfre}
There are continuous embeddings
\begin{equation} \label{32asembd}
\mathcal{W}^{s'}_{\textup{loc}}(\Omega,\mathcal{K}_{\mathcal{P}}^{s',\gamma}(X^\wedge)) \hookrightarrow  \mathcal{W}^s_{\textup{loc}}(\Omega,\mathcal{K}_{\mathcal{P}}^{s,\gamma}(X^\wedge))                   
\end{equation}
for every $s'\geq s.$
\end{prop}
\begin{proof}
Analogously as \eqref{32assum} we consider
\begin{equation} \label{32assumco}
\mathcal{W}^s_{\textup{comp}}(U,\mathcal{K}^{s ,\gamma +\beta}(X^\wedge)) +  \mathcal{W}^s_{b,\mathcal{P}}                 
\end{equation}
and tacitly assume in this case (to have a short notation for the second summand) that the respective function $\hat{\delta}_b(y,\eta)$ (or, equivalently $\hat{f}_b(y,z,\eta)$) has compact support in $U$ with respect to $y.$
Now for fixed $b=(c,U)$ the space of all $\tilde{u}$ mentioned at the beginning of the proof defines a Fr\'echet subspace of $\mathcal{W}^s_{\textup{loc}}(\Omega,\mathcal{K}^{s,\gamma}(X^\wedge)),$ and then $\mathcal{W}^s_{\textup{loc}}(\Omega,\mathcal{K}_{\mathcal{P}}^{s,\gamma}(X^\wedge))$ itself is Fr\'echet in the topology of the projective limit of those spaces over $c$ and then over $U.$ The continuous embeddings \eqref{32asembd} easily follow from the respective continuous embeddings $\mathcal{W}^{s'}_{\textup{loc}}(\Omega,\mathcal{K}^{s',\gamma}(X^\wedge))\hookrightarrow\mathcal{W}^s_{\textup{loc}}(\Omega,\mathcal{K}^{s,\gamma}(X^\wedge)),$  $\mathcal{W}^{s'}_{\textup{loc}}(U,\mathcal{K}^{s' ,\gamma +\beta}(X^\wedge))\hookrightarrow\mathcal{W}^s_{\textup{loc}}(U,\mathcal{K}^{s ,\gamma +\beta}(X^\wedge))$ and $\mathcal{W}^{s'}_{b,\mathcal{P}}\hookrightarrow\mathcal{W}^s_{b,\mathcal{P}}$.
\end{proof}
If $E$ and $\tilde{E}$ are Fr\'echet spaces with (say, monotonic) semi-norm systems $(\pi _j)_{j\in \N}$ and $(\tilde{\pi} _j)_{j\in \N},$ respectively, an operator $A:E\rightarrow \tilde{E}$ is continuous if for every $j\in \N$ there is a $k=k(j)\in \N$ such that $\tilde{\pi} _j(Au)\leq c\pi _k(u)$ for all $u\in E,$ for constants $c=c(A;j,k)>0.$ By $A\rightarrow 0$ in $\mathcal{L}(E,\tilde{E})$ we mean that for some choice of a map $j\mapsto k(j)$ such that the above estimates hold we have $c(A;j,k(j))\rightarrow 0$ for all $j$.
\begin{rem}\label{32lomult}
The multiplication by $b\in C^\infty (\Omega )$ induces a continuous operator $$b:\mathcal{W}^s_{\textup{loc}}(\Omega,\mathcal{K}_{\mathcal{P}}^{s,\gamma}(X^\wedge))\rightarrow \mathcal{W}^s_{\textup{loc}}(\Omega,\mathcal{K}_{\mathcal{P}}^{s,\gamma}(X^\wedge))$$ for every variable discrete  asymptotic type $\mathcal{P}$ over $\Omega ,$ and $b\rightarrow 0$ in $C^\infty (\Omega )$ entails $b\rightarrow 0$ in $\mathcal{L}(\mathcal{W}^s_{\textup{loc}}(\Omega,\mathcal{K}_{\mathcal{P}}^{s,\gamma}(X^\wedge))).$
\end{rem}
\begin{rem}\label{32locc}
Let $(U_\iota)_{\iota \in I} $ be a covering of $\Omega $ by open subsets. Then $u\in \mathcal{W}^s_{\textup{loc}}(\Omega,\mathcal{K}_{\mathcal{P}}^{s,\gamma}\newline (X^\wedge))$ is equivalent to $u|_{U_\iota }\in \mathcal{W}^s_{\textup{loc}}(U_\iota ,\mathcal{K}_{\mathcal{P|_{U_\iota }}}^{s,\gamma}(X^\wedge))$ for every $\iota \in I.$ Moreover, if the covering is locally finite and $(\varphi _\iota)_{\iota \in I} $ a subordinate partition of unity we have $u=\sum_{\iota \in I}\varphi _\iota u$ where $\varphi _\iota u\in \mathcal{W}^s_{\textup{loc}}(U_\iota ,\mathcal{K}_{\mathcal{P|_{U_\iota }}}^{s,\gamma}(X^\wedge)).$ In particular, in order to characterise elements of $\mathcal{W}^s_{\textup{loc}}(\Omega,\mathcal{K}_{\mathcal{P}}^{s,\gamma}(X^\wedge))$ it suffices to consider $u$ in open subsets of $\Omega $ of arbitrarily small diameter.
\end{rem}
If
\begin{equation} \label{32adj}
C:\mathcal{W}^s_{\textup{comp}}(\Omega ,\mathcal{K}^{s ,\gamma}(X^\wedge)) \rightarrow \mathcal{W}^{s-m}_{\textup{loc}}(\Omega ,\mathcal{K}^{s-m ,\gamma -\mu }(X^\wedge))
\end{equation}
is continuous for every $s\in \R$ we have the formal adjoint $C^*:\mathcal{W}^s_{\textup{comp}}(\Omega ,\mathcal{K}^{s ,-\gamma+\mu }\newline(X^\wedge)) \rightarrow \mathcal{W}^{s-m}_{\textup{loc}}(\Omega ,\mathcal{K}^{s-m ,-\gamma  }(X^\wedge))$ from a corresponding sesquilinear pairing between the involved spaces with respect to the $\mathcal{W}^0(\Omega ,\mathcal{K}^{0,0}(X^\wedge))$-scalar product ($\mathcal{W}^0(\Omega ,\cdot):=\mathcal{W}^0(\R^q ,\cdot)|_\Omega ).$
\begin{defn}\label{32smothing}
An operator $C:\mathcal{W}^s_{\textup{comp}}(\Omega ,\mathcal{K}^{s ,\gamma}(X^\wedge)) \rightarrow \mathcal{W}^\infty _{\textup{loc}}(\Omega ,\mathcal{K}^{\infty  ,\gamma -\mu }(X^\wedge))$ which is continuous for all $s$ is called smoothing in the edge algebra with variable discrete asymptotics if it induces continuous operators
\begin{equation} \label{32adjsm}
\begin{split}
&C:\mathcal{W}^s_{\textup{comp}}(\Omega ,\mathcal{K}^{s ,\gamma}(X^\wedge)) \rightarrow \mathcal{W}^\infty _{\textup{loc}}(\Omega ,\mathcal{K}_{\mathcal{P}}^{\infty  ,\gamma -\mu }(X^\wedge)),\\&C^*:\mathcal{W}^s_{\textup{comp}}(\Omega ,\mathcal{K}^{s ,-\gamma+\mu }(X^\wedge)) \rightarrow \mathcal{W}^\infty _{\textup{loc}}(\Omega ,\mathcal{K}_{\mathcal{Q}}^{\infty  ,-\gamma -\mu }(X^\wedge))
\end{split}
\end{equation}
for all $s\in \R,$ for ($C$-dependent) variable discrete asymptotic types $\mathcal{P},\mathcal{Q}.$
\end{defn}
\begin{defn}\label{32mplusg}
An operator $A=\Op_y(m+g)+C$ for $m+g\in R^\mu _{M+G}(\Omega \times \R^q,{\bf{g}}),{\bf{g}}=(\gamma ,\gamma -\mu ,\Theta ),$ and a smoothing operator $C$ in the sense of Definition \ref{32smothing} is called a smoothing Mellin plus Green operator in the edge calculus with variable discrete asymptotics. If the Mellin summand vanishes we talk about a Green operator in this calculus.
\end{defn}
\begin{prop}\label{31grree}
Let $m+g\in R^\mu _{M+G}(\Omega \times \R^q,{\bf{g}}),{\bf{g}}=(\gamma ,\gamma -\mu ,\Theta ),$ and $\varphi ,\varphi '\in C^\infty (\Omega ), \textup{supp}\,\varphi \cap\textup{supp}\,\varphi '=\emptyset.$ Then the operator $C:=\varphi \Op_y(m+g)\varphi '$ is smoothing in the sense of Definition \textup{\ref{32smothing}}.
\end{prop}
\begin{proof}
Let us check the first mapping property of \eqref{32adjsm}; the second one is of analogous structure when we apply the conclusions to the symbol $\varphi '(y)(m+g)^*(y',\eta )\varphi (y')$ for the pointwise formal adjoint $(m+g)^*$ which has similar properties, cf. Proposition \ref{21smad} and relation \eqref{22grnadas}. We apply the Taylor formula to $\varphi '(y')$ and write 
\begin{equation}
\begin{split}
\varphi '(y') = & \sum_{|\alpha |\leq N}1/\alpha !(y'-y)^\alpha \partial _y^\alpha \varphi '(y)\\&+(N+1)!\sum_{|\alpha |= N+1} 1/\alpha !(y'-y)^\alpha \int_0^\infty (1-t)^N\partial _y^\alpha \varphi '(y+t(y'-y))dt.
\end{split}
\end{equation}
Setting $\psi _\alpha (y,y')=(N+1)!/\alpha !\int_0^\infty (1-t)^N\partial _y^\alpha \varphi '(y+t(y'-y))dt$ it follows that
\begin{equation} \label{32arest}
\begin{split}
Cu(y)&=\sum_{|\alpha |= N+1}\int\!\!\int e^{i(y-y')\eta }(m+g)(y,\eta )(y'-y)^\alpha \psi _\alpha (y,y')u(y')dy'\dbar\eta \\&=\sum_{|\alpha |=N+1}\int\!\!\int (-D_\eta )^\alpha e^{i(y-y')\eta }(m+g)(y,\eta )\psi _\alpha (y,y')u(y')dy'\dbar\eta\\&=\sum_{|\alpha |=N+1}\int\!\!\int  e^{i(y-y')\eta }D_\eta ^\alpha(m+g)(y,\eta )\psi _\alpha (y,y')u(y')dy'\dbar\eta.
\end{split}
\end{equation}
By virtue of Propositions \ref{24grop} and \ref{24grdiff} we have $g_\alpha (y,\eta ):=D_\eta ^\alpha(m+g)(y,\eta )\in R^{\mu -|\alpha| }_G(\Omega \times \R^q,{\bf{g}})$ for $N>0.$ We consider the summands $C_\alpha $ on the right of the latter equation separately for any fixed $\alpha $. A tensor product expansion for $\psi _\alpha (y,y')\in C^\infty (\Omega \times \Omega )$ gives us $\psi _\alpha (y,y')=\sum_{j=0}^\infty \lambda _jb_j(y)b'_j(y')$ for coefficients $\lambda _j\in \C,\sum_{j=0}^\infty |\lambda _j|<\infty $ and $b_j(y),b'_j(y')\in C^\infty (\Omega ),$ tending to $0$ as $j\rightarrow \infty .$ This gives us
\begin{equation}\label{32areco}
C_\alpha u(y)=\sum_{j=0}^\infty \lambda _jb_j(y)\Op_y(g_\alpha )b'_j(y')u(y').
\end{equation}
Let us now apply Theorem \ref{32gr} below (the proof does not depend on the present proposition).
Then using the known fact that $b'_ju$ tends to zero in the subspace of $\mathcal{W}^s_{\textup{comp}}(\Omega,\mathcal{K}^{s,\gamma}(X^\wedge))$ of elements with fixed compact support in $\Omega $ (namely, $\textup{supp}\,u$), Theorem \ref{32gr} yields $\Op_y(g_\alpha )b'_ju\rightarrow \!0$ in $\mathcal{W}^{s-(\mu+|\alpha |)}_{\textup{loc}}\!(\Omega,\mathcal{K}_{\mathcal{P}}^{\infty,\gamma-(\mu+|\alpha |)}(X^\wedge))$ for some variable discrete asymptotic type $\mathcal{P}.$ Remark \ref{32lomult} shows that \eqref{32areco} converges in $\mathcal{W}^{s-(\mu+|\alpha |)}_{\textup{loc}}(\Omega,\mathcal{K}_{\mathcal{P}}^{\infty,\gamma-(\mu+|\alpha |)}(X^\wedge)).$ A closed graph argument (first applied to elements with fixed compact support in $\Omega $) gives us the continuity of 
$$C:\mathcal{W}^s_{\textup{comp}}(\Omega,\mathcal{K}^{s,\gamma}(X^\wedge))\rightarrow \mathcal{W}^{s-(\mu+|\alpha |)}_{\textup{loc}}\!(\Omega,\mathcal{K}_{\mathcal{P}}^{\infty,\gamma-(\mu+|\alpha |)}(X^\wedge)).$$
Since $\alpha $ is arbitrary, we obtain the desired mapping property of $C.$
\end{proof}

\subsection{The action of operators in spaces with variable discrete asymptotics}
An important aspect of the program of regularity is that the operators of the edge calculus 
preserve variable discrete asymptotics. The main issue is to have operators  $$\mathcal{W}^s_{\textup{comp}}(\Omega,\mathcal{K}_{(\mathcal{P})}^{s,\gamma}(X^\wedge))\rightarrow\mathcal{W}^{s-\mu}_{\textup{loc}}(\Omega,\mathcal{K}_{(\mathcal{Q})}^{s-\mu,\gamma-\mu}(X^\wedge))$$
where $\mathcal{P}$ or $\mathcal{Q}$ in parentheses mean the corresponding spaces with or without such asymptotic types.  \\
We employ the fact that when $a(y,y', \eta)$ is a symbol in $S^m_{(\textup{cl})}(\Omega\times\Omega\times\R^q;H,\tilde{H})$ for (Hilbert or Fr\'echet) spaces $H,\tilde{H}$ with group action the operator $\Op_y(a)u(y)=\int\!\!\int e^{i(y-y')\eta}a(y,y', \eta)u(y')dy'\dbar\eta$ induces continuous operators 
\begin{equation} \label{32co}                       \Op(a):\mathcal{W}^s_{\textup{comp}}(\Omega,H)\rightarrow\mathcal{W}^{s-m}_{\textup{loc}}(\Omega,\tilde{H})\quad\mbox{for any}\quad s\in\R.
\end{equation}
\begin{thm} \label{32gr}
Let $g(y,\eta)$ be a Green symbol in the sense of Definition \textup{\ref{21gr1}}. Then $$\Op_y(g):\mathcal{W}^s_{\textup{comp}}(\Omega,\mathcal{K}^{s,\gamma}(X^\wedge))\rightarrow\mathcal{W}^{s-m}_{\textup{loc}}(\Omega,\mathcal{K}^{\infty,\gamma-\mu}(X^\wedge))$$  induces a continuous operator
\begin{equation} \label{32asef}
\Op_y(g):\mathcal{W}^s_{\textup{comp}}(\Omega,\mathcal{K}^{s,\gamma}(X^\wedge))\rightarrow\mathcal{W}^{s-m}_{\textup{loc}}(\Omega,\mathcal{K}_{\mathcal{Q}}^{\infty,\gamma-\mu}(X^\wedge))
\end{equation}
for some variable discrete asymptotic type $\mathcal{Q}$ over $\Omega,$ for any $s\in\R.$
\end{thm}
\begin{proof}
The continuity of $\Op_y(g)$ between edge spaces without asymptotics follows from relation \eqref{22g1}. For the continuity of \eqref{32asef} it suffices to show that for every $b=(c,U),(n+1)/2-(\gamma -\mu )+\vartheta <c<(n+1)/2-(\gamma -\mu ),U\in \mathcal{U}(\Omega ),$ the operator is continuous as
\begin{equation} \label{32rU}
\Op_y(g):\mathcal{W}^s_{\textup{comp}}(U,\mathcal{K}^{s,\gamma}(X^\wedge))\rightarrow\mathcal{W}^{s-m}_{\textup{loc}}(U,\mathcal{K}^{\infty,\gamma-\mu+\beta }(X^\wedge))+\mathcal{W}^{s-m}_{b,\mathcal{Q}}
\end{equation}
for $\beta =\beta _0+\varepsilon ,0<\varepsilon <\varepsilon (b),\beta _0=(n+1)/2-(\gamma-\mu)-c.$ According to \eqref{22gad5} we can write 
\begin{equation} \label{32rUdec}
g(y,\eta) =g_0(y,\eta)+ g_{b,\mathcal{Q}}(y,\eta)
\end{equation}
for a $g_0(y,\eta)\in S^m_{\textup{cl}}(U\times \R^q;\mathcal{K}^{s,\gamma ;g}(X^\wedge),\mathcal{K}^{\infty ,\gamma -\mu +\beta ;\infty }(X^\wedge))$ and
\begin{equation} \label{32asant}
g_{b,\mathcal{Q}}(y,\eta)u(r,x)=\int\!\!\int_0^\infty \omega_\eta(r)\langle \zeta(y,\eta,r'[\eta],x,x'),(r[\eta])^{-z}\rangle u(r',x'))(r')^n dr'dx',
\end{equation}
for a $\zeta (y,\eta ;r',x,x')(=\zeta _b)\in C^\infty (U,\mathcal{A}'(K_b,L))^\bullet,K_b\Subset \{(n+1)/2-(\gamma-\mu)+\vartheta<\textup{Re}\,z<(n+1)/2-(\gamma-\mu)\},$ described by a family $f (=f _b)$ of $L$-valued meromorphic functions subordinate to $Q,$ and $L:= S^\nu_{\textup{cl}}(\R^q)\hat{\otimes}_\pi C^\infty(X_x)\hat{\otimes}_\pi\mathcal{K}^{\infty,-(\gamma-\mu );\infty}(X^\wedge_{r',x'})),$ $\nu=m+n+1.$ 
The operator $\Op(g_0)$ induces continuous maps
\begin{equation} \label{32asg0st}
\Op(g_0):\mathcal{W}^s_{\textup{comp}}(U,\mathcal{K}^{s,\gamma}(X^\wedge))\rightarrow\mathcal{W}^{s-m}_{\textup{loc}}(U,\mathcal{K}^{\infty,\gamma-\mu+\beta }(X^\wedge))
\end{equation}
for all $s.$ Moreover, a tensor product expansion gives us 
$$\zeta (y,\eta ;r',x,x')=\sum_{l=0}^\infty \lambda _l \zeta _l (y)a_l(\eta )t_l(r',x')$$
for $\lambda _l\in \C,\sum|\lambda _l|<\infty ,\zeta _l (y)\in \mathcal{A}'(K_b)\hat{\otimes}_\pi C^\infty (U)\hat{\otimes}_\pi C^\infty (X), a_l(\eta )\in S^{m +n+1}_{\textup{cl}}(\R^q_\eta ),\newline t_l(r',x')\in \mathcal{K}^{s,-\gamma ;-g}(X^\wedge),$  tending to 0 in the respective spaces for $l\rightarrow \infty .$ Then
\begin{equation} \label{32asanttr}
T_l(\eta)u=a_l(\eta )\int_X\!\!\int_0^\infty [\eta ]^{(n+1)/2}t_l(r'[\eta ],x') u(r',x')(r')^n dr'dx',
\end{equation}
defines a $T_l(\eta )\in S^{m +(n+1)/2}_{\textup{cl}}(\R^q;\mathcal{K}^{s,\gamma ;g}(X^\wedge),\C).$ For $u\in \mathcal{W}^s(\R^q,\mathcal{K}^{s,\gamma ;g}(X^\wedge))$ it follows that $w_l :=\Op_y (T_l)u\in H^{s-\nu -(n+1)/2}(\R^q).$ Write $\hat{w_l}(\eta )=[\eta ]^{(n+1)/2}\hat{w}_l^0(\eta )$ for a corresponding $w_l^0\in H^{s-m }(\R^q).$ Moreover, for
\begin{equation} \label{32asapot}
K_l(y,\eta)u=\omega (r[\eta ])\langle\zeta _l(y),(r[\eta ])^{-z}\rangle
\end{equation}
we have
\begin{equation} \label{32assuum}
g_{b,\mathcal{Q}}(y,\eta )=\sum _{l=0}^\infty \lambda _l K_l(y,\eta)T_l(\eta ).
\end{equation}
It follows that 
\begin{equation} \label{32asgrop}
\begin{split}
\Op_y(g_{b,\mathcal{Q}})u&=\sum _{l=0}^\infty \lambda _l \Op_y(K_lT_l)u\\&=\sum _{l=0}^\infty \lambda _l F_{\eta \rightarrow y}^{-1}K_l(y,\eta )(F\Op_y(T_l)u)\\&=\sum _{l=0}^\infty \lambda _l F_{\eta \rightarrow y}^{-1}[\eta ]^{(n+1)/2}\omega (r[\eta ])\langle\zeta _l (y),(r[\eta ])^{-z}\rangle \hat{w}^0_l(\eta ).
\end{split}
\end{equation}
The series 
\begin{equation} \label{32asdelta}
\delta _b(y,\eta ):=\sum_{l=0}^\infty \lambda _l \zeta _l (y) \hat{w}^0_l(\eta )
\end{equation}
is convergent in $C^\infty (U,\mathcal{A}'(K_b,E^{s-m }))^\bullet$ subordinate to $\mathcal{Q}$ (cf. the notation in \eqref{31asef}). It
follows that $\Op_y(g_{b,\mathcal{Q}})u=F^{-1}_{\eta\rightarrow y}\big\{[\eta]^{(n+1)/2}\omega(r[\eta])\langle\hat{\delta}_b(y,\eta),(r[\eta])^{-z}\rangle\big\}$ which belongs to $\mathcal{W}^{s-m}_{b,\mathcal{Q}}.$ In other words we proved that 
\begin{equation} \label{32asdeltas}
\Op_y(g_{b,\mathcal{Q}}):\mathcal{W}^s_{\textup{comp}}(U,\mathcal{K}^{s,\gamma}(X^\wedge))\rightarrow \mathcal{W}^{s-m}_{b,\mathcal{Q}}
\end{equation}
is a linear map. By virtue of the continuity of \eqref{32asg0st} for the continuity of \eqref{32rU} it remains to show that \eqref{32asdeltas} is continuous. This follows from the closed graph theorem when we have a closed graph. However, this is the case, since $\Op_y(g_{b,\mathcal{Q}})$ is continuous as an operator into an analogue of the space $\mathcal{W}^{s-m}_{b,\mathcal{Q}}$ based on continuous asymptotics carried by the compact set $K_b.$ Then, if a sequence of pairs $(u_j,v_j)_{j\in \N}$ in the graph of $\Op_y(g_{b,\mathcal{Q}})$ converges in the latter sense to a limit $(u,v)$ in the product space where the second component relies on the space with continuous asymptotics, we have $\Op_y(g_{b,\mathcal{Q}})u=v.$ But $v_j$ belongs to $\mathcal{W}^{s-m}_{b,\mathcal{Q}}$ in the pointwise discrete sense, and the limit in the continuous analogue is automatically pointwise discrete. This completes the proof of Theorem \ref{32gr}.
\end{proof}

\begin{thm} \label{32mel}
Let $m(y,\eta)$ be a smoothing Mellin symbol of the edge calculus of the form \eqref{21mmmel}. Then \begin{equation}\label{32melconw}
\Op_y(m):\mathcal{W}^s_{\textup{comp}}(\Omega,\mathcal{K}^{s,\gamma}(X^\wedge))\rightarrow\mathcal{W}^{s-\mu}_{\textup{loc}}(\Omega,\mathcal{K}^{\infty,\gamma-\mu}(X^\wedge))
\end{equation}  
restricts to a continuous operator
\begin{equation} \label{32asem}
\Op_y(m):\mathcal{W}^s_{\textup{comp}}(\Omega,\mathcal{K}_{\mathcal{P}}^{s,\gamma}(X^\wedge))\rightarrow\mathcal{W}^{s-\mu}_{\textup{loc}}(\Omega,\mathcal{K}_{\mathcal{Q}}^{\infty,\gamma-\mu}(X^\wedge))
\end{equation}
for every variable discrete asymptotic type $\mathcal{P}$ for some resulting variable discrete asymptotic type $\mathcal{Q},$ both over $\Omega,$ for any $s\in\R.$
\end{thm} 
\begin{proof}
An operator \eqref{21mmmel} is a sum of expressions of the form
\begin{equation} \label{32smel}
r^{-\mu +j}\omega _\eta \op_M^{\gamma _{j,\alpha }-n/2} f(y)\eta ^\alpha \omega '_\eta 
\end{equation}
where $j\in \N,|\alpha |\leq j,\gamma -j \leq \gamma _{j,\alpha } \leq \gamma ,$ and $f(y)\in C^\infty (\Omega ,\mathcal{M}^{-\infty }_{\mathcal{R}}(X))$ for some variable discrete Mellin asymptotic type $\mathcal{R}=\mathcal{R}_{j,\alpha }$ and $f$ supported with respect to $y$ in an open subset of small diameter, such that $\Gamma _{(n+1)/2-\gamma _{j,\alpha }}\cap \pi _\C \mathcal{R}_{j,\alpha }(y)=\emptyset$ for all $y\in U.$ By virtue of Proposition \eqref{21smmel} the operator \eqref{32melconw} is continuous for every $s.$ To show the continuity of \eqref{32asem} we mainly check that $\Op_y(m)$ defines a linear map between the corresponding spaces. The continuity is then a consequence of the closed graph theorem which applies for similar reasons as in the proof of Theorem \ref{32gr}. By Remark \ref{32locc} it suffices to assume that the argument function $u\in \mathcal{W}^s_{\textup{comp}}(\Omega,\mathcal{K}_{\mathcal{P}}^{s,\gamma}(X^\wedge))$ is supported with respect to $y$ in an open set $U\subseteq \Omega $ of sufficiently small diameter. We may (and will)  take the same $U$ as before for the Mellin symbol $f$ since the localisation of a Mellin operator off the diagonal gives rise to a Green operator, cf. Proposition \ref{31grree}. More precisely we assume $f\in C_0^\infty (U,\mathcal{M}_{\mathcal{R}}(X)).$ For $b:=(c,U)$ and $\tilde{b}:=(\tilde{c},U),$ 
$$(n+1)/2-\gamma +\vartheta <c<(n+1)/2-\gamma,(n+1)/2-(\gamma -\mu )+\vartheta <\tilde{c}<(n+1)/2-(\gamma -\mu )$$
according to \eqref{32assum} for our mapping we consider the spaces 
\begin{equation} \label{32spacedec}
\mathcal{W}^s_{\textup{comp}}(U,\mathcal{K}^{s ,\gamma +\beta}(X^\wedge)) +  \mathcal{W}^s_{b,\mathcal{P}},\,\mathcal{W}^{s-\mu }_{\textup{loc}}(U,\mathcal{K}^{s ,\gamma -\mu +\tilde{\beta}}(X^\wedge)) +  \mathcal{W}^{s-\mu }_{\tilde{b},\mathcal{Q}} 
\end{equation}
for $\beta =\beta _0+\varepsilon ,\beta _0=(n+1)/2-\gamma -c,\tilde{\beta}=\tilde{\beta}_0+\tilde{\varepsilon },\tilde{\beta}_0=(n+1)/2-(\gamma -\mu)-\tilde{c},$ for sufficiently small $\varepsilon ,\tilde{\varepsilon }>0$ (by definition we have $\beta_0 >0,$ and $\gamma +\beta $ indicates flatness of order $\beta $ relative to $\gamma ;$ a similar role play $\tilde{\beta },\tilde{\beta }_0$). It suffices to show that for every $\tilde{c}$ there is a $c$ such that 
\begin{equation} \label{32smelcontU}
\begin{split}
\Op_y(r^{-\mu +j}\omega _\eta \op_M^{\gamma _{j,\alpha }-n/2} f(y)\eta ^\alpha \omega '_\eta ):&\mathcal{W}^s_{\textup{comp}}(U,\mathcal{K}^{s ,\gamma +\beta}(X^\wedge)) +  \mathcal{W}^s_{b,\mathcal{P}}\\ &\rightarrow \mathcal{W}^{s-\mu }_{\textup{loc}}(U,\mathcal{K}^{s ,\gamma -\mu +\tilde{\beta}}(X^\wedge)) +  \mathcal{W}^{s-\mu }_{\tilde{b},\mathcal{Q}}
\end{split}
\end{equation}
is continuous. Let us mainly consider the case $\mu =0,$ and $j=0;$ then $\gamma _{j,\alpha }=\gamma .$  The general case can be treated by a slight modification of arguments.  By assumption we have $f\in C^\infty (\Omega ,\mathcal{M}^{-\infty }_{\mathcal{R}}(X))$ and $\Gamma _{(n+1)/2-\gamma }\cap\pi _\C \mathcal{R}(y)=\emptyset$ for all $y\in \Omega .$ Then the assertion can be decomposed into the mapping properties 
\begin{equation} \label{32smeltU1}
\Op_y(\omega _\eta \op_M^{\gamma -n/2} f(y) \omega '_\eta )\!:\!\mathcal{W}^s_{\textup{comp}}(U,\mathcal{K}^{s ,\gamma +\beta}(X^\wedge)) \rightarrow \!\mathcal{W}^s_{\textup{loc}}(U,\mathcal{K}^{s ,\gamma+\tilde{\beta}}(X^\wedge)) +  \mathcal{W}^s_{\tilde{b},\mathcal{L}},
\end{equation}
and
\begin{equation} \label{32smeltU2}
\Op_y(\omega _\eta \op_M^{\gamma -n/2} f(y) \omega '_\eta ):  \mathcal{W}^s_{b,\mathcal{P}}\rightarrow \mathcal{W}^s_{\textup{loc}}(U,\mathcal{K}^{s ,\gamma+\tilde{\beta}}(X^\wedge)) +  \mathcal{W}^s_{\tilde{b},\mathcal{M}}
\end{equation}
for variable discrete asymptotic types $\mathcal{L},\mathcal{M}.$ We will see that 
\begin{equation} \label{32smeltU3}
\mathcal{L}(y)=\{(r(y),l(y))\in \mathcal{R}(y):(n+1)/2-\gamma +\vartheta <\textup{Re}\,r(y)<(n+1)/2-\gamma\},
\end{equation}
and
\begin{equation} \label{32smeltV1}
\mathcal{M}(y)=\mathcal{P}(y)\cup\mathcal{L}(y),
\end{equation}
$y\in \Omega .$
Let us first consider the case \eqref{32smeltU1}. We choose $\tilde{b} =(\tilde{c},U)$ as mentioned before (now for $\mu =0$). Here we have some freedom, since the only condition is that all $U$ in consideration form a finite open covering of $\overline{U}_0$ for any given $U_0\in \mathcal{U}(\Omega );$  the diameter of $U$ may be taken as small as we want. Moreover, it suffices to assume that $\tilde{c}$ belongs to a countable sequence $(\tilde{c}_\iota )_{\iota \in \N}$ such that $\tilde{c}_\iota \rightarrow (n+1)/2-\gamma +\vartheta $ as $\iota \rightarrow \infty .$ For its choice we start with any such sequence $(\tilde{\tilde{c}}_\iota )_{\iota \in \N}$ and then by an approximation we pass to another one, namely,  $(\tilde{c}_\iota )_{\iota \in \N}$ with analogous properties. Let us write $\tilde{\tilde{c}}:= \tilde{\tilde{c}}_\iota $ for any fixed $\iota $. Given $\tilde{\tilde{c}}$ we first set $\tilde{\tilde{\beta }}_0:=(n+1)/2-\gamma -\tilde{\tilde{c}}.$ Then for every fixed $y_0\in \overline{U}_0$
we find a sufficiently small $\delta >0$ such that $\Gamma _{(n+1)/2-(\gamma+\tilde{\tilde{\beta }}_0+\delta)}\cap \pi _\C\mathcal{L}(y_0)=\emptyset.$ Now we set $\tilde{\beta }_0:=\tilde{\tilde{\beta }}_0+\delta $ for such a $\delta $ and $\tilde{c}:=(n+1)/2-(\gamma+\tilde{\beta }_0).$ Then we have 
\begin{equation} \label{32smappro}
\Gamma _{(n+1)/2-(\gamma+\tilde{\beta }_0+\tilde{\varepsilon })}\cap \pi _\C\mathcal{L}(y_0)=\emptyset
\end{equation}
first for $y=y_0$ and all $0<\tilde{\varepsilon  } \leq \tilde{\varepsilon  }_0$ for some sufficiently small $\tilde{\varepsilon  }_0>0,$ and then for all $y$ in an open neighbourhood $U(y_0)$ of $y_0.$ Then we set $U:=U(y_0).$ Summing up we made an appropriate choice of $\tilde{b} =(\tilde{c},U)$. 
For the proof of \eqref{32smeltU1} we set $\beta =\tilde{\beta },\beta _0=\tilde{\beta } _0, c=\tilde{c },$ such that $\beta =\beta _0+\varepsilon ,\beta _0=(n+1)/2-\gamma -c.$ Recall that $\mathcal{L}$ in this connection is used in the meaning
\begin{equation} \label{32smabb}
\mathcal{L}_b=p_{K_b}\mathcal{L}|_U,
\end{equation}
cf. the formula \eqref{32contras}. Let $u\in \mathcal{W}^s_{\textup{comp}}(U,\mathcal{K}^{s ,\gamma +\beta}(X^\wedge)).$ By Remark \ref{31potK} we have 
\begin{equation} \label{32spacedecv}
 (F_{y'\rightarrow \eta }u)(r,\eta )=[\eta ]^{(n+1)/2}\hat{v}(r[\eta ],\eta )\,\,\mbox{for a}
\,\,v(r,y)\in H^s(\R^q,\mathcal{K}^{s ,\gamma +\beta}(X^\wedge))
\end{equation}
(recall that the $x$-variable is suppressed in the notation). By virtue of Theorems \ref{22melgreen1} and \ref{32gr} we may assume $\omega =\omega '.$ We have
\begin{equation} \label{32spacedecvme}
 \omega (r[\eta ])(F_{y'\rightarrow \eta }u)(r,\eta )=[\eta ]^{(n+1)/2}\omega (r[\eta ])\hat{v}(r[\eta ],\eta )=\kappa _{[\eta ]}\{\omega (r)\hat{v}(r,\eta )\}.
\end{equation}
Since a Mellin operator commutes with $\kappa _{[\eta ]}$ when the Mellin symbol has constant coefficients in $r, $ cf. the relation \eqref{1mopcom}, for
\begin{equation} \label{32spaabvme}
 m(y,\eta ):=\omega _\eta \op_M^{\gamma -n/2} f(y) \omega _\eta
\end{equation}
 it follows that
\begin{equation} \label{32spme}
\begin{split}
 \Op_y(m)u(r,y)&=F^{-1}_{\eta \rightarrow y }\omega (r[\eta ])\op_M^{\gamma -n/2}(f)(y)\omega (r[\eta ])(F_{y'\rightarrow \eta }u)(r,\eta )\\&=F^{-1}_{\eta \rightarrow y }\omega (r[\eta ])\op_M^{\gamma -n/2}(f)(y)\kappa _{[\eta ]}\{\omega (r)\hat{v}(r,\eta )\}\\&=F^{-1}_{\eta \rightarrow y }\kappa _{[\eta ]}\{\omega (r)\op_M^{\gamma -n/2}(f)(y)\omega (r)\hat{v}(r,\eta )\}.
 \end{split}
\end{equation}
Let $\hat{b}(z,\eta ):=M_{\gamma -n/2}\{\omega (r)\hat{v}(r,\eta )\};$ then
$$\omega (r)\op_M^{\gamma -n/2}(f)(y)\{\omega (r)\hat{v}(r,\eta )\}=\omega (r)\int_{\Gamma_{(n+1)/2-\gamma } }r^{-z}f(y,z)\hat{b}(z,\eta )\dbar z.$$
For  $y\in U$ we can write 
\begin{equation} \label{32spdif}
\omega (r)\op_M^{\gamma -n/2}(f)(y)\{\omega (r)\hat{v}(r,\eta )\}=\hat{d}(r,y,\eta )+\hat{c}(r,y,\eta )
\end{equation}
for
\begin{equation} \label{32spdif1}
\begin{split}
\hat{d}(r,y,\eta )=&\,\omega (r)\int_{\Gamma_{(n+1)/2-\gamma } }r^{-z}f(y,z)\hat{b}(z,\eta )\dbar z\\&-\omega (r)\int_{\Gamma_{(n+1)/2-(\gamma +\beta)} }r^{-z}f(y,z)\hat{b}(z,\eta )\dbar z,
\end{split}
\end{equation}
and
\begin{equation} \label{32spdif2}
\hat{c}(r,y,\eta )=\omega (r)\int_{\Gamma_{(n+1)/2-(\gamma +\beta )} }r^{-z}f(y,z)\hat{b}(z,\eta )\dbar z.
\end{equation}
We have
\begin{equation} \label{32spdif3}
\hat{d}(r,y,\eta )=\omega (r)\int_C r^{-z}f(y,z)\hat{b}(z,\eta )\dbar z
\end{equation}
for a smooth compact curve $C$ counter-clockwise surrounding the poles of $f(y,z)$ in $\{(n+1)/2-(\gamma +\beta )<\textup{Re}\,z<(n+1)/2-\gamma\}$ for all $y\in U.$ Moreover, we write
\begin{equation} \label{32spdif4}
\begin{split}
\hat{c}(r,y,\eta )=\omega (r)\int_{-\infty }^\infty & r^{-((n+1)/2-(\gamma +\beta )+i\rho )}f(y,(n+1)/2-(\gamma +\beta )+i\rho )\\ &\hat{b}((n+1)/2-(\gamma +\beta )+i\rho,\eta )\dbar \rho \\
&=r^\beta \omega (r)\int_{\Gamma_{(n+1)/2-\gamma } } r^{-\tilde{z}}f(y,\tilde{z}-\beta  )\hat{b}(\tilde{z}-\beta ,\eta )\dbar \tilde{z}.
\end{split}
\end{equation}
It follows that $\omega (r)\op_M^{\gamma -n/2}(f)(y)\{\omega (r)\hat{v}(r,\eta )\}=\hat{d}(r,y,\eta )+\hat{c}(r,y,\eta )$ for all $y\in U,$ and \eqref{32spme} gives us
\begin{equation} \label{32spdif5}
\Op_y(m)u(r,y)=F^{-1}_{\eta \rightarrow y }\kappa _{[\eta ]}\{\big(\hat{d}(r,y,\eta )+\hat{c}(r,y,\eta )\big)\}.
\end{equation}
In order to recognise the nature of the latter expression we first consider the summand with $\hat{c}(r,y,\eta )$ and show that 
\begin{equation} \label{32spdif6}
F^{-1}_{\eta \rightarrow y }\kappa _{[\eta ]}\hat{c}(r,y,\eta )\in \mathcal{W}^s(\R^q,\mathcal{K}^{\infty ,\gamma +\beta }(X^\wedge)).
\end{equation}
To this end we employ the fact that 
$$f(y,\tilde{z}-\beta  )\in C_0^\infty (U,L^{-\infty }(X;\Gamma _{(n+1)/2-\gamma })).$$
Applying a tensor product expansion we write
$$f(y,\tilde{z}-\beta )|_{\Gamma_{(n+1)/2-\gamma } }=\sum_{j=0}^\infty \lambda _j\psi _j f_j$$
where $\sum_j|\lambda _j|<\infty ,$ and $\psi _j\in C_0^\infty (U),f_j\in L^{-\infty }(X;\Gamma _{(n+1)/2-\gamma })$ tending to $0$ in the respective spaces for $j\rightarrow \infty .$ This gives us
$$F^{-1}_{\eta \rightarrow y }\kappa _{[\eta ]}\hat{c}(r,y,\eta )=\sum_{j=0}^\infty F^{-1}_{\eta \rightarrow y }\kappa _{[\eta ]} \lambda _j\psi _j(y)\hat{c}_j(r,\eta )$$
for $\hat{c}_j(r,\eta )=r^\beta \omega (r)\int_{\Gamma_{(n+1)/2-\gamma } } r^{-\tilde{z}}f_j(\tilde{z}-\beta  )\hat{b}(\tilde{z}-\beta ,\eta )\dbar \tilde{z}.$\,We obtain
$$F^{-1}_{\eta \rightarrow y }\kappa _{[\eta ]} \psi _j(y)\hat{c}_j(r,\eta )=\psi _j(y)F^{-1}_{\eta \rightarrow y }\kappa _{[\eta ]}F_{y' \rightarrow \eta  }(c_j(r,y'))(\eta ),$$
where $\omega (r)c_j(r,y')\in H^s(\R^q,\mathcal{K}^{\infty ,\gamma +\beta }(X^\wedge)),$ i.e.
$$F^{-1}\kappa _{[\eta ]}F(c_j(r,y'))\in \mathcal{W}^s(\R^q,\mathcal{K}^{\infty ,\gamma +\beta }(X^\wedge))$$
which tends to $0$ as $j\rightarrow \infty .$ The multiplication by $\psi _j(y)$ acts on $\mathcal{W}^s(\R^q,\mathcal{K}^{\infty ,\gamma +\beta  }\newline (X^\wedge))$ as a continuous operator, and its norm tends to zero as $\psi _j(y) \rightarrow 0.$ This gives us altogether \eqref{32spdif6}, as desired.\\
Let us now characterise the contribution of $\hat{d}(r,y,\eta )$ in \eqref{32spdif5}. First note that 
$$\hat{b}(z,\eta )\in M_{\gamma -n/2,r\rightarrow z}F_{y'\rightarrow \eta }\omega (r)H^s(\R^q,\mathcal{K}^{\infty ,\gamma +\beta }(X^\wedge)).$$
The Mellin transform of $\omega (r)H^s(\R^q_{y'},\mathcal{K}^{\infty ,\gamma +\beta }(X^\wedge))$
consists of $H^s(\R^q_{y'},H^s(X))$-valued holomorphic functions in $\{\textup{Re}\,z>(n+1)/2-(\gamma +\beta )\}$ which restrict to elements in 
$$H^s(\R^q_{y'},\hat{H}^s(\Gamma _\alpha \times X))$$
for every $\alpha  >(n+1)/2-(\gamma +\beta ),$ uniformly in compact $\alpha  $-intervals. The product $f(y,z)\hat{b}(z,\eta )$ occurring in \eqref{32spdif3} is holomorphic in $z$ in a neighbourhood of the curve $C$ for all $y\in U$ and extends for every $y$ to a meromorphic function inside, with values in $H^s(\R^q_{y'},C^\infty (X)).$ Thus, \eqref{32spdif3} represents the pairing of a function $\hat{\delta }(y,\eta )\in C^\infty (U,\mathcal{A}'(K,C^\infty (X)\hat{\otimes}_\pi \hat{H}^s(\R^q_\eta )))^\bullet$ with $r^{-z}.$ This gives us a function $\delta (y,y')$ as in Definition \ref{31wedgeas}. Thus, the first summand on the right of \eqref{32spdif5} takes the form
\begin{equation} \label{32spdif7}
F^{-1}_{\eta \rightarrow y }\kappa _{[\eta ]}(\omega (r)\langle\delta (y,y'),r^{-z}\rangle)=F^{-1}_{\eta \rightarrow y }[\eta ]^{(n+1)/2}(\omega (r[\eta ])\langle\delta (y,y'),(r[\eta ])^{-z}\rangle).
\end{equation}
This yields finally
\begin{equation} \label{32spdif8}
\Op_y(m)u(r,y)\in \mathcal{W}^s_{\textup{loc}}(\Omega ,\mathcal{K}^{\infty ,\gamma }_{\mathcal{L}}(X^\wedge)).
\end{equation}
In order to show \eqref{32smeltU2} we assume $u\in \mathcal{W}^s_{b,\mathcal{P}}$ (recall that in this connection the function $u$ has compact support in $U$ with respect to $y,$ cf. the comment after \eqref{32assumco}). In other words, $\Op_y(m)$ acts on 
\begin{equation} \label{32auf}
u(r,y)= F^{-1}_{\eta\rightarrow y}\big\{[\eta]^{(n+1)/2}\omega(r[\eta])\langle\hat{\delta}_b(y,\eta),(r[\eta])^{-z}\rangle\}               
\end{equation}
for a $\hat{\delta}_b(y,\eta)\in C_0^\infty(U,\mathcal{A}'(K_b,E^s))^\bullet,$ cf. \eqref{31asef}, \eqref{31ased}, and \eqref{32contras}. Moreover, we assume again $f\in C_0^\infty (U, \mathcal{M}^{-\infty }_{\mathcal{R}}(X))$ which is admitted on the expense of a Green term, cf. Proposition \ref{31grree} and Theorem \ref{32gr}. Then we have to characterise
\begin{equation} \label{32amob}
\Op_y(\omega _\eta \op_M^{\gamma -n/2} (f)(y) \omega _\eta\{F^{-1}_{\tilde{\eta}\rightarrow y}\big\{[\tilde{\eta}]^{(n+1)/2}\omega_{\tilde{\eta}} )\langle\hat{\delta}_b(y,\tilde{\eta}),(r[\tilde{\eta}])^{-z}\rangle\}\}.               
\end{equation}
A tensor product expansion gives us
\begin{equation} \label{32atens}
\hat{\delta}_b(y,\eta )=\sum_{j=0}^\infty \lambda _j\zeta _j(y)\hat{v}_j(\eta ),               
\end{equation}
with $\sum_{j=0}^\infty |\lambda _j|<\infty $ and zero sequences $\zeta _j$ and $\hat{v}_j$ in the spaces $C_0^\infty (U,\mathcal{A}'(K_b,C^\infty\newline (X) ))^\bullet$ and $\hat{H}^s(\R^q),$ respectively (where $\zeta _j$ is defined by functions subordinate to $\mathcal{P}$). The potential symbols
\begin{equation} \label{32atepot}
k_{\zeta _j}(y',\eta )=\omega_{\tilde{\eta}}[\tilde{\eta}]^{(n+1)/2}\langle\zeta _j(y'),(r[\tilde{\eta}])^{-z}\rangle\in S^0_{\textup{cl}}(U\times \R^q;\C,\mathcal{K}^{\infty ,\gamma }(X^\wedge))              
\end{equation}
tend to zero in this space of symbols. From 
\begin{equation} \label{32atedev}
\begin{split}
u(r,y)=F^{-1}_{\tilde{\eta}\rightarrow y}&\big\{[\tilde{\eta}]^{(n+1)/2}\omega_{\tilde{\eta}} \langle\hat{\delta}_b(y,\tilde{\eta}),(r[\tilde{\eta}])^{-z}\rangle\}\hat{v}_j (\tilde{\eta})\\&=  \sum_{j=0}^\infty \lambda _j F^{-1}_{\tilde{\eta}\rightarrow y}\big\{[\tilde{\eta}]^{(n+1)/2}\omega_{\tilde{\eta}} \langle\zeta _j(y),(r[\tilde{\eta}])^{-z}\rangle \hat{v}_j (\tilde{\eta})\} \\&= \sum_{j=0}^\infty \lambda _j K_jv_j\,\,\,\mbox{for}\,\,\,K_j=\Op_y(k_{\zeta _j})
\end{split}       
\end{equation}it follows that 
\begin{equation} \label{32atepch}
\Op_y(m)u=\sum_{j=0}^\infty \lambda _j\Op_y\big(\omega _\eta \op_M^{\gamma -n/2} (f)(y) \omega _\eta\big)\Op_{y}\big([\eta]^{(n+1)/2}\omega_{\eta}k_{\zeta _j}(y,\eta )\big)v_j.
\end{equation}
To compute the right hand side we first consider the $j$-th summand, and drop for the moment the factor $\lambda _j.$ For 
\begin{equation} \label{32abc}
m(y,\eta )=\omega _\eta \op_M^{\gamma -n/2} (f)(y) \omega _\eta, \,\,n_j(y,\eta )=k_{\zeta _j}(y,\eta )
\end{equation}
we have 
\begin{equation} \label{32abcomp}
\Op_y\big(\omega _\eta \op_M^{\gamma -n/2}\! (f)(y) \omega _\eta\big)\Op_{y}\big(k_{\zeta _j}(y,\eta )\big)=\Op_y(m(y,\eta ))\Op_{y}( n_j(y,\eta )).
\end{equation}
According to the composition rule $\Op_y(m(y,\eta ))\Op_{y}( n_j(y,\eta ))=\sum_{|\alpha |\leq N}\Op_y\big(1/\alpha !\newline(\partial _\eta ^\alpha m(y,\eta ))D_\eta ^\alpha n_j(y,\eta )\big)+\Op_y\big(r_{N,j}(y,\eta )\big)$ for a remainder $r_{N,j}$ of standard form cf. the formula \eqref{32rem1} below, $N\in \N,$ we have to characterise 
\begin{equation} \label{32abcch}
\Op_y\big((\partial _\eta ^\alpha m(y,\eta ))D_y ^\alpha n_j(y,\eta )\big)v_j\,\,\mbox{and}\,\,\Op_y\big(r_{N,j}(y,\eta )\big)v_j
\end{equation}
for $v_j\in H^s(\R^q)$ as elements of $\mathcal{W}^s_{\textup{comp}}(U,\mathcal{K}^{\infty,\gamma  }_{\mathcal{
M}}(X^\wedge))$, tending to zero as $j\rightarrow \infty .$
It will be sufficient to consider the case $N=0.$ In the following computations we first omit $j$ and consider 
\begin{equation} \label{32abcch1}
\Op_y\!\big( m(y,\eta ) n(y,\eta )\!\big)v=\Op_y\!\big(\!(\omega _\eta \op_M^{\gamma -n/2} \! (f)(y) \omega _\eta)([\eta]^{(n+1)/2}\omega_{\eta}\langle\zeta _w(y),\!(r[\eta])^{-w}\rangle )\big)v.
\end{equation}
Similarly as \eqref{32spme} it follows that the latter expression is equal to
\begin{equation} \label{32abcch2}
F^{-1}_{\eta \rightarrow y}\{\kappa _{[\eta ]}\omega (r)\op_M^{\gamma -n/2} (f)(y)\tilde{\omega }(r)\langle\zeta _w(y),r^{-w}\rangle \hat{v}(\eta )\}
\end{equation}
for $\tilde{\omega }:=\omega ^2.$ Setting $b(y,z):=M_{\gamma -n/2,r\rightarrow z}\{\tilde{\omega }(r)\langle\zeta _w(y),r^{-w}\rangle \}$ and $\hat{b}(y,z,\eta ):=b(y,z)\hat{v}(\eta )$ it follows that
\begin{equation} \label{32abcch3}
\omega (r)\op_M^{\gamma -n/2} (f)(y)\tilde{\omega }(r)\langle\zeta _w(y),r^{-w}\rangle \hat{v}(\eta )=\omega (r)\int_{\Gamma _{(n+1)/2-\gamma }}r^{-z}f(y,z)\hat{b}(y,z,\eta )\dbar z.
\end{equation}
The following considerations make sense again first for a fixed $y_0\in \overline{U}$ such that the respective weight lines are free of poles of the integrands in $z$ and then in an open neighbourhood of that point. In other words, we may concentrate on a suitable $U$ of sufficiently small diameter. Then, similarly as \eqref{32spdif} for $y\in U$ we can write 
\begin{equation} \label{32spdi1}
\omega (r)\op_M^{\gamma -n/2} (f)(y)\tilde{\omega }(r)\langle\zeta _w(y),r^{-w}\rangle \hat{v}(\eta )=\hat{d}(r,y,\eta )+\hat{c}(r,y,\eta )
\end{equation}
for
\begin{equation} \label{32spdi0}
\begin{split}
\hat{d}(r,y,\eta )=\,&\omega (r)\int_{\Gamma_{(n+1)/2-\gamma } }r^{-z}f(y,z)\hat{b}(y,z,\eta )\dbar z\\&-\omega (r)\int_{\Gamma_{(n+1)/2-(\gamma +\beta)} }r^{-z}f(y,z)\hat{b}(y,z,\eta )\dbar z,
\end{split}
\end{equation}
and
\begin{equation} \label{32spdi5}
\hat{c}(r,y,\eta )=\omega (r)\int_{\Gamma_{(n+1)/2-(\gamma +\beta )} }r^{-z}f(y,z)\hat{b}(y,z,\eta )\dbar z.
\end{equation}
We have
\begin{equation} \label{32spdi6}
\hat{d}(r,y,\eta )=\omega (r)\int_C r^{-z}f(y,z)\hat{b}(y,z,\eta )\dbar z
\end{equation}
for a smooth compact curve $C$ counter-clockwise surrounding the poles of \newline $f(y,z)b(y,z)$ in $\{(n+1)/2-(\gamma +\beta )<\textup{Re}\,z<(n+1)/2-\gamma\}$ for all $y\in U.$ Moreover, we write
\begin{equation} \label{32spdi7}
\begin{split}
\hat{c}(r,y,\eta )=\omega (r)\int_{-\infty }^\infty & r^{-((n+1)/2-(\gamma +\beta )+i\rho )}f(y,(n+1)/2-(\gamma +\beta )+i\rho )\\ &\hat{b}(y,(n+1)/2-(\gamma +\beta )+i\rho,\eta )\dbar \rho \\
&=r^\beta \omega (r)\int_{\Gamma_{(n+1)/2-\gamma } } r^{-\tilde{z}}f(y,\tilde{z}-\beta  )\hat{b}(y,\tilde{z}-\beta ,\eta )\dbar \tilde{z}.
\end{split}
\end{equation}
It follows that $\omega (r)\op_M^{\gamma -n/2}(f)(y)\tilde{\omega }(r)\langle\zeta _w(y),r^{-w}\rangle \hat{v}(\eta )=\hat{d}(r,y,\eta )+\hat{c}(r,y,\eta )$ for all $y\in U,$ and \eqref{32spme} gives us
\begin{equation} \label{32spdi8}
\Op_y(m)u(r,y)=F^{-1}_{\eta \rightarrow y }\kappa _{[\eta ]}\{\big(\hat{d}(r,y,\eta )+\hat{c}(r,y,\eta )\big)\}.
\end{equation}
In order to recognise the nature of the latter expression we first note that the property
\begin{equation} \label{32spdi9}
F^{-1}_{\eta \rightarrow y }\kappa _{[\eta ]}\hat{c}(r,y,\eta )\in \mathcal{W}^s(\R^q,\mathcal{K}^{\infty ,\gamma +\beta }(X^\wedge))
\end{equation}
can be proved in an analogous manner as \eqref{32spdif6}. Next we characterise the contribution of $\hat{d}(r,y,\eta )$ in \eqref{32spdi8}.
The product $f(y,z)\hat{b}(y,z,\eta )=f(y,z)b(y,z)\hat{v}(\eta )$ occurring in \eqref{32spdi6} is holomorphic in $z$ in a neighbourhood of the curve $C$ for all $y\in U$ and extends for every $y$ to a meromorphic function inside, with values in $C^\infty (U,\mathcal{A}(\C\setminus K,E^s))$  for $E^s=C^\infty (X)\hat{\otimes}_\pi \hat{H}^s(\R^q_\eta ).$ The function \eqref{32spdi6} represents altogether the pairing of a function $\hat{\delta }(y,\eta )\in C^\infty (U,\mathcal{A}'(K,E^s))^\bullet$ with $r^{-z},$ cf. also the notation \eqref{21ofu}. This gives us a $\delta (y,y')$ as in Definition \ref{31wedgeas}. Thus, the first summand on the right of \eqref{32spdi8} takes the form
\begin{equation} \label{32spdof7}
F^{-1}_{\eta \rightarrow y }\kappa _{[\eta ]}(\omega (r)\langle\delta (y,y'),r^{-z}\rangle)=F^{-1}_{\eta \rightarrow y }\{[\eta ]^{(n+1)/2}\omega (r[\eta ])\langle\delta (y,y'),(r[\eta ])^{-z}\rangle\}.
\end{equation}
Next we turn to the remainder term in \eqref{32abcch} which is of the form (when we omit again $j$)
\begin{equation} \label{32rem1}
r_N(y,\eta )\!=\!(N+1)\!\!\!\!\!\sum_{|\alpha |=N+1}\!\int_0^1 \!\!\!(1-t)^N\!/\alpha !\!\!\int\!\!\!\int \!\!e^{-ix\xi }(\partial _\eta ^\alpha m)(y,\eta +t\xi )(D_y^\alpha n)(y+x,\eta )dx\dbar\xi dt.
\end{equation}
For $N=0$ we have
\begin{equation} \label{32rem2}
r_0(y,\eta )=\!\!\!\!\sum_{|\alpha |=1}\!\int_0^1 \!\!\int\!\!\!\int \!\!e^{-ix\xi }(\partial _\eta ^\alpha m)(y,\eta +t\xi )(D_y^\alpha n)(y+x,\eta )dx\dbar\xi dt.
\end{equation}
We may consider the summands separately. For convenience we assume $q=1$ (the general case is completely analogous). Then our remainder is the sum of the following two terms
\begin{equation} \label{32rem4}
\begin{split}
F^{-1}_{\eta \rightarrow y}\int_0^1 \!\int\!\!\!\int & e^{-ix\xi }\{ \omega(r[\eta +t\xi] ) \op_M^{\gamma -n/2}(f)(y)(\partial _\eta \omega_\eta  )(r[\eta +t\xi])\}\\&\{\omega (r[\eta ]) [\eta ]^{(n+1)/2}\langle(\partial _y \zeta )(y+x),(r[\eta ])^{-z}\rangle)\hat{v}(\eta )\}dx\dbar\xi dt.
\end{split}
\end{equation}
and
\begin{equation} \label{32rem3}
\begin{split}
F^{-1}_{\eta \rightarrow y}\int_0^1 \!\int\!\!\!\int & e^{-ix\xi }\{(\partial _\eta  \omega_\eta )(r[\eta +t\xi] ) \op_M^{\gamma -n/2}(f)(y)\omega (r[\eta +t\xi])\} \\&\{\omega (r[\eta ]) [\eta ]^{(n+1)/2}\langle(\partial _y \zeta )(y+x),(r[\eta ])^{-z}\rangle)\hat{v}(\eta )\}dx\dbar\xi dt
\end{split}
\end{equation}
According to Kumano-go's calculus \cite{Kuma1} on the structure of oscillatory integrals in compositions, here generalised to the set-up of operator-valued smbols of the kind \eqref{21symb}, the expression \eqref{32rem3} can be written as $\Op_y(a_0)v$ for a symbol $a_0(y,\eta )\in S^{-1}(\R^q\times \R^q;\C,\mathcal{K}^{\infty ,\infty }(X^\wedge)),$ continuously depending on the involved factors, in particular, on $\zeta $. This contributes to the convergence claimed after \eqref{32abcch}. For \eqref{32rem4} we write $\Op_y(a_1)v$ for
\begin{equation} \label{32rem5}
\begin{split}
a_1(y,\eta ):=\int_0^1 \!\int\!\!\!\int & e^{-ix\xi }\{ \omega(r[\eta +t\xi] ) \op_M^{\gamma -n/2}(f)(y)(\partial _\eta \omega _\eta )(r[\eta +t\xi])\}\\&\{\omega (r[\eta ]) [\eta ]^{(n+1)/2}\langle(\partial _y \zeta )(y+x),(r[\eta ])^{-z}\rangle\}dx\dbar\xi dt.
\end{split}
\end{equation}
In this case we have $a_1(y,\eta )\in S^0(\R^q\times \R^q;\C,\mathcal{K}^{\infty ,\gamma  }(X^\wedge)),$ again by a generalisation of Kumano-go's formalism. However, we want more, namely, variable discrete asymptotics in the image under $\Op_y(a_1)v$. To this end we decompose \eqref{32rem5} into $a_1(y,\eta )=\tilde{a}_0(y,\eta )+a(y,\eta )$ for
\begin{equation} \label{32rem6}
\begin{split}
\tilde{a}_0(y,\eta ):=\int_0^1 \!\int\!\!\!\int & e^{-ix\xi }\{ \big(\omega(r[\eta +t\xi] )-\omega(r[\eta ] )\big) \op_M^{\gamma -n/2}(f)(y)(\partial _\eta \omega _\eta )(r[\eta +t\xi])\}\\&\{\omega (r[\eta ]) [\eta ]^{(n+1)/2}\langle(\partial _y \zeta )(y+x),(r[\eta ])^{-z}\rangle)\}dx\dbar\xi dt,
\end{split}
\end{equation}
and
\begin{equation} \label{32rem7}
\begin{split}
a(y,\eta ):=\int_0^1 \!\int\!\!\!\int & e^{-ix\xi }\{ \omega(r[\eta ] ) \op_M^{\gamma -n/2}(f)(y)(\partial _\eta \omega _\eta )(r[\eta +t\xi])\}\\&\{\omega (r[\eta ]) [\eta ]^{(n+1)/2}\langle(\partial _y \zeta )(y+x),(r[\eta ])^{-z}\rangle)\}dx\dbar\xi dt.
\end{split}
\end{equation}
In Lemma \ref{32rem86}  below we will show that
\begin{equation} \label{32rem8}
\tilde{a}_0(y,\eta )\in S^{-1}(\R^q\times \R^q;\C,\mathcal{K}^{\infty ,\infty }(X^\wedge)).
\end{equation}
For $a(y,\eta )$ we write
\begin{equation} \label{32rem81}
a(y,\eta )= \omega(r[\eta ] ) \op_M^{\gamma -n/2}(f)(y)\omega(r[\eta ] )b(y,\eta )
\end{equation}
for
\begin{equation} \label{32rem9}
\begin{split}
b(y,\eta ):=&\int_0^1 \!\int\!\!\!\int  e^{-ix\xi }\{(\partial _\eta \omega _\eta )(r[\eta +t\xi])\}\\&\{\omega '(r[\eta ]) [\eta ]^{(n+1)/2}\langle(\partial _y \zeta )(y+x),(r[\eta ])^{-z}\rangle\}dx\dbar\xi dt.
\end{split}
\end{equation}
for any cut-off function $\omega '\succ \omega .$

For $v\in H^s(\R^q)$ we obtain
\begin{equation} \label{32rem82}
\Op_y(r_0)v=F^{-1}_{\eta \rightarrow y}\{a_0(y,\eta )+\tilde{a}_0(y,\eta )+a(y,\eta )\}\hat{v}(\eta )
\end{equation}
where
\begin{equation} \label{32rem83}
F^{-1}_{\eta \rightarrow y}\{a_0(y,\eta )\hat{v}(\eta )\},F^{-1}_{\eta \rightarrow y}\{\tilde{a}_0(y,\eta )\hat{v}(\eta )\}\in \mathcal{W}^s_{\textup{comp}}(U,\mathcal{K}^{\infty ,\infty }(X^\wedge)).
\end{equation}
What concerns the contribution of $a(y,\eta)$ Lemma \ref{32rem87} below will show that $b(y,\eta )\hat{v}(\eta)\in C_0^\infty (U_y,\hat{\mathcal{W}}^s(\R^q_\eta ,\mathcal{K}^{\infty ,\infty }(X^\wedge))).$ Then, to characterise 
\begin{equation} \label{32rem84}
F^{-1}_{\eta \rightarrow y}\{a(y,\eta )\}\hat{v}(\eta )=F^{-1}_{\eta \rightarrow y}\{\omega(r[\eta ] ) \op_M^{\gamma -n/2}(f)(y)\omega(r[\eta ] )b(y,\eta )\}\hat{v}(\eta )
\end{equation}
we are in a similar situation as in \eqref{32spme}, and it follows altogether that
\begin{equation} \label{32spdif85}
\Op_y(r_0)v(r,y)\in \mathcal{W}^s_{\textup{loc}}(\Omega ,\mathcal{K}^{\infty ,\gamma }_{\mathcal{L}}(X^\wedge)),
\end{equation}
cf. the notation in \eqref{32spdif8}. 
\end{proof}
Let $V$ be a Fr\'echet space with the semi-norm system $(\pi _j)_{j\in \N}$ which defines its Fr\'echet topology, and let $\boldsymbol{\nu }=(\nu _j)_{j\in \N}$ and $\boldsymbol{\mu }=(\mu _j)_{j\in \N}$ be sequences of reals. Then $S^{\boldsymbol{\mu };\boldsymbol{\nu }}(\R^q\times \R^q,V)$ is defined to be the space of all $a(x,\xi  )\in C^\infty (\R^q\times \R^q,V)$ such that 
\begin{equation} \label{32sy1}
\pi _j(D_x^\alpha D_\xi  ^\beta a(x,\xi  ))\leq c\langle \xi  \rangle^{\mu _j}\langle x\rangle^{\nu _j}
\end{equation}
for all $(x,\xi  )\in \R^q\times \R^q,\alpha ,\beta \in \N^q,j\in \N,$ for constants $c=c(\alpha ,\beta, j)>0.$ The space $S^{\boldsymbol{\mu },\boldsymbol{\nu }}(\R^q\times \R^q,V)$ is Fr\'echet with the optimal constants $c=c(\alpha ,\beta, j)(a)$ in the symbolic estimates \eqref{32sy1} as semi-norms. Set $S^{\boldsymbol{\infty  };\boldsymbol{\infty  }}(\R^q\times \R^q,V):=\bigcup_{\boldsymbol{\mu },\boldsymbol{\nu }}S^{\boldsymbol{\mu };\boldsymbol{\nu }}(\R^q\times \R^q,V).$ The machinery of oscillatory integrals in the sense of \cite{Kuma1} (here generalised to the vector-valued case) tells us that for any $\chi \in \mathcal{S}(\R^q\times \R^q),\chi\equiv 1$ near $(x,\xi )=0$ the limit
\begin{equation} \label{32osc1}
\int\!\!\!\int e^{-ix\xi }a(x,\xi )dx\dbar\xi :=\textup{lim}_{\varepsilon \rightarrow 0}\int\!\!\!\int e^{-ix\xi }\chi(\varepsilon x,\varepsilon \xi )a(x,\xi )dx\dbar\xi
\end{equation}
 exists and defines a continuous operator
\begin{equation} \label{32osc2}
S^{\boldsymbol{\mu },\boldsymbol{\nu }}(\R^q\times \R^q,V)\rightarrow V,
\end{equation}
independent of the choice of $\chi .$\\This construction can be applied, in particular, to the case  $V:=V^\mu :=S^\mu _{\textup{cl}}(\R^q_\eta ;\C,\newline\mathcal{K}^{\infty ,\infty }(X^\wedge))$ for some $\mu \in \R.$
According to Remark \ref{21syfr} the space $V^\mu $ is Fr\'echet with the semi-norm system
\begin{equation}\label{32rem8713}
\pi _{\tilde{\beta } ,N}(k):=\textup{sup}_{\eta \in \R^q}\langle\eta \rangle^{-\mu +|\tilde{\beta } |}\|\kappa _{\langle\eta \rangle}^{-1}D_\eta ^{\tilde{\beta } }k(\eta )\|_{\mathcal{K}^{N, N}(X^\wedge)}, \,N\in \N,\tilde{\beta } \in \N^q,
\end{equation}
together with the semi-norms of the components $k_{(\mu -j)}(\eta ),j\in \N,$ in the respective spaces of homogeneous functions $S^{(\mu -j)}(\R^q\setminus \{0\};\C,\mathcal{K}^{\infty ,\infty }(X^\wedge)),j\in \N,$ cf. the formulas \eqref{21hm} and \eqref{21hm1}, plus the semi-norms of the remainders $k(\eta )-\sum_{j=0}^M\chi (\eta )k_{(\mu -j)}(\eta ),M\in \N,$ in $S^{\mu -(M+1)}(\R^q_\eta ;\C,\mathcal{K}^{\infty ,\infty }(X^\wedge))$ for any excision function $\chi (\eta )$ (the latter semi-norms are as in \eqref{32rem8713} for $\mu -(M+1)$ rather than $\mu $).
\begin{lem}\label{32rem871}
Let $\sigma (r)\in C_0^\infty (\R_+),\rho (r)\in C_0^\infty (\overline{\R}_+),p(\eta )\in S_{\textup{cl}}^{\nu _1}(\R^q),s(\eta )\in S_{\textup{cl}}^{\nu _2}(\R^q)$ For every fixed $y$ and $t\in [0,1]$ the function 
\begin{equation}\label{32rem8711}
g_t:(x,\xi )\mapsto  \sigma (r[\eta +t\xi ])p(\eta +t\xi )s(\eta )[\eta ]^{(n+1)/2}\rho (r[\eta ])\omega (r[\eta ])\langle\zeta (x+y),(r[\eta ])^{-z}\rangle
\end{equation}
defines an element 
\begin{equation}\label{32rem87111}
g_t(x,\xi )\in S^{\boldsymbol{\rho  };\boldsymbol{\delta  }}(\R^q\times \R^q,V^\mu ),\,\mu :=\nu _1+\nu _2,
\end{equation}
for some sequences $\boldsymbol{\rho  },\boldsymbol{\delta  },$ and $t\mapsto g_t$ is continuous as a map $[0,1]\rightarrow V^\mu ;$ then $(g_t)_{t\in [0,1]}$ is a bounded set in the space $V^\mu .$
\end{lem}
\begin{proof}
For convenience we assume again $q=1;$ the general case is completely analogous. The expression on the right hand side of \eqref{32rem8711} in the variables $(r,\cdot)\in X^\wedge$ is an element of $\mathcal{K}^{\infty , \infty}(X^\wedge)$ for every fixed $x,\xi ,y,\eta ,t.$ In fact, we have $\omega (r[\eta ])\langle\zeta (x+y),(r[\eta])^{-z}\rangle \in \mathcal{K}^{\infty ,\gamma  }(X^\wedge)$ since the compact carriers of the involved analytic functionals are contained in $\{\textup{Re}\,z<(n+1)/2-\gamma \},$ and we have $\varphi \mathcal{K}^{\infty ,\gamma  }(X^\wedge)\subseteq \mathcal{K}^{\infty ,\infty  }(X^\wedge)$ for any $\varphi \in C_0^\infty (\R_+).$ In addition for any fixed $x,\xi ,y,t,$ the function $g_t(x,\xi )$ belongs to $V^\mu :=S^\mu _{\textup{cl}}(\R_\eta ;\C,\mathcal{K}^{\infty ,\infty }(X^\wedge)).$\\
We have to show that for any semi-norm $\pi$ from the Fr\'echet space topology of the space $V^\mu $ there are orders $\rho (\pi), \delta (\pi)\in \R $ such that
\begin{equation}\label{32reesti}
\pi \big(D_x^\alpha D_\xi ^\beta  g_t(x,\xi )\big)\leq  c\langle \xi  \rangle^{\rho (\pi)}\langle x\rangle^{\delta (\pi)}
\end{equation}
for all $(x,\xi )\in \R\times \R$ and $\alpha ,\beta  \in \N,$ for some constants $c=c(\alpha ,\beta  )>0.$ The semi-norms $\pi $ for the estimates \eqref{32reesti} are as in Remark \ref{21syfr}, for $H=\C,\kappa _\lambda =\textup{id},$ and $\tilde{H}=\mathcal{K}^{N, N}(X^\wedge),\tilde{\kappa }_\lambda u(r,\cdot)=\lambda ^{(n+1)/2}u(\lambda r,\cdot),$ for every $N\in \N,$ $n=\mathrm{dim}\,X.$ Let us discuss in this proof the semi-norms \eqref{32rem8713}; the others are easy as well and left to the reader. In other words we estimate the norms $\|\kappa _{\langle\eta \rangle}^{-1}D_\eta ^{\tilde{\beta } }\big(D_x^\alpha D_\xi ^\beta g_t(x,\xi )\big)\|_{\mathcal{K}^{N, N}(X^\wedge)}$ for all $(x,\xi )\in \R\times\R$ and $\alpha,\beta\in\N.$ In our case it suffices to replace  $\mathcal{K}^{N, N}(X^\wedge)$ by $\mathcal{H}^{N, N}(X^\wedge)$ since $g_t$ contains the cut-off factor $\omega (r[\eta ])$ which is of bounded support in $r,$ uniformly in $\eta \in \R.$ The $\mathcal{H}^{N, N}(X^\wedge)$-norm can be estimated (up to a constant) by finitely many expressions of the form
\begin{equation}\label{32kuest0}
\sum_{j,m}\big\|(r\partial _r)^j D^m\{\kappa _{\langle\eta \rangle}^{-1}D_\eta ^{\tilde{\beta } }\big(D_x^\alpha D_\xi ^\beta  g_t(x,\xi )\big)\}\big\|_{r^{N-n/2}L^2(\R_+\times X)}
\end{equation}
where $j+m\leq N,$ and $D^m$ is a polynomial in vector fields tangential to $X$ of degree $m.$ 
Let us write $g:=g_t, \sigma :=\sigma (r[\eta +t\xi ]),p:=p(\eta +t\xi ), l:=[\eta ]^{(n+1)/2},f:=\rho (r[\eta ])\omega (r[\eta ])\langle\zeta (x+y),(r[\eta ])^{-z}\rangle;$ then $g=\sigma pslf.$ Applying $\partial _\eta $ gives us
\begin{equation}\label{32kuest}
\partial _\eta g=(\partial _\eta\sigma) pslf+\sigma (\partial _\eta p)slf+\sigma p(\partial _\eta s)lf+\sigma ps(\partial _\eta l)f+\sigma psl(\partial _\eta f).
\end{equation}
Let us denote by $m=m(\eta )$ different elements of $S_{\textup{cl}}^{-1}(\R_\eta ),$  and set $[\eta ]':=\partial _\eta [\eta ]$ which belongs to $S_{\textup{cl}}^0(\R_\eta ).$ Then we have
\begin{equation}\label{32kuest1}
\partial _\eta\sigma (r[\eta +t\xi ])=(\partial _\eta\sigma)(r[\eta +t\xi ])[\eta +t\xi]'[\eta ] ^{-1}(r[\eta ]),
\end{equation}
\begin{equation}\label{32kuest2}
\partial _\eta p(\eta +t\xi )=p'(\eta +t\xi )\,\,\mbox{for a}\,\,p'(\eta  )\in S_{\textup{cl}}^{\nu _1-1}(\R_\eta ),
\end{equation}
\begin{equation}\label{32kuest3}
\partial _\eta s(\eta)\in S_{\textup{cl}}^{\nu _2-1}(\R_\eta ),\,\partial _\eta l(\eta)=m(\eta )l(\eta ),
\end{equation}
\begin{equation}\label{32kuest4}
\partial _\eta \rho (r[\eta ])=(\partial _r \rho )(r[\eta ])(r[\eta ])m(\eta ),\,\,\partial _\eta \omega  (r[\eta ])=(\partial _r \omega )(r[\eta ])(r[\eta ])m(\eta )\tilde{\omega }(r[\eta ])
\end{equation}
for any cut-off function $\tilde{\omega }\succ \omega ,$
\begin{equation}\label{32kuest5}
\partial _\eta \langle\zeta (x+y),(r[\eta ])^{-z}\rangle=(r[\eta ])m(\eta )\langle z\zeta (x+y),(r[\eta ])^{-z}\rangle.
\end{equation}
The relations \eqref{32kuest1}-\eqref{32kuest5} show that the summands of \eqref{32kuest} have the same structure as $g$ itself, with the only exception that now $\nu_1+ \nu_2 $ is diminished by $1.$ By iterating the conclusion it suffices to consider the case $\tilde{\beta } =0,$ since the original orders $\nu_1, \nu_2 $ are arbitrary. In other words, in \eqref{32kuest0} we may assume $\tilde{\beta } =0.$ Clearly it suffices to study the summands separately. The summand containing $(r\partial _r)^j$ has the form
\begin{equation}\label{32kuest6}
\Big\{\!\! \!\int\big|r^{-N}\!(r\partial _r)^jD^m\kappa _{\langle\eta \rangle}^{-1}D_x^\alpha D_\xi ^\beta  \big(\sigma (r[\eta +t\xi ])p(\eta +t\xi)s(\eta )l(\eta )f(r[\eta ],x+y)\big)\big|^2r^ndrdx\!\Big\}^{1/2}
\end{equation}
for $f(r[\eta ],x+y):=\rho (r[\eta ])\omega (r[\eta ])\langle\zeta (x+y),(r[\eta ])^{-z}\rangle.$ 
The differentiation in $x$ is entirely harmless since $\zeta $ is of compact support with respect to the spatial variables. So we content ourselves with $\alpha  =0.$ Moreover, the $D^m$-derivatives have no  influence to the growth of \eqref{32kuest6} in $\eta $ or $\xi $ since $D^m$ only acts on the values of the involved analytic functionals in $C^\infty (X),$ and the result after applying $D^m$ has the same quality as before. Therefore, it suffices to consider the case $m=0.$ Thus, taking into account that $\kappa _{\langle\eta \rangle}^{-1}$ compensates the factor $l(\eta )=[\eta ]^{(n+1)/2}$ the expression \eqref{32kuest6} is reduced to 
\begin{equation}\label{32kuest7}
\Big\{\!\! \int\big|r^{-N}\!(r\partial _r)^j\delta  _{\langle\eta \rangle}^{-1} D_\xi ^\beta  \big(\sigma (r[\eta +t\xi ])p(\eta +t\xi)s(\eta )f(r[\eta ],x+y)\big)\big|^2r^ndrdx\!\Big\}^{1/2}
\end{equation}
for $(\delta _\lambda u)(r,\cdot):=u(\lambda r,\cdot).$ We have $\partial _\xi \{\sigma (r[\eta +t\xi ])p(\eta +t\xi)\}=t \{(\partial _r\sigma )(r[\eta +t\xi ])r[\eta +t\xi ](\partial _\xi [\eta +t\xi ])[\eta +t\xi ]^{-1}p(\eta +t\xi)+\sigma (r[\eta +t\xi ])(\partial _\eta p)(\eta +t\xi)\}.$ By iterating $\xi $-derivatives it follows that $D_\xi ^\beta  \{\sigma (r[\eta +t\xi ])p(\eta +t\xi)\}$ is equal to $t^\beta  $ times a finite linear combination of products of the form $\tau (r[\eta +t\xi ])b(\eta +t\xi)$ for certain $\tau (r)\in C_0^\infty (\R_+)$ and $b(\eta )\in S_{\textup{cl}}^{\nu _1-\gamma }(\R_\eta ).$ Thus, \eqref{32kuest7} can be estimated by a finite linear combination of terms of the form 
\begin{equation}\label{32kuest8}
\Big\{\!\! \int\big|r^{-N}\!(r\partial _r)^j t^\beta  \big(\tau (r\langle\eta \rangle^{-1}[\eta +t\xi ])b(\eta +t\xi)s(\eta )f(r\langle\eta \rangle^{-1}[\eta ],x+y)\big)\big|^2r^ndrdx\!\Big\}^{1/2}.
\end{equation} 
For purposes below we pass to the equivalent expression
\begin{equation}\label{32kuest8a}
\Big\{\!\! \int\big|r^{-\tilde{N}}\!(r\partial _r)^j t^\beta  \big(\tau (r\langle\eta \rangle^{-1}[\eta +t\xi ])b(\eta +t\xi)s(\eta )f(r\langle\eta \rangle^{-1}[\eta ],x+y)\big)\big|^2r^{\tilde{n}}drdx\!\Big\}^{1/2}.
\end{equation}
for
\begin{equation}\label{32kuest8b}
\tilde{N}:= N-\gamma ,\,\,\tilde{n}:=n-2\gamma .
\end{equation}
Now we apply the $r\partial _r$-derivatives. Observe that for any real-valued $\Theta $ and a function $\tau _0\in C_0^\infty (\R_+)$ we have $r\partial _r\tau _0(r\Theta )=r\Theta (\partial _r\tau _0)(r\Theta )=:\tau _1(r\Theta )$ for some $\tau _1(r)\in C_0^\infty (\R_+).$ Thus, $(r\partial _r)^l\tau _0(r\Theta )=\tau _l(r\Theta ),l\in \N,$ for some $\tau _l(r)\in C_0^\infty (\R_+).$ We apply this to the function $\tau =:\tau _0$ under the integral \eqref{32kuest7} as well as to $\rho _0(r):=\rho (r)\omega (r)$ occurring in the definition of $f$ which gives us functions $\rho  _l(r)\in C_0^\infty (\R_+),l\in \N.$ It remains to look at what happens when we apply $(r\partial _r)^l$ to $\langle\zeta (x+y),(r[\eta ])^{-z}\rangle$ contained in $f;$ but this remains unchanged at all. Thus, we do not change the quality of the expression \eqref{32kuest8} when we set $j=0.$ Inserting $r^{-\tilde{N}}= (r\langle\eta \rangle^{-1}[\eta +t\xi ])^{-\tilde{N}}(\langle\eta \rangle^{-1}[\eta +t\xi ])^{\tilde{N}}$ and $\tilde{\tau }(r):=\tau  (r)r^{-\tilde{N}}\in C_0^\infty (\R_+)$ from \eqref{32kuest8} we obtain 
\begin{equation}\label{32kuest9}
\Big\{\!\! \int\big|\!(\langle\eta \rangle^{-1}[\eta +t\xi ])^{\tilde{N}} t^\beta  \tilde{\tau } (r\langle\eta \rangle^{-1}[\eta +t\xi ])b(\eta +t\xi)s(\eta )f(r\langle\eta \rangle^{-1}[\eta ],x+y)\big|^2r^{\tilde{n}}drdx\!\Big\}^{1/2}.
\end{equation}
We have the inequalities
\begin{equation}\label{32un1}
|\tilde{\tau } (r\langle\eta \rangle^{-1}[\eta +t\xi ])|\leq c\,\,\mbox{for all}\,\,r\in \R_+,0\leq t\leq 1,\eta ,\xi\in \R,
\end{equation}
\begin{equation}\label{32un2}
|b(\eta +t\xi)|\leq c\langle\eta +t\xi\rangle^{\nu _1}\,\,\mbox{for all}\,\,0\leq t\leq 1,\eta ,\xi\in \R, 
\end{equation}
and
\begin{equation}\label{32un3}
|s(\eta )|\leq c\langle\eta \rangle^{\nu _2}\,\,\mbox{for all}\,\,\eta\in \R, 
\end{equation}
for different constants $c>0.$ Since the carriers of the analytic functionals involved in $f$ are contained in $\{\textup{Re}\,z<(n+1)/2-\gamma \},$ we have $\{ \int|f(r\langle\eta \rangle^{-1}[\eta ],x+y)|^2r^{n-2\gamma }drdx\}^{1/2}<\infty $ uniformly in $\eta .$ Thus, we can estimate \eqref{32kuest9} by
\begin{equation}\label{32un4}
c(\langle\eta \rangle^{-1}[\eta +t\xi ])^{N-\gamma  }t^\beta \langle\eta +t\xi\rangle^{\nu _1}\langle\eta \rangle^{\nu _2}.
\end{equation}
From Peetre's inequality we obtain $[\eta +t\xi ]^{N-\gamma  }\leq c\langle\eta \rangle^{N-\gamma  }\langle t\xi  \rangle^{|N-\gamma  |},\langle\eta +t\xi \rangle^{\nu _1  }\leq c\langle\eta \rangle^{\nu _1  }\langle t\xi \rangle^{|\nu _1  |}.$ It follows that \eqref{32kuest9} can be estimated by $ct^\beta \langle t\xi \rangle^{|N-\gamma |+|\nu _1  |}\langle \eta  \rangle^{\nu _1+\nu _2}.$ Comparing that with \eqref{32rem8713} and using $\mu =\nu _1+\nu _2$ we just obtain \eqref{32sy1} for $\mu _j=|N-\gamma |+|\nu _1  |$ and $\nu_j=0.$ Applying the oscillatory integral process \eqref{32osc2} we obtain the relation \eqref{32rem87111} for every $t\in [0,1].$ The boundedness of the set $(g_t)_{t\in [0,1]}$ in $V^\mu $ is a direct consequence of the estimates.The way to obtain the continuity of $t\mapsto V^\mu $ is evident after the above considerations. 
\end{proof}

\begin{lem}\label{32rem86}
We have $\tilde{a}_0(y,\eta )\in S^{-1}_{\textup{cl}}(\R^q\times \R^q;\C,\mathcal{K}^{\infty ,\infty }(X^\wedge)).$
\end{lem}
\begin{proof}
According to what is done in the operator-valued analogue of Kumano-go's oscillatory integral techniques it suffices to show that for every fixed $y$ and $t\in [0,1]$ the function 
\begin{equation}\label{32lem1}
\begin{split}
h_t:(x,\xi )\mapsto   &\big(\omega(r[\eta +t\xi] )-\omega(r[\eta ] )\big)\op_M^{\gamma -n/2}(f)(y)(\partial _\eta \omega _\eta )(r[\eta +t\xi])\\&\{\omega (r[\eta ]) [\eta ]^{(n+1)/2}\langle(\partial _y \zeta )(y+x),(r[\eta ])^{-z}\rangle\},
\end{split}
\end{equation}
contained in \eqref{32rem6}, belongs to the space $S^{\boldsymbol{\rho  };\boldsymbol{\delta  }}(\R^q\times \R^q,V^{-1})$ for $V^{-1}:=S^{-1} _{\textup{cl}}(\R^q_\eta ;\C,\newline\mathcal{K}^{\infty ,\infty }(X^\wedge)),$
for some sequences $\boldsymbol{\rho  },\boldsymbol{\delta  },$ and that $t\mapsto h_t$ is continuous as a map $[0,1]\rightarrow V^{-1};$ then $(g_t)_{t\in [0,1]}$ is a bounded set in the space $V^{-1}.$
 Similarly as in the proof of Lemma \ref{32rem871} we first observe that for every fixed $x,\xi ,y,\eta ,t$ the right hand side of \eqref{32lem1} belongs to $\mathcal{K}^{\infty ,\infty  }(X^\wedge).$ More precisely we show that for every fixed $x,\xi ,y,t$ the function $h_t$ takes values in $V^{-1}.$ For convenience we assume again $q=1.$ We have $p(\eta ):=\partial _\eta [\eta ]\in S^0 _{\textup{cl}}(\R^q_\eta) $ which gives us $\partial _\eta \omega _\eta =\sigma (r[\eta +t\xi] )rp(\eta +t\xi)$ for $\sigma (r):=(\partial _r\omega )(r)\in C_0^\infty (\R_+).$ Choose cut-off functions $\omega _4\succ \omega _3\succ \omega _2\succ \omega _1\succ \omega $ Then the right hand side of \eqref{32lem1} can be written as
\begin{equation}
h_t(x,\xi )=n_t(\xi )\omega _3(r[\eta ])\op_M^{\gamma -n/2}(f)(y)\omega _2(r[\eta ])g_t(x,\xi )
\end{equation}
for
\begin{equation}
n_t(\xi )=\big (\omega (r[\eta +t\xi ])-\omega (r[\eta  ])\big)\omega _4(r[\eta ]), 
\end{equation}
and
\begin{equation}
g_t(x,\xi )=\sigma (r[\eta +t\xi ])p(\eta +t\xi )[\eta ]^{-1}\rho (r[\eta ])\omega (r[\eta ])[\eta ]^{(n+1)/2}\langle(\partial _y\zeta )(y+x),(r[\eta ])^{-z}\rangle,
\end{equation}
for $\rho (r[\eta ]):=r[\eta ]\omega _1(r[\eta ]).$\\ Because of Lemma \ref{21clsymb} it is clear that $\omega _3(r[\eta ])\op_M^{\gamma -n/2}(f)(y)\omega _2(r[\eta ])$ belongs to $S^0_{\textup{cl}}(\R^q_\eta ;\mathcal{K}^{\infty ,\infty  }(X^\wedge),\mathcal{K}^{\infty ,\gamma  }(X^\wedge))$ (recall that $y$ is kept fixed for the moment, but there is smoothness and compact support with respect to $y$). Moreover, similarly as in Lemma \ref{32rem871} the function $g_t(x,\xi )$ takes values in $S^{-1} _{\textup{cl}}(\R^q_\eta ;\C,\mathcal{K}^{\infty ,\infty }(X^\wedge)).$  Then  $b_t(x,\xi )$ also takes values in $V^{-1}=S^{-1} _{\textup{cl}}\!(\R^q_\eta ;\C,\mathcal{K}^{\infty ,\infty }(X^\wedge))$ when $n_t(\xi )\!\in \!S^0_{\textup{cl}}\!(\R^q_\eta ;\mathcal{K}^{\infty ,\gamma  }(X^\wedge),\mathcal{K}^{\infty ,\infty }(X^\wedge))$ for fixed $\xi .$ The latter property will be the final point of the proof, since the symbolic estimates for $b_t$ in $x,\xi $ are again straighforward.\\
Let us set $\psi (\eta +t\xi ):=\omega (r[\eta +t\xi ]),$ and apply the Taylor formula
\begin{equation}
\psi (\eta +t\xi )=\sum_{|\alpha |\leq N}(t\xi )^\alpha /\alpha !(\partial _\eta ^\alpha \psi )(\eta )+r_N(\eta ,\xi ,t),
\end{equation}
$r_N(\eta ,\xi ,t)=(N+1)\sum_{|\alpha |= N+1}(t\xi )^\alpha /\alpha !\int_0^1(1-\vartheta )^N(\partial _\eta ^\alpha \psi )(\eta +\vartheta t\xi )d\vartheta .$ We need this only for $N=0.$ Recall that in this proof we consider the case $q=1.$ Then 
\begin{equation}
n_t(\xi )=\int_0^1(\partial _\eta  \psi )(\eta +\vartheta t\xi )d\vartheta \omega _4(r[\eta ])=\int_0^1(\partial _r\omega )(r[\eta +\vartheta t\xi ])p_0(\eta +\vartheta t\xi)d\vartheta r\omega _4(r[\eta ])
\end{equation}
for a $p_0(\eta)\in S^0_{\textup{cl}}(\R^q).$ Now an elementary computation shows that $n_t(\xi )\!\in \!S^0_{\textup{cl}}\!(\R^q_\eta ;\newline\mathcal{K}^{\infty ,\gamma  }(X^\wedge),\mathcal{K}^{\infty ,\infty }(X^\wedge)).$ 
\end{proof}
\begin{lem}\label{32rem87}
We have $b(y,\eta )\in  S^0_{\textup{cl}}(\R^q\times \R^q;\C,\mathcal{K}^{\infty ,\infty }(X^\wedge)).$ 
\end{lem}
\begin{proof}
We can form
\begin{equation}\label{32rem875}
g_t:(x,\xi )\mapsto(\partial _\eta \omega _\eta )(r[\eta +t\xi])\omega '(r[\eta ]) [\eta ]^{(n+1)/2}\langle(\partial _y \zeta )(y+x),(r[\eta ])^{-z}\rangle
\end{equation}
for $t\in [0,1]$ and then argue in a similar manner as in the proof of Lemma \ref{32rem86}, i.e. apply an oscillatory integral argument to the $S^0_{\textup{cl}}(\R^q\times \R^q;\C,\mathcal{K}^{\infty ,\infty }(X^\wedge))$-valued function \eqref{32rem875} occurring in the expression $b(y,\eta )=\int_0^1\int\!\!\int e^{-ix\xi }g_t(x,\xi )dx\dbar\xi dt,$ cf. the formula \eqref{32rem9}.
\end{proof}



\begin{thebibliography}{1}




\bibitem{Benn2} J. Bennish,
\textit{Asymptotics for elliptic boundary value problems for systems of pseudo-differential equations},
J. Math. Anal. and Appl. \textbf{179} (1993), 417-445.

\bibitem{Benn1} J. Bennish, \textit{Variable discrete asymptotics of solutions to elliptic
boundary-value problems}, Math. Res. \textbf{100}, ``Differential Equations, Asymptotic Analysis, and Mathematical Physics'', Akademie Verlag Berlin, 1997, pp. 26-31.

\bibitem{Dine4} N. Dines, \textit{Elliptic operators on corner manifolds}, Ph.D. thesis, University of Potsdam, 2006.

\bibitem{Eski2} G.I. Eskin, \textit{Boundary value problems for elliptic pseudodifferential equations}, Transl. of Nauka, Moskva, 1973, Math. Monographs, Amer. Math. Soc. \textbf{52}, Providence, Rhode Island 1980.

\bibitem{Gohb3} I.C. Gohberg and E.I. Sigal, \textit{An operator generalization of the logarithmic residue theorem and the theorem of {R}ouch{\'e}}, Math. USSR Sbornik \textbf{13}, 4 (1971), 603-625.

\bibitem{Haru13} G. Harutjunjan and B.-W. Schulze, \textit{Elliptic mixed, transmission and singular crack problems}, European Mathematical Soc., Z\" urich, 2008.

\bibitem{Horm5} L. H{\"o}rmander, \textit{The analysis of linear partial differential operators}, vol. 1 and 2, Springer-Verlag, New York, 1983.

\bibitem{Kapa10} D. Kapanadze  and B.-W. Schulze, \textit{Crack theory and edge singularities}, Kluwer Academic Publ., Dordrecht, 2003.

\bibitem{Kond1} V.A. Kondratyev, \textit{Boundary value problems for elliptic equations in domains with conical points}, Trudy Mosk. Mat. Obshch. \textbf{16}, (1967), 209-292.

\bibitem{Krai2} T. Krainer, \textit{Parabolic pseudodifferential operators and long-time asymptotics of solutions}, Ph.D. thesis, University of Potsdam, 2000.

\bibitem{Kuma1} H. Kumano-go, \textit{Pseudo-differential operators}, The MIT Press, Cambridge, Massachusetts and London, England, 1981.

\bibitem{Liu3} X. Liu and B.-W. Schulze, \textit{Boundary value problems in edge representation}, Math. Nachr. \textbf{280}, 5-6 (2007), 1-41.

\bibitem{Remp1} S. Rempel and B.-W. Schulze, \textit{Parametrices and boundary symbolic calculus for elliptic boundary problems without transmission property}, Math. Nachr. \textbf{105}, (1982), 45-149.

\bibitem{Remp6}  S. Rempel and B.-W. Schulze, \textit{Branching of asymptotics for elliptic operators on manifolds with edges}, Proc. ``Partial Differential Equations'', Banach Center Publ. \textbf{19}, PWN Polish Scientific Publisher, Warsaw, 1984.

\bibitem{Schu28} B.-W. Schulze, \textit{Regularity with continuous and branching asymptotics for elliptic operators on manifolds with edges}, Integral Equations Operator Theory \textbf{11} (1988), 557-602.

\bibitem{Schu32} B.-W. Schulze, \textit{Pseudo-differential operators on manifolds with edges}, Teubner-Texte zur Mathematik, Symp. ``Partial Differential Equations'', Holzhau 1988, \textbf{112}, Leipzig, 1989, pp. 259-287.

\bibitem{Schu2} B.-W. Schulze, \textit{Pseudo-differential operators on manifolds with singularities}, North-Holland, Amsterdam, 1991.

\bibitem{Schu34} B.-W. Schulze, \textit{The variable discrete asymptotics in pseudo-differential boundary value problems {I}}, Adv. in Partial Differential Equations  ``Pseudo-Differential Calculus and Mathematical Physics'', Akademie Verlag, Berlin, 1994, pp. 9-96.

\bibitem{Schu36} B.-W. Schulze, \textit{The variable discrete asymptotics in pseudo-differential boundary value problems {II}}, Adv. in Partial Differential Equations ``Boundary Value Problems, Schr\"odinger Operators, Deformation Quantization'', Akademie Verlag, Berlin, 1995, pp. 9-69.

\bibitem{Schu20} B.-W. Schulze, \textit{Boundary value problems and singular pseudo-differential operators}, J. Wiley, Chichester, 1998.

\bibitem{Schu25}  B.-W. Schulze, \textit{Operator algebras with symbol hierarchies on manifolds with singularities}, Oper. Theory: Adv. Appl. \textbf{125}, Adv. in Partial Differential Equations ``Approaches to Singular Analysis'' (J. Gil, D. Grieser, and M. Lesch, eds.), Birkh\" auser Verlag, Basel, 2001, pp. 167-207.

\bibitem{Schu55} B.-W. Schulze and A. Volpato, \textit{Green operators in the edge calculus}, Rend. Sem. Mat. Univ. Pol. Torino \textbf{66}, 1 (2008), 29-58.

\bibitem{Schu61} B.-W. Schulze and A. Volpato, \textit{Cone and edge calculus with discrete asymptotics}, arXiv  \textbf{0911.3813}, 19 Nov. 2009.

\bibitem{Schu62} B.-W. Schulze and A. Volpato, \textit{Continuous and variable discrete asymptotics}, Preprint  \textbf{2009/04}, University of Potsdam, Institut f\"ur Mathematik, Potsdam, 2009.

\bibitem{Schu69} B.-W. Schulze and A. Volpato, \textit{Edge operators and ellipticity with branching asymptotics}, (in preparation).

\bibitem{Seil2} J. Seiler, \textit{The cone algebra and a kernel characterization of {Green} operators}, Oper. Theory Adv. Appl. \textbf{125} Adv. in Partial Differential Equations ``Approaches
to Singular Analysis'' (J. Gil, D. Grieser, and M. Lesch, eds.), Birkh\"auser, Basel, 2001, pp. 1-29.

\bibitem{Witt3} I. Witt, \textit{On the factorization of meromorphic Mellin symbols}, Advances in Partial Differential Equations (Approaches to Singular Analysis) (S. Albeverio, M. Demuth, E. Schrohe,  and B.-W. Schulze, eds.), Oper. Theory: Adv. Appl., Birkh\" auser Verlag, Basel, 2002, pp. 279-306.

\end{thebibliography}
\end{document}